\documentclass{article}
\usepackage{authblk}

\title{Mutation of signed valued quivers and\\presentations of simple complex Lie algebras} 

\author[$\dagger$]{Joseph Grant}
\author[$\star$]{Davide Morigi}

\affil[$\dagger$]{School of Mathematics, University of East Anglia, Norwich NR4 7TJ, UK}
\affil[$\star$]{Department of Mathematics, Aarhus University, Ny Munkegade 118, DK-8000 Aarhus C, Denmark}

\date{\vspace{-5ex}}

\usepackage[UKenglish]{isodate}
\cleanlookdateon
\usepackage{enumitem}
\setlist[itemize]{
 topsep=0pt}
\setlist[enumerate]{
 topsep=0pt}

\usepackage[nottoc,notlot,notlof]{tocbibind}
\usepackage{latexsym}
\usepackage{verbatim}
\usepackage{amsmath}
\usepackage{amssymb}
\usepackage{mathrsfs}
\usepackage{amsthm}
\usepackage{sfmath}

\usepackage{stmaryrd} 
\usepackage{graphicx} 

\usepackage{hyperref}  

\usepackage[british]{babel}

\usepackage{pgfplots}

\tikzset{
    expand bubble/.style={
        preaction={draw,line width=10.4pt},
        white,fill,draw,line width=10pt,
    },
}

\usepackage{tikz}
\usetikzlibrary{calc,decorations.markings,%
arrows,backgrounds}
\tikzset{->-/.style={decoration={
  markings,
  mark=at position .5 with {\arrow{>}}},postaction={decorate}}}
  \usepackage{tikz-cd}

\input xy
\usepackage[all,cmtip,2cell]{xy}
\UseTwocells

\newcommand{\arr}[1]{\stackrel{#1}{\to}}
\newcommand{\rra}[1]{\stackrel{#1}{\leftarrow}}
\newcommand{\arrr}[1]{\stackrel{#1}{\longrightarrow}}

\newcommand{\Z}{\mathbb{Z}}

\newcommand{\R}{\mathbb{R}}

\newcommand{\st}{\;\left|\right.\;}
\newcommand{\abs}[1]{\left|#1\right|}

\newcommand{\GL}{\operatorname{GL}\nolimits}
\newcommand{\e}{\varepsilon}
\newcommand{\sgn}{\operatorname{sgn}}
\newcommand{\Dih}{\operatorname{Dih}}
\newcommand{\ad}{\operatorname{ad}}
\newcommand{\g}{\mathfrak{g}}
\newcommand{\h}{\mathfrak{h}}
\newcommand{\diag}{\operatorname{diag}\nolimits}
\newcommand\scalemath[2]{\scalebox{#1}{\mbox{\ensuremath{\displaystyle #2}}}}

\newtheorem{theorem}{Theorem}[section]
\newtheorem{corollary}[theorem]{Corollary}
\newtheorem{lemma}[theorem]{Lemma}
\newtheorem{proposition}[theorem]{Proposition}

\newtheorem*{thma}{Theorem A}
\newtheorem*{thmb}{Theorem B}

\theoremstyle{definition}
\newtheorem{definition}[theorem]{Definition}
\newtheorem{remark}[theorem]{Remark}
\newtheorem{example}[theorem]{Example}

\begin{document}

\maketitle

\begin{abstract}
We introduce a signed variant of (valued) quivers and a mutation rule that generalizes the classical Fomin-Zelevinsky mutation of quivers.  To any signed valued quiver we associate a matrix that is a signed analogue of the Cartan counterpart appearing in the theory of cluster algebras.  From this matrix, we construct a Lie algebra via a ``Serre-like'' presentation.

In the mutation Dynkin case, we define root systems using the signed Cartan counterpart and show compatibility with mutation of roots as defined by Parsons.  Using results from Barot-Rivera and Perez-Rivera, we show that mutation equivalent signed quivers yield isomorphic Lie algebras, giving presentations of simple complex Lie algebras.


\end{abstract}

\tableofcontents

\setlength{\parindent}{0pt} 
\setlength{\parskip}{1em plus 0.5ex minus 0.2ex}


\section{Introduction}
The theory of cluster algebras has been one of the mathematical success stories of the 21st century.  From early in the theory's development, it was clear that mutation gave new ways to understand finite type phenomena in Lie theory when the Dynkin classification arose in classifying mutation classes of finite type cluster algebras \cite{fz2}.  This led to generalizations of Cartan matrices \cite{bgz} and simple roots \cite{par}, and then to new presentations of finite crystallographic reflection groups \cite{bm} and braid groups \cite{gm}.  As we get new presentations of objects in Lie theory, it is natural to ask whether we get new presentations of the original object of the Dynkin classification: simple complex Lie algebras.

Independently of the theory of cluster algebras, presentations of Lie algebras associated to various unit forms were studied \cite{bkl} and it was realised that as well as the relations in the classical Serre presentation, one should add further relations according to cycles appearing in a graphical representation of generalized Cartan matrices.  This theory was refined, and extended to the setting of quasi-Cartan matrices appearing in the cluster algebras literature \cite{bgz}, in both the simply-laced \cite{br} and skew-symmetrizable \cite{pr} cases.  In these studies, graphs with two kinds of edges play a central role.  They have been denoted by solid and broken (dotted) edges, but we will represent them by edges decorated with $-$ and $+$ signs, respectively, to match the signs appearing in the generalized Cartan matrix.

Cluster algebras start with a skew-symmetrizable matrix, and the mutation rule is a case-by-case operation which changes this matrix.  In the simpler skew-symmetric case, this can be represented graphically by a quiver where orientation corresponds to sign, and the mutation rule can be described by a three step process.  The graphical description of the general skew-symmetrizable matrices has long been known: one should use valued quivers, which have two positive integers attached to each arrow.

In this article we introduce signed variants of quivers and valued quivers, where each arrow has a sign, together with a mutation rule of these objects.  As in \cite{fz1}, our theory starts with matrices, but the signs are already used to encode orientation.  So, in order to keep track of signs, we introduce a new variable $t$ which squares to $1$.  We give a mutation rule for these generalized skew-symmetric (gss) matrices and show that everything can be encoded graphically via signed valued quivers in the mutation Dynkin case.  We use earlier results on mutation \cite{bgz}, and recover various classical phenomena on specialization to the classical case.

Given a signed valued quiver $(Q,v)$ we define a Cartan counterpart $C(Q,v)$.  As with traditional Cartan matrices, this has diagonal entries $2$.  It is modelled on a definition from cluster algebras \cite{fz1,fz2}, where the off-diagonal entries are always non-positive, but in our definition we get both positive and negative off-diagonal entries.  We define a Lie algebra $\g(Q,v)=\g(C(Q,v))$ depending on this Cartan counterpart, and using previous work of Barot-Rivera and P\'erez-Rivera \cite{br,pr} we show the following result.
\begin{thma}[{Theorem \ref{thm:isoLie}}]
If $(Q,v)$ is a signed valued quiver of mutation Dynkin type and $(Q',v')=\mu_k(Q,v)$ then we have an isomorphism of Lie algebras
\[\varphi_k:\g(Q',v')\arr\sim \g(Q,v).\]
\end{thma}
As a corollary (Corollary \ref{cor:presDynk}) we obtain that each signed valued quiver $(Q,v)$ of mutation Dynkin type $\Delta$ gives a presentation of the simple complex Lie algebra $\g(\Delta)$.

One of the key tools in the study of semisimple Lie algebras is the theory of root systems.  These traditionally make use of simple roots, which form a basis of the vector space spanned by all the roots.  An alternative basis known as a companion basis, which is adapted to a particular quiver in the mutation class of our Dynkin diagram, was proposed by Parsons in the simply-laced case \cite{par}, and later extended to the skew-symmetrizable case \cite{bm}.  We introduce a signed variant of companion bases, and show that our refined variant is preserved by Parsons's mutation rule.

Via iterated mutation of roots, the formal simple basis of our root system is mapped to a signed companion basis in the original Dynkin root system: see Theorem \ref{thm:signedCB}.  Therefore, our theory of signed quiver mutation can be seen as a combinatorial precursor of both of the above theories: presentations of Lie algebras, and companion bases.
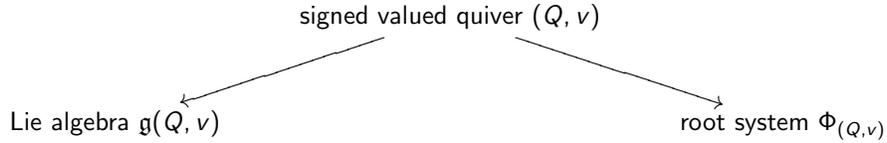
\[
\xymatrix{
&\text{signed valued quiver $(Q,v)$} \ar[dl]\ar[dr] & \\
\text{Lie algebra $\g(Q,v)$} && \text{root system $\Phi_{(Q,v)}$}
}
\]
If we write $\g=\g(Q,v)$, $\g'=\g(Q',v')$, $\Phi=\Phi_{(Q,v)}$, and $\Phi'=\Phi_{(Q',v')}$, then we have mutations of all three of the above objects, as illustrated in the following diagram:
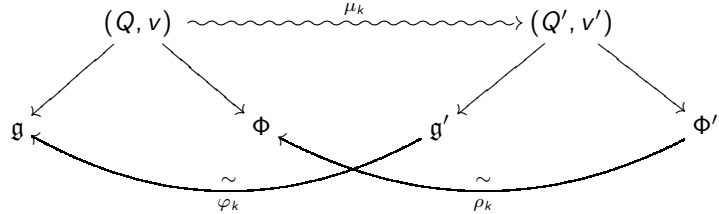
\[
\xymatrix{
&(Q,v) \ar[ld]\ar[dr] \ar@{~>}[rrrr]^{\mu_k} &&&& (Q',v')\ar[ld]\ar[dr]  &\\
\g&& \Phi &&\g'\ar@/^2pc/[llll]^{\varphi_k}_\sim && \Phi'\ar@/^2pc/[llll]^{\rho_k}_\sim\\
&
}
\]
This diagram suggests there might be some compatibility between mutation of the Lie algebras and the root systems.  We give a compatibility result for root space decompositions.
Let $\h,\h'$ denote the Cartan subalgebras of $\g,\g'$, respectively, and let $\g^\alpha$ denote the root space of $\alpha\in\Phi$, where all $h\in\h$ act by $\alpha(h)$.
\begin{thmb}[{Theorem \ref{thm:rootspacedecomp}}]
$\varphi_k(\g'^{\beta'})=\g^{\rho_k(\beta')}.$
\end{thmb}

\textbf{Notation:} 
Fix $n\geq1$.  All matrices are $n\times n$, and $i,j,k$ are integers between $1$ and $n$.  We mainly work in our signed setting, but use a tilde to identify objects coming from the classical setting.  So we use $\tilde B$ to denote a skew-symmetric (ss) matrix, and $B$ to denote a generalized skew-symmetric (gss) matrix.

\textbf{Acknowledgements:} 
Thanks to Bethany Marsh for pointers to the literature and to the anonymous referee for a careful reading of the paper.

This article was completed while J. G. visited Aarhus University.  J. G. thanks The Danish National Research Foundation (grant no. DNRF156) and The Independent Research Fund Denmark (grant no. 1026-00050B) for their support, and thanks Peter J\o rgensen and the Aarhus Homological Algebra Group for their hospitality.
  
D. M. was supported by a research grant (VIL42076) from VILLUM FONDEN.

%
%
%
%
%


\section{Mutation of signed valued quivers} \label{s:quivers}


\subsection{Generalized skew-symmetrizable matrices}

An $n\times n$ matrix $\tilde B=(\tilde b_{ij})$ with integer entries is called skew-symmetrizable if there exist positive integers $d_i$ such that $d_i \tilde b_{ij}=-d_j \tilde b_{ji}$.  The diagonal matrix $D=\diag(d_1,\ldots,d_n)$ is called the skew-symmetrizing matrix, and $D\tilde B=-(D\tilde B)^T$.

\begin{definition}\label{defn_spec_matr}
A \textbf{generalized skew-symmetrizable (gss) matrix} is an $n\times n$ matrix $B=(b_{ij})$ with entries $b_{ij}\in \Z[t]/(t^2-1)$ which is skew-symmetrizable, i.e., there exist $d_1,\ldots, d_n\in \Z_{>0}$ such that
$$
d_i b_{ij}=-d_j b_{ji}
$$
for all $i,j=1,\ldots,n$.
If $B$ is a gss matrix, its \textbf{specialization} $B(1)$ is given by evaluating $t=1$. \end{definition}

\begin{remark}
If $B$ is a gss
matrix then its specialization $B(1)$ is also skew-symmetrizable, with the same skew-symmetrizing matrix $D=\diag(d_1,\ldots,d_n)$ as $B$.
\end{remark}

Recall that for a nonzero real number $a$, its \emph{sign} is $\sgn(a)=a/\abs{a}$, and $\sgn(0)=0$.  We have $\sgn(ab)=\sgn(a)\sgn(b)$, and for a skew-symmetrizable matrix $\tilde B$ we have $\sgn(\tilde b_{ij})=-\sgn(\tilde b_{ji})$.
For every $k\in\{1,\ldots, n\}$, Fomin and Zelevinsky's mutation rule for integer matrices $\tilde B'=\mu_k^{FZ}(\tilde B)=(\tilde{b}_{ij}')$ \cite[Definition 4.2]{fz1} can be written as \cite[(1.1)]{bgz}:
\[
\tilde b_{ij}'=\begin{cases}
    -\tilde b_{ij} & \text{if } i=k \text{ or } j=k; \\
    \tilde b_{ij}+\sgn(\tilde b_{ik})\tilde b_{ik}\tilde b_{kj} & \text{if } \sgn(\tilde b_{ik}\tilde b_{kj})=1;\\
    \tilde b_{ij} & \text{otherwise}.\\
\end{cases} 
\]
We give a mutation rule for generalized matrices.
\begin{definition}\label{defn_sign}
For $a+bt\in \Z[t]/(t^2-1)$, define its \textbf{sign} as
\[ \sgn(a+bt)=\sgn(a+b).\]
\end{definition}
The important property of this definition is that if $x=a\in \Z$ then $\sgn(x)=\sgn(a)$, and if $x=tb\in t\Z$ then $\sgn(x)=\sgn(b)$.  It is also useful to note that, for $x,y\in \Z[t]/(t^2-1)$, we have $\sgn(tx)=\sgn(x)$ and $\sgn(x)\sgn(y)=\sgn(xy)$.
\begin{definition}\label{defn_signed_mut}
Let $B=(b_{ij})$ be a gss matrix and let $k\in\{1,\ldots, n\}$. 
Define the \textbf{signed mutation} $B'=\mu_k(B)=(b_{ij}')$ of $B$  as follows:
\[
b_{ij}'=\begin{cases}
    -b_{ij} & \text{if } i=k \text{ and }\sgn(b_{ij})=1, \text{ or } j=k \text{ and } \sgn(b_{ij})=-1; \\
    -tb_{ij} & \text{if } i=k \text{ and }\sgn(b_{ij})=-1,  \text{ or }  j=k \text{ and } \sgn(b_{ij})=1; \\
    t(b_{ij}+\sgn(b_{ik})b_{ik}b_{kj}) & \text{if }\sgn (b_{ik}b_{kj})=1; \\
    b_{ij} & \text{otherwise}.
\end{cases} 
\]
\end{definition}

\begin{example}\label{eg_mut_spec_matrix}
Consider the gss matrix:
$$
B=\left(\begin{matrix}
0 & -t & 0 \\
t & 0 & -2t \\
0 & t & 0
\end{matrix}\right),
$$
with skew-symmetrizing matrix diag$(1,1,2)$. Then its signed mutation at vertex $k=2$ is the following:
$$
\mu_2(B)=\left(\begin{matrix}
0 & t & -2t \\
-t & 0 & 2 \\
t & -1 & 0
\end{matrix}\right).
$$
\end{example}

The next result says that our mutation rule specializes to the classical Fomin-Zelevinsky mutation of skew-symmetrizable matrices.
\begin{proposition}\label{prop_mutat_agree}
Let $k\in\{1,\ldots, n\}$. Then $\mu_k^{\text{FZ}}(B(1))=(\mu_kB)(1)$, i.e., the following diagram commutes:
$$
\xymatrix @C=40pt {B \ar[r]^{\mu_k}\ar[d] 
 & \mu_k(B)\ar[d] \\
B(1)\ar[r]_{\mu_k^{\text{FZ}}} & (\mu_kB)(1)}
$$
\end{proposition}
\begin{proof}
Note that $\sgn(b_{ij}(1))=\sgn(b_{ij})$ and compare 
the definitions of $\mu_k$ and $\mu_k^{\text{FZ}}$.
\end{proof}

\begin{lemma}\label{lem_sign}
Let $B=(b_{ij})$ be a gss matrix.  Then:
\begin{enumerate}[label=(\roman*)]
    \item $\sgn(b_{ij})=-\sgn(b_{ji})$,
    \item $\sgn(b_{ij}')=-\sgn(b_{ij})$ if $k=i,j$,
    \item $\sgn(b_{ij}b_{\ell m})=\sgn(b_{ij})\sgn(b_{\ell m})$,
    \item\label{lem_sign_bw} $\sgn(b_{ik}b_{kj})=\sgn(b_{jk}b_{ki})$.
\end{enumerate}
\end{lemma}
\begin{proof}
All properties follow from the skew-symmetrizable property and $\sgn(x)\sgn(y)=\sgn(xy)$.
\end{proof}

Fomin-Zelevinsky mutation preserves the property of being skew-symmetrizable; in fact, the skew-symmetrizing matrix remains constant \cite[Proposition 4.5]{fz1}.  We have a similar result in our setting.
\begin{lemma}
If $B$ is a gss matrix with skew-symmetrizing matrix $D=\diag(d_1,\ldots,d_n)$ 
and $B'=\mu_k(B)$, then $B'$ is also a gss matrix with skew-symmetrizing matrix $D$.
\end{lemma}
\begin{proof}
We need to check all four cases of Definition \ref{defn_signed_mut}. 
First, if $k=i$ and $\sgn(b_{ij})=1$, then $\sgn(b_{ji})=-1$, so 
$d_ib_{ij}'=-d_ib_{ij}=d_jb_{ji}=-d_jb_{ji}'$. 
The other possibilities in the first two cases are similar.
Next suppose $\sgn(b_{ik}b_{kj})=1$.  We have 
$d_ib_{ik}b_{kj}=-d_kb_{ki}b_{kj}=d_jb_{jk}b_{ki}$ and $\sgn (b_{ik})=\sgn(b_{kj})=-\sgn(b_{jk})$, so $d_ib_{ij}'=td_i(b_{ij}+\sgn(b_{ik})b_{ik}b_{kj})=td_j(-b_{ji}-\sgn(b_{jk})b_{jk}b_{ki})=-d_jb_{ji}'$.
Finally, in the fourth case, $\sgn(b_{ik}b_{kj})\neq1$, so $\sgn(b_{jk}b_{ki})\neq1$, so $d_ib_{ij}'=d_ib_{ij}=-d_jb_{ji}=-d_jb_{ji}'$.
\end{proof}

Famously, Fomin-Zelevinsky mutation is an involution.  The same does not hold in our situation, but we do have the following:
\begin{lemma}\label{lem:mutate_twice}
Let $B$ be a gss matrix and let $B''=\mu_k^2(B)$.  Then:
\[ b''_{ij}=\begin{cases}
tb_{ij} & \text{if $i=k$ or $j=k$}; \\  
b_{ij} & \text{otherwise}.
\end{cases} \]
\end{lemma}
\begin{proof}
Write $B^{(\ell)}=\mu_k^{\ell}(B)$.
If $i=k$ or $j=k$, Definition \ref{defn_signed_mut} gives $b_{ij}^{(2)}=tb_{ij}^{(0)}$.  
 
Now let $i,j$ be such that $\sgn (b_{ik}^{(0)}b_{kj}^{(0)})=1$.  By Lemma \ref{lem_sign} we have 
\[ \sgn(b_{ik}^{(1)}b_{kj}^{(1)}) = \sgn(b_{ik}^{(1)})\sgn(b_{kj}^{(1)})=(-1)^2\sgn(b_{ik}^{(0)})\sgn(b_{kj}^{(0)})= \sgn(b_{ik}^{(0)}b_{kj}^{(0)}) =1, \]
so Definition \ref{defn_signed_mut} and Lemma \ref{lem_sign} give:
     \begin{align*}
         b_{ij}^{(2)}& = t(b_{ij}^{(1)}+\sgn(b_{ik}^{(1)})b_{ik}^{(1)}b_{kj}^{(1)}) \\
         & = t(t(b_{ij}^{(0)}+\sgn(b_{ik}^{(0)})b_{ik}^{(0)}b_{kj}^{(0)})+\sgn(b_{ik}^{(1)})b_{ik}^{(1)}b_{kj}^{(1)}) \\
         & = t(tb_{ij}^{(0)} +t\sgn(b_{ik}^{(0)})b_{ik}^{(0)}b_{kj}^{(0)} +t\sgn(b_{ik}^{(1)})b_{ik}^{(0)}b_{kj}^{(0)}) \\
         & = t(tb_{ij}^{(0)} +t\sgn(b_{ik}^{(0)})b_{ik}^{(0)}b_{kj}^{(0)} -t\sgn(b_{ik}^{(0)})b_{ik}^{(0)}b_{kj}^{(0)}) \\
         & = t^2 b_{ij}^{(0)}=b_{ij}^{(0)}.
     \end{align*}
In all other cases we have $b_{ij}^{(2)}=b_{ij}^{(1)}=b_{ij}^{(0)}$, concluding the proof. 
\end{proof}

As an immediate consequence, we get:
\begin{proposition}\label{prop_order4}
$\mu_k^4(B)=B$.
\end{proposition}


\subsection{Signed valued quivers}

The classical theory of mutation has a graphical interpretation using valued quivers.  To develop a corresponding interpretation for mutation of gss matrices, we need the following definition.
\begin{definition}
An $n\times n$ gss matrix $B$ is called \textbf{pure} if, for all $i,j\in\{1,\ldots,n\}$, either $b_{ij}\in\mathbb{Z}$ or $b_{ij}\in t\mathbb{Z}$.
\end{definition}

A quiver is called \emph{simple} if its underlying graph is simple, so it has no loops and has at most one arrow between any two vertices.  A quiver is called \emph{cluster} if it satisfies the usual assumptions in cluster theory: it has no loops or 2-cycles.  So the Kronecker quiver is cluster, but is not simple.

We recall some standard notions which can be found in, for example, the introduction of \cite{dr76}.
A \emph{valued graph} $(\Gamma, a)$ is a set $\Gamma$ of vertices
together with non-negative integers $a_{ij}$ for all $i,j\in \Gamma$ such that $a_{ii}=0$ and such that there exist positive integers $d_i$ such that $d_ia_{ij}=d_ja_{ji}$.  An \emph{orientation} $\Omega$ of  $(\Gamma, a)$ is a choice of order $i\to j$ or $j\to i$ for each unordered pair $i,j\in\Gamma$ with $d_{ij}\neq0$.  We can combine these data structures using a \emph{valued quiver}: this is a simple quiver $Q$ together with an ordered pair $\tilde{v}(\alpha)=(a_{ij},a_{ji})$ of positive integers for each arrow $\alpha:i\to j$, satisfying the $d_ia_{ij}=d_ja_{ji}$ condition.

We will need a signed variant of the previous definition.
\begin{definition}
A \textbf{signed valued quiver} is a pair $(Q,v)$ where $Q$ is a simple quiver and 
\[v=(v_1,v_2):Q_1\to \mathbb{Z}\times \mathbb{Z}\] 
is a function on the arrows of $Q$ such that, for all $\alpha\in Q_1$,  $\sgn(v_1(\alpha))=\sgn(v_2(\alpha))\neq0$.  
Moreover, there should exist positive integers $d_i$, for $i\in Q_0$, such that $d_iv_1(\alpha)=d_jv_2(\alpha)$ for all arrows $\alpha:i\to j$.
The ordered pair $v(\alpha)$ is called the \textbf{value} of $\alpha$
and the product $v_1(\alpha)v_2(\alpha)$ is called the \textbf{weight} of $\alpha$. 
We say that an arrow $\alpha$ is $\textbf{positive}$ and has $\textbf{sign +1}$ if $v_1(\alpha)>0$, and is \textbf{negative} and has $\textbf{sign -1}$ if $v_1(\alpha)<0$.  If $(Q,v)$ is a signed valued quiver, its \textbf{specialization} $(Q,v^+)$ is the valued quiver with $v^+(\alpha)=(\abs{v_1(\alpha)},\abs{v_2(\alpha)})$ for each $\alpha\in Q_1$.
\end{definition}
For an arrow $\alpha:i\to j$ with $v(\alpha)=(a,b)$,
we often write $\alpha:i\arr{(a,b)} j$, or sometimes $\alpha:j\rra{(a,b)} i$.

It is occasionally convenient to write $\alpha:i\arr{(0,0)} j$ when there is no arrow between $i$ and $j$, and $\alpha:i\arr{\pm}j$ when there is an arrow of value $(\pm1,\pm1)$. 

If all the arrows in $(Q,v)$ have value $(\pm 1,\pm 1)$, so a sign on the arrows is enough to determine the values, we will simply talk about a \textbf{signed quiver}.

We now explain how to associate a signed valued quiver to a pure gss matrix.  The rough idea is that $+,-$ signs of matrix entries correspond to orientation of arrows, as in the classical setting, and $1,t$ coefficients correspond to positive/negative sign of arrows.  
\begin{definition}\label{def:associated}
Let $B=(b_{ij})$ be a pure gss matrix. Define its \textbf{associated signed valued quiver} to be the pair $(Q,v)=(Q_B,v_B)$ where:
\begin{itemize}
    \item $Q$ has vertex set $\{1,\ldots, n\}$;
    \item there is an arrow $i\to j$ if and only if $\sgn (b_{ij})=1$;
    \item $v:Q_1\to \mathbb{Z}\times \mathbb{Z}$ is given by $v(\alpha)=(b_{ij}(-1),-b_{ji}(-1))$.
\end{itemize}

Given a signed valued quiver $(Q,v)$, we define its  \textbf{associated gss matrix} to be the pure gss matrix $B=B(Q,v)=(b_{ij})$ where
\[ b_{ij} =
\begin{cases}
v_1(\alpha) &\text{if }\exists\text{ positive }\alpha:i\to j;\\
-v_2(\alpha) &\text{if }\exists\text{ positive }\alpha:j\to i;\\
-tv_1(\alpha) &\text{if }\exists\text{ negative }\alpha:i\to j;\\
tv_2(\alpha) &\text{if }\exists\text{ negative }\alpha:j\to i;\\
0 &\text{if there is no arrow between $i$ and $j$.}
\end{cases}
\]
\end{definition}

\begin{example}\label{eg:mutate-quivers}
The signed valued quivers associated to the matrices $B$ and $\mu_2(B)$ in Example \ref{eg_mut_spec_matrix} are the following.
\[
\begin{tikzpicture}[xscale=2,baseline=(bb.base),
  quivarrow/.style={black, -latex}] 
\path (0,0) node (bb) {}; 

\node (v1) at (0,0) {1};
\node (v2) at (0.5,0) {2};
\node (v3) at (1,0) {3,};

\draw [quivarrow, shorten <=-1pt, shorten >=-1pt] (v2) -- (v1) node[midway, above]{\tiny $(-1,-1)$};
\draw [quivarrow, shorten <=-1pt, shorten >=-1pt] (v3) -- (v2) node[midway, above]{\tiny $(-1,-2)$};

\begin{scope}[shift={(2.4,0)}]

\node (v1) at (0,-0.4) {1};

\node (v3) at (0.6,-0.4) {3};
\node (v2) at (0.3,0.4) {2};

\draw [quivarrow, shorten <=-1pt, shorten >=-1pt] (v1) -- (v2) node[midway, above left]{\tiny $(-1,-1)$};
\draw [quivarrow, shorten <=-1pt, shorten >=-1pt] (v2) -- (v3) node[midway, above right]{\tiny $(2,1)$};
\draw [quivarrow, shorten <=-1pt, shorten >=-1pt] (v3) -- (v1) node[midway, below]{\tiny $(-1,-2)$};
    
\end{scope}
\end{tikzpicture}
\]
\end{example}

\begin{proposition}
The two constructions in Definition \ref{def:associated} are mutually inverse.
\end{proposition}
\begin{proof}
Start with a pure gss matrix $B=(b_{ij})$.  Let $(Q,v)$ be its associated signed valued quiver, and let $A=(a_{ij})$ be the associated gss matrix of $(Q,v)$.  If $b_{ij}=0$ then also $b_{ji}=0$, so there is no arrow between $i$ and $j$, so $a_{ij}=0$.  If $b_{ij}\neq 0$ and $\sgn (b_{ij})=1$, we get an arrow $\alpha:i\to j$ in $Q$ with $v(\alpha)=(b_{ij}(-1),-b_{ji}(-1))$.  If $b_{ij}\in\Z$ then $v_1(\alpha)=b_{ij}$.  As $\sgn (b_{ij})=1$ we have $b_{ij}>0$, so $\alpha$ is positive, so $a_{ij}=b_{ij}$.  If $b_{ij}=bt\in t\Z$ then $v_1(\alpha)=-b<0$ 
because $\sgn (b_{ij})=1$.  So $\alpha$ is negative, so $a_{ij}=(-t)(-b)=bt=b_{ij}$.     If $b_{ij}\neq 0$ and $\sgn (b_{ij})=-1$, just swap $i$ and $j$.

Now start with a signed valued quiver $(Q,v)$.  Construct the associated gss matrix $B=(b_{ij})$ and its associated signed valued quiver $(R,w)$.  If $Q$ has no arrow between $i$ and $j$ then $b_{ij}=0$, so neither does $R$.  If $\alpha:i\to j$ in $Q$ is positive then $b_{ij}=v_1(\alpha)$ is a positive integer, and $b_{ji}=-v_2(\alpha)\in\Z$.  So we get an arrow $\alpha:i\to j$ in $R$ and $w(\alpha)=(v_1(\alpha),v_2(\alpha))$.  If instead $\alpha:i\to j$ in $Q$ is negative then $b_{ij}=-tv_1(\alpha)$ is a positive multiple of $t$, and $b_{ji}=tv_2(\alpha)\in\Z$.  So we get an arrow $\alpha:i\to j$ in $R$ and again $w(\alpha)=(v_1(\alpha),v_2(\alpha))$.
\end{proof}

We would like to define mutation of signed valued quivers by first constructing the associated (pure) gss matrix, then applying Definition \ref{defn_signed_mut}, and then taking the associated signed valued quiver.  However, as the following example shows, the mutation of a pure gss matrix may not be pure.
\begin{example}
Let $Q$ be an oriented $3$-cycle and give all arrows value $(-1,-1)$.  Then the associated gss matrix $B$ and the signed mutation $B'=\mu_2(B)$ are:
\[B=\begin{pmatrix}
0 &t &-t\\
-t &0 &t\\
t &-t& 0
\end{pmatrix}
\;\; \text{ and } \;\;
B'=\begin{pmatrix}
0 &-1 &-1+t\\
1 &0 &-t\\
1-t &t& 0
\end{pmatrix}.\]
\end{example}
\begin{definition}
We say that a pure gss matrix $B$ satisfies the \textbf{positive 3-cycle condition at $k$} if, for all $i,j\neq k$ 
with $\sgn (b_{ik}b_{kj})=1$, we have
\[ b_{ij}b_{ik}b_{kj}\in \mathbb{Z}. \]
\end{definition}
\begin{remark}
In terms of the associated signed valued quiver of $B$, the positive 3-cycle condition at $k$ says that, whenever we have arrows $i\to k\to j$ and an arrow between $i$ and $j$ in either direction, the product of the signs of the three arrows is positive.  We will apply the condition to signed valued quivers below.
\end{remark}
\begin{lemma}\label{lem_mut_specializ}
Let $B=(b_{ij})$ be a pure gss matrix. Then $B'=\mu_k(B)=(b_{ij}')$ is pure if and only if $B$ satisfies the positive 3-cycle condition at $k$.
\end{lemma}
\begin{proof}
Since $B$ is pure, we have $b_{ij}'\in\mathbb{Z}\cup t\mathbb{Z}$ if $k=i,j$ by definition of $\mu_k$. This is also true if $\sgn(b_{ik}b_{kj})\neq 1$, and if $\sgn (b_{ik}b_{kj})=1$ and $b_{ij}= 0$.
So assume $\sgn (b_{ik}b_{kj})=1$ and $b_{ij}\neq0$.  Then $B'$ is pure if and only if
\[ b_{ij}'=t(b_{ij}+\sgn (b_{ik})b_{ik}b_{kj})\in\mathbb{Z}\cup t\mathbb{Z}
\]
for all $i,j$ such that $\sgn(b_{ik}b_{kj})=1$.  Consider the $\Z$ and $t\Z$ situations separately:
\begin{itemize}
    \item $b_{ij}'\in\mathbb{Z}$ if and only if ($b_{ij},b_{ik}\in t \mathbb{Z}$ and $b_{kj}\in\mathbb{Z}$) or ($b_{ij},b_{kj}\in t\mathbb{Z}$ and $b_{ik}\in\mathbb{Z}$);
    \item $b_{ij}'\in t\mathbb{Z}$ if and only if ($b_{ij},b_{ik},b_{kj}\in \mathbb{Z}$) or ($b_{ij}\in \mathbb{Z}$ and $b_{ik},b_{kj}\in t \mathbb{Z}$).
\end{itemize}
The cases described above reduce to $b_{ij}b_{ik}b_{kj}\in \mathbb{Z}$.
\end{proof}

\begin{remark}
Under the assumptions of Lemma \ref{lem_mut_specializ}, the pure matrix $B'=\mu_k(B)$ also satisfies the positive 3-cycle condition at $k$. This is a straightforward consequence of the proof above. For example, if $b_{ij}'\in \mathbb{Z}$ and $b_{ij},b_{ik}\in t \mathbb{Z}$ and $b_{kj}\in\mathbb{Z}$, Definition \ref{defn_signed_mut} yields $b_{ik}'b_{kj}'\in \mathbb{Z}$, and so 
$b_{ij}'b_{ik}'b_{kj}'\in\mathbb{Z}$.
\end{remark}

Before describing mutation of signed valued quivers, it helps to introduce some terminology.
\begin{definition}
For arrows $\alpha=i\arr{(a,b)} j$ and $\beta=j\arr{(c,d)} h$, we define the following operations:
\begin{itemize}
    \item \textbf{reverse} $\alpha$ to get $\alpha^*:j\arr{(b,a)} i$;
    \item \textbf{negate} $\alpha$ to get $-\alpha:i\arr{\scalemath{.7}{(-a,-b)}} j$;
    \item \textbf{compose} $\alpha$ and $\beta$ to get
    $[\alpha\beta]:i\arr{\scalemath{.7}{(ac,bd)}} h$.
\end{itemize}
\end{definition}
\begin{definition}\label{lem:trichot}
For arrows $\alpha:i\arr{(a,b)} j$ and $\beta:j\arr{(c,d)} i$, 
we say that $\alpha$ is:
\begin{itemize}
\item \textbf{smaller} than $\beta$ if $\abs{a}<\abs d$ and $\abs b<\abs c$;
\item \textbf{bigger} than $\beta$ if $\abs{a}>\abs d$ and $\abs b>\abs c$;
\item \textbf{the same size} as $\beta$ if $a=d$ and $b=c$. 
\end{itemize}
We say that $\alpha$ is \textbf{comparable} to $\beta$ if one of the three conditions above hold.
\end{definition}

For the rest of this section, let $(Q,v)$ denote a signed valued quiver.
\begin{lemma}\label{lem_cond_signed_val_3_cycle}
For an oriented 3-cycle $1\arr\alpha 2\arr\beta 3\arr\gamma1$ in $(Q,v)$, where $[\alpha\beta\gamma]$ is positive, 
the composed arrow $[\alpha\beta]$
is comparable to $\gamma$.
\end{lemma}
\begin{proof}
Suppose our arrows have the following values:
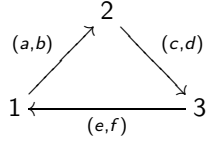
\[ \xymatrix{
 &2 \ar[dr]^{(c,d)} &\\
1\ar[ur]^{(a,b)} && 3\ar[ll]^{(e,f)}
}\]
Let $D=\diag (d_1,d_2,d_3)$ be the skew-symmetrizing matrix of $Q$, so that: 
\begin{align*}
  d_1a &= bd_2 \\
  d_2c &= dd_3 \\
  d_3e &= fd_1
\end{align*}
with $d_1,d_2,d_3>0$. We have:
\[ \abs{ac}-\abs{f}= \abs{bd}d_3/d_1-\abs{e}d_3/d_1=\frac{d_3}{d_1}(\abs{bd}-\abs{e}),\]
so $\abs{ac}>\abs{f} \iff \abs{bd}>\abs{e}$
and $\abs{ac}<\abs{f} \iff \abs{bd}<\abs{e}$.  
Moreover, if $\abs{ac}=\abs{f}$ and $\abs{bd}=\abs{e}$ then, because $[\alpha\beta\gamma]$ is positive, $[\alpha\beta]$ and $\gamma$ must have the same sign, so ${ac}={f}$ and ${bd}={e}$.
Therefore we are in one of the cases of Definition \ref{lem:trichot}.
\end{proof}

We now give the mutation procedure for signed valued quivers.  Lemma \ref{lem_cond_signed_val_3_cycle} ensures that the second part of Step 3 always makes sense.
\begin{definition}\label{def:mut_signed_val_quivs}
Let $(Q,v)$ be a signed valued quiver which satisfies the positive 3-cycle condition at $k\in Q_0$. 
Define the \textbf{signed quiver mutation} $\mu_k(Q,v)$ as follows:
\begin{enumerate}[align=left]
    \item[Step 1:] If there are arrows $i\arr{\alpha} k\arr{\beta} j$, add a new arrow $-[\alpha\beta]:i\to j$ and 
    replace any arrow 
    between $i$ and $j$
    with its negation.
    \item[Step 2:] Reverse any arrow going out of $k$, and reverse and negate any arrow going into $k$. 
    \item[Step 3:] If we have two arrows $\gamma,\delta$ between $i$ and $j$ as a result of Step 1, then:
    \begin{itemize}
        \item if $\gamma,\delta:i\to j$, replace them with a single arrow having value $v(\gamma)+v(\delta)$;
        \item if $\gamma:j\to i$ and $\delta:i\to j$, reverse and negate the smaller (or same size) arrow, and then combine them as above.
    \end{itemize}
\end{enumerate}
\end{definition}

The following diagram exhibits the mutation explicitly (we assume $a,b,c,d\neq 0$):
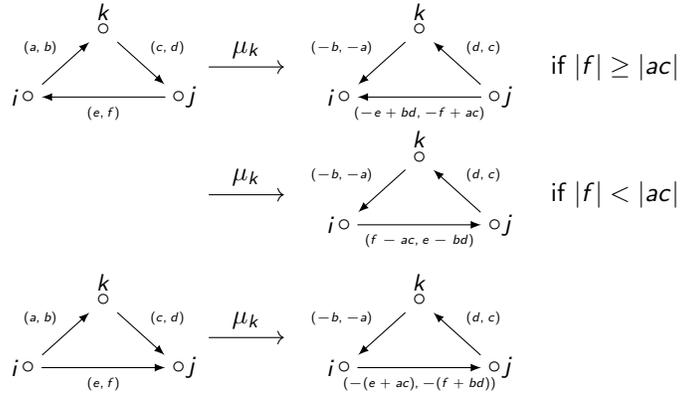
\[    \centering
\begin{tikzpicture}[xscale=2,baseline=(bb.base),
  quivarrow/.style={black, -latex}] 
\path (0,0) node (bb) {}; 

\node (i1) at (0,0) {$\circ$};
\node [left] at (i1) {\small {$i$}};

\node (j1) at (1,0) {$\circ$};
\node [right] at (j1) {\small {$j$}};

\node (k1) at (0.5,0.9) {$\circ$};
\node [above] at (k1) {\small {$k$}};

\draw [quivarrow, shorten <=-1pt, shorten >=-1pt] (i1) -- (k1) node[midway, above left]{\tiny $(a,b)$};
\draw [quivarrow, shorten <=-1pt, shorten >=-1pt] (k1) -- (j1) node[midway, above right]{\tiny $(c,d)$};
\draw [quivarrow, shorten <=-1pt, shorten >=-1pt] (j1) -- (i1) node[midway, below]{\tiny $(e,f)$};

\draw [->] (1.2,0.4) to node[above] {$\mu_k$} (1.7,0.4);
\draw [->] (1.2,-1.3) to node[above] {$\mu_k$} (1.7,-1.3);

\begin{scope}[shift={(2.1,0)}]
\node (i1) at (0,0) {$\circ$};
\node [left] at (i1) {\small {$i$}};

\node (j1) at (1,0) {$\circ$};
\node [right] at (j1) {\small {$j$}};

\node (k1) at (0.5,0.9) {$\circ$};
\node [above] at (k1) {\small {$k$}};

\draw [quivarrow, shorten <=-1pt, shorten >=-1pt] (k1) -- (i1) node[midway, above left]{\tiny $(-b,-a)$};
\draw [quivarrow, shorten <=-1pt, shorten >=-1pt] (j1) -- (k1) node[midway, above right]{\tiny $(d,c)$};
\draw [quivarrow, shorten <=-1pt, shorten >=-1pt] (j1) -- (i1) node[midway, below]{\tiny $(-e+bd,-f+ac)$};

\node (text1) at (1.8,0.4) {\text{if }$|f|\geq |ac|$};
\end{scope}

\begin{scope}[shift={(2.1,-1.7)}]
\node (i1) at (0,0) {$\circ$};
\node [left] at (i1) {\small {$i$}};

\node (j1) at (1,0) {$\circ$};
\node [right] at (j1) {\small {$j$}};

\node (k1) at (0.5,0.9) {$\circ$};
\node [above] at (k1) {\small {$k$}};

\draw [quivarrow, shorten <=-1pt, shorten >=-1pt] (k1) -- (i1) node[midway, above left]{\tiny $(-b,-a)$};
\draw [quivarrow, shorten <=-1pt, shorten >=-1pt] (j1) -- (k1) node[midway, above right]{\tiny $(d,c)$};
\draw [quivarrow, shorten <=-1pt, shorten >=-1pt] (i1) -- (j1) node[midway, below]{\tiny $(f-ac,e-bd)$};

\node (text1) at (1.8,0.4) {\text{if }$|f|<|ac|$};
    
\end{scope}

\begin{scope}[shift={(0,-3.6)}]
\node (i1) at (0,0) {$\circ$};
\node [left] at (i1) {\small {$i$}};

\node (j1) at (1,0) {$\circ$};
\node [right] at (j1) {\small {$j$}};

\node (k1) at (0.5,0.9) {$\circ$};
\node [above] at (k1) {\small {$k$}};

\draw [quivarrow, shorten <=-1pt, shorten >=-1pt] (i1) -- (k1) node[midway, above left]{\tiny $(a,b)$};
\draw [quivarrow, shorten <=-1pt, shorten >=-1pt] (k1) -- (j1) node[midway, above right]{\tiny $(c,d)$};
\draw [quivarrow, shorten <=-1pt, shorten >=-1pt] (i1) -- (j1) node[midway, below]{\tiny $(e,f)$};

\draw [->] (1.2,0.4) to node[above] {$\mu_k$} (1.7,0.4);

\begin{scope}[shift={(2.1,0)}]
\node (i1) at (0,0) {$\circ$};
\node [left] at (i1) {\small {$i$}};

\node (j1) at (1,0) {$\circ$};
\node [right] at (j1) {\small {$j$}};

\node (k1) at (0.5,0.9) {$\circ$};
\node [above] at (k1) {\small {$k$}};

\draw [quivarrow, shorten <=-1pt, shorten >=-1pt] (k1) -- (i1) node[midway, above left]{\tiny $(-b,-a)$};
\draw [quivarrow, shorten <=-1pt, shorten >=-1pt] (j1) -- (k1) node[midway, above right]{\tiny $(d,c)$};
\draw [quivarrow, shorten <=-1pt, shorten >=-1pt] (i1) -- (j1) node[midway, below]{\tiny $(-(e+ac),-(f+bd))$};

\end{scope}
\end{scope}

\end{tikzpicture}
\]

Recall that $B(Q,v)$ denotes the associated gss matrix of a signed valued quiver $(Q,v)$.
\begin{proposition}\label{prop_mut_signed_val_quivs}
If $(Q,v)$ satisfies the positive 3-cycle condition at $k$ then $\mu_k B(Q,v)=B\mu_k(Q,v)$.
\end{proposition}
\begin{proof}
This is a case-by-case check using Definitions \ref{defn_signed_mut}, \ref{def:mut_signed_val_quivs}, and \ref{def:associated}.  We do not write out all the details; instead we write out some checks in full and then list what other checks need to be done.

If $Q$ has a positive arrow $\alpha:i\to k$ with $v(\alpha)=(a,b)$, then $b_{ik}=v_1(\alpha)=a>0$ and $b_{ki}=-v_2(\alpha)=-b<0$.  So $b_{ik}'=-tb_{ik}=-ta$ and $b_{ki}'=-tb_{ki}=tb$.  As $\sgn(b_{ki}')=+1$ we get an arrow $k\to i$ with value $(b_{ki}'(-1),\,-b_{ik}'(-1))=(-b,-a)$.  So $\alpha$ has been negated and reversed, as instructed by Step 2.  One checks that the same thing happens if $\alpha$ is negative.  Then one can check that arrows $\beta:k\to j$, of positive or negative sign, are reversed.  Next one can check that if $Q$ has an arrow $\gamma:i\to j$ 
and there do not exist paths $i\to k\to j$ or $j\to k\to i$ in $Q$, then $\gamma$ is unchanged.

Now suppose that $Q$ has arrows $\alpha:i\to k$, $\beta:k\to j$, and $\gamma:i\to j$, all positive, with $v(\alpha)=(a,b)$, $v(\beta)=(c,d)$, and $v(\gamma)=(e,f)$.  
Then
\begin{align*}
b_{ij}' &= t(b_{ij}+b_{ik}b_{kj})=t(e+ac); \\
b_{ji}' &= t(b_{ji}-b_{jk}b_{ki})=-t(f+bd).
\end{align*}
As $\sgn(b_{ij}')=+1$ we get an arrow $i\to j$ with value $(-e-ac,\,-f-bd)$.  This matches the first part of Step 3: we have replaced the arrows $-\gamma$ and $-[\alpha\beta]$ from Step 1 with a single arrow, summing their values. One can do similar checks where $\alpha,\beta,\gamma$ have positive or negative sign, subject to the positive 3-cycle condition, and get the same result in all cases.  This argument also works if there are no arrows between $i$ and $j$.

Finally suppose that $Q$ has arrows $\alpha:i\to k$, $\beta:k\to j$, and $\gamma:j\to i$.  If $\gamma$ is positive then, in the mutation of the associated gss matrix, $b_{ij}'=t(ac-f)$ and $b_{ji}'=t(e-bd)$.  If instead $\gamma$ is negative then $b_{ij}'=f-ac$ and $b_{ji}'=bd-e$.  The associated signed valued quiver depends on whether $\abs{ac}>\abs{f}$ or $\abs{ac}\leq\abs{f}$: in the former case we get an arrow $i\to j$ with value $(f-ac,e-bd)$, irrespective of the sign of $\gamma$, and in the latter case we get $j\to i$ with value $(bd-e,ac-f)$.  This matches the second part of Step 3.
\end{proof}

\begin{example}
Mutating the first signed valued quiver in Example \ref{eg:mutate-quivers} at vertex 2 gives the second.
\end{example}

\begin{example}
Consider the following signed valued quiver: 
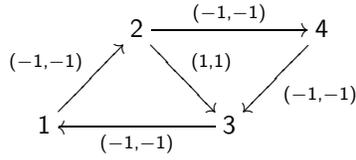
\[
\xymatrix{ & 2\ar[dr]^{(1,1)}\ar[rr]^{(-1,-1)} & & 4\ar[dl]^{(-1,-1)} \\
1\ar[ur]^{(-1,-1)} & & 3\ar[ll]^{(-1,-1)} &}
\]
Note that the positive 3-cycle condition is satisfied at all vertices.  We mutate at vertex $4$ to get the following signed valued quiver: 
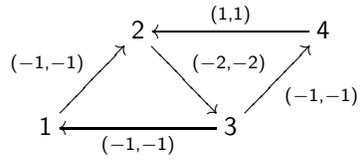
\[
\xymatrix{ & 2\ar[dr]^{(-2,-2)} & & 4\ar[ll]_{(1,1)} \\
1\ar[ur]^{(-1,-1)} & & 3\ar[ll]^{(-1,-1)}\ar[ur]_{(-1,-1)} &}.
\]
Now the 3-cycle $1\to 2\to 3\to 1$ has negative product of signs, so the above signed valued quiver does not satisfy the positive 3-cycle condition at vertices $1$, $2$, or $3$.
\end{example}


\subsection{Mutation Dynkin type}

\begin{definition}
We say that two gss matrices $B$ and $B'$ are \textbf{mutation equivalent} if there exists a sequence of indices $i_1,\ldots,i_r$ such that $B'=\mu_{i_r}\ldots\mu_{i_1}B$.  The \textbf{mutation class} of $B$ is the set of all gss matrices which are mutation equivalent to $B$.
\end{definition}
If $(Q,v)$ is a signed valued quiver, we can apply the previous definition by taking the associated gss matrix $B=B(Q,v)$.  But, in general, the mutation class of a signed valued quiver may include gss matrices which are not pure, and so have no associated signed valued quiver.  We will soon see that this problem does not occur for the appropriate notion of Dynkin signed valued quivers.

\begin{proposition}\label{prop_tree_assign_values}
Let $(Q,v)$ be a signed valued quiver whose underlying graph is a tree. 
Let $(Q',v')$ be obtained by reversing and/or negating some arrows of $(Q,v)$.  Then $(Q,v)$ and $(Q',v')$ are mutation equivalent.
\end{proposition}
\begin{proof}
It is enough to show we can reverse or negate a single arrow.  We start with negation.  As the underlying graph of $Q$ is a tree, there is a unique (undirected) path between any two vertices.  Given an arrow $\alpha:i\to j$ in $Q$, let
\[P_-(\alpha)=\{k\in Q_0\st \text{the unique path from $k$ to $i$ does not pass through  $j$}\}\]
and write this as the disjoint union
\[ P_-(\alpha)=P_-^0(\alpha)\cup P_-^1(\alpha) \]
where $P_-^0(\alpha)$ and $P_-^1(\alpha)$ denote the subsets where the path has even and odd length, respectively.
By Lemma \ref{lem:mutate_twice}, $\mu_k^2$ just changes the sign of all arrows incident to $k$.  Note that $\mu_h^2\mu_k^2=\mu_k^2\mu_h^2$ for any $h,k\in Q_0$.
So
\[\left(\prod_{k\in {P}_-(\alpha)}\mu_k^2\right)(Q,v)=\left(\prod_{k\in {P}^1_-(\alpha)}\mu_k^2\right)\left(\prod_{k\in {P}^0_-(\alpha)}\mu_k^2\right)(Q,v)
\]
is $(Q,v)$ with $\alpha$ negated,
because $\prod_{k\in {P}^0_-(\alpha)}\mu_k^2$ changes the sign of all arrows incident with $P_-(\alpha)$, and $\prod_{k\in {P}^1_-(\alpha)}\mu_k^2$ changes the sign of all arrows incident with $P_-(\alpha)-\{i\}$.

Now we consider reversing arrows.  If $Q'$ is another orientation of $Q$, we know from classical results that the (unsigned) quivers $Q$ and $Q'$ are mutation equivalent via mutations only at 
sources (see, e.g., \cite[Proposition 2.2.8]{lampe} for a proof).
From Definition \ref{def:mut_signed_val_quivs}, source mutations just reverse arrows, 
so this sequence of signed mutations gives the result.
\end{proof} 

We want a data structure that includes the classical Dynkin diagrams.  We 
will use graphs with a choice of ordering $i<j$, $i=j$, or $i>j$ whenever $i$ and $j$ have at least 2 edges between them.
\begin{definition}\label{def:unsigned-diagram}
The \textbf{unsigned diagram} $\tilde\Gamma=\tilde\Gamma(Q,v)$ of a signed valued quiver $(Q,v)$ has the same vertex set as $Q$.  If $Q$ has an arrow $\alpha:i\to j$ then $\tilde\Gamma$ has $v_1(\alpha)v_2(\alpha)$ edges between $i$ and $j$, with ordering $i>j$ iff $\abs{v_1(\alpha)}>\abs{v_2(\alpha)}$ and $i<j$ iff $\abs{v_1(\alpha)}<\abs{v_2(\alpha)}$.
\end{definition}
Recall the Dynkin diagrams of types $A_n$, $B_n$, $C_n$, $D_n$, $E_n$, $F_4$, and $G_2$: see, e.g., \cite[Section V.15]{serre}.
\begin{definition}
A gss matrix is called \textbf{Dynkin} if it is pure and the unsigned diagram of its associated signed valued quiver is a Dynkin diagram.  A gss matrix is called \textbf{mutation Dynkin} if it is mutation equivalent to a Dynkin gss matrix.
\end{definition}
Proposition \ref{prop_tree_assign_values} tells us that any two gss matrices which are Dynkin and of the same type are mutation equivalent.  The mutation classes of different Dynkin types are distinct: this follows from the classical result (see \cite[Theorem 1.7]{fz2}) together with Proposition \ref{prop_mutat_agree}.

\begin{example}
The signed valued quiver
\[\xymatrix{
1 \ar[r]^{(1,1)} &2  & 3\ar[l]_{(-1,-1)} & 4\ar[l]_{(1,2)}
}
\]
has unsigned diagram
\[\xymatrix {
1 \ar@{-}[r] &2 \ar@{-}[r] & 3\ar@{=}[r] |{\SelectTips{cm}{}\object@{>}} & 4
}
\]
and so is Dynkin of type $B_4$.  The signed valued quiver
\[\xymatrix{
1 \ar[r]^{(2,2)} &2
}
\]
is not mutation Dynkin.
\end{example}

\begin{definition}
Let $B$ be a gss matrix. A $p$\textbf{-cycle} in $B$ is a $p$-tuple $(i_1,\ldots, i_p)$ with $i_j\in\{1,\ldots, n\}$ distinct such that $b_{i_j,i_{j+1}}\neq 0$ for all $j\in\{1,\ldots, p\}$ (we interpret all subscripts modulo $p$). 

The cycle is called:
\begin{itemize}
\item \textbf{oriented} if $\sgn(b_{i_j,i_{j+1}})=1$ for all $j\in\{1,\ldots, p\}$;
\item \textbf{chordless} if for all $j\in\{1,\ldots, p\}$ and all $k\in\{1,\ldots, p\}\setminus \{j+1,j-1\}$ we have $b_{i_ji_k}=0$.
\end{itemize} 
\end{definition}
Note that if $B$ is pure then a chordless cycle in $B$ corresponds to a chordless cycle in $Q_B$, as defined in the introduction of \cite{bgz}.

The following definition is reminiscent of the condition in \cite[Proposition 1.4]{bgz}.
\begin{definition}
A $p$-cycle $(i_1,\ldots, i_p)$ in a pure gss matrix $B$ is called \textbf{dangerous} if it is chordless and 
\[\prod_{j=1}^p b_{i_j,i_{j+1}}\in t^{p}\mathbb{Z}.\]
\end{definition}

\begin{remark}\label{rmk_dang_pos_3_cyc}
Note that a dangerous $3$-cycle must contain a vertex which does not satisfy the positive $3$-cycle condition, and therefore has a mutation which is not pure.  We will prove that a similar phenomenon holds for all dangerous $p$-cycles, but first we illustrate this graphically.
\end{remark}

\begin{example}
The first quiver below has a dangerous oriented $5$-cycle $(1,2,3,4,5)$.  We construct a dangerous oriented $3$-cycle $(3,4,5)$ by mutating twice:
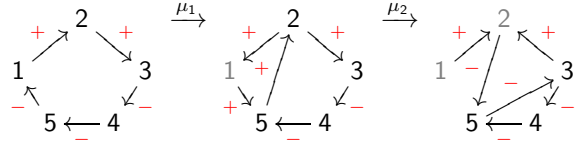
\[
\xymatrix@C=0.1em@R=0.7em{ & & 2\ar[drr]^{\color{red}{+}} & & \\
1\ar[urr]^{\color{red}{+}} & & & & 3\ar[dl]^{\color{red}{-}} \\
& 5\ar[ul]^{\color{red}{-}} & & 4\ar[ll]^{\color{red}{-}} & }\xrightarrow{\mu_1}
\xymatrix@C=0.1em@R=0.7em{ & & 2\ar[drr]^{\color{red}{+}}\ar@[gray][dll]_{\color{red}{+}} & & \\
\color{gray}{1}\ar@[gray][dr]_{\color{red}{+}} & & & & 3\ar[dl]^{\color{red}{-}} \\
& 5\ar[uur]^{\color{red}{+}} & & 4\ar[ll]^{\color{red}{-}} & }\xrightarrow{\mu_2}\xymatrix@C=0.1em@R=0.7em{ & & \color{gray}{2}\ar@[gray][ddl]_{\color{red}{-}} & & \\
\color{gray}{1}\ar@[gray][urr]^{\color{red}{+}} & & & & 3\ar[dl]^{\color{red}{-}}\ar@[gray][ull]_{\color{red}{+}} \\
& 5\ar[urrr]^{\color{red}{-}} & & 4\ar[ll]^{\color{red}{-}} & }
\]
\end{example}

\begin{lemma}\label{lem_non_spec_cycle}
Let $B$ be a pure gss matrix with a dangerous oriented $p$-cycle $(i_1,\ldots, i_p)$.  Then
\[
\mu_{i_{p-2}}\mu_{i_{p-3}}\ldots \mu_{i_1}(B)
\]
is not pure.
\end{lemma}
\begin{proof}
Let $B^{(h)}=(b_{ij}^{(h)})=\mu_{i_h}\ldots\mu_{i_1}(B)$ for $h=0,\ldots, p-3$. We want to show $B^{(p-2)}$ is not pure. 

We claim that $B^{(h)}$ has a dangerous oriented $(p-h)$-cycle $(i_{h+1},\ldots,i_{p-1},i_p)$ 
for all $h=0,\ldots, p-3$.  The base case for $B^{(0)}$ is true by assumption, so assume it holds for $B^{(h)}$, and we need to show it holds for $B^{(h+1)}=\mu_{h+1}B^{(h)}$.

First, as the cycle $(i_{h+1},\ldots,i_{p-1},i_p)$ in $B^{(h)}$ is chordless, we have
$b^{(h+1)}_{i_ji_{j+1}}=b^{(h)}_{i_ji_{j+1}}$ for $j=h+1,\ldots,p-1$.  
as we are always in the fourth case of Definition \ref{defn_signed_mut}.  Then using the third case we get
\[ b^{(h+1)}_{i_pi_{h+2}}=t(0+1(b_{i_pi_{h+1}}^{(h)}b_{i_{h+1}i_{h+2}}^{(h)}))=tb^{(h)}_{i_pi_{h+1}}b^{(h)}_{i_{h+1}i_{h+2}}\]
so we have shown that $(i_{h+2},\ldots,i_{p-1},i_p)$ is a chordless oriented $(p-h-1)$-cycle, and it follows directly that
\[\left(\prod_{j=h+2}^{p-1} b^{(h+1)}_{i_j,i_{j+1}}\right)b^{(h+1)}_{i_pi_{h+2}}=t\left(\prod_{j=h+1}^{p-1} b^{(h)}_{i_j,i_{j+1}}\right)b^{(h)}_{i_pi_{h+1}}
\in t\left(t^{p-h}\mathbb{Z}\right)=t^{p-h-1}\mathbb{Z}\]
showing that our cycle is dangerous.

Therefore $B^{(p-3)}$ has a dangerous $3$-cycle so, by Lemma \ref{lem_mut_specializ}, $B^{(p-2)}$ is not pure.
\end{proof}

\begin{remark}\label{rmk:dangerous}
One can show that any pure gss matrix with a dangerous $p$-cycle is mutation equivalent to a gss matrix which is not pure, whether the cycle is oriented or not.  First, by mutating if necessary, we get some $1\leq j\leq p$ with $\sgn(b_{i_{j-1}i_j})=\sgn(b_{i_ji_{j+1}})=1$.  Then, as in the proof of Lemma \ref{lem_non_spec_cycle}, we mutate at $j$ to create a dangerous $(p-1)$-cycle, and continue this process.
\end{remark}
\begin{example}
It is not necessary to have a dangerous cycle in order to be mutation equivalent to a gss matrix which is not pure.  The following valued quiver is acyclic, but mutating at the vertex $5$ creates a dangerous $4$-cycle $(1,2,3,4)$.  Then, following the procedure described in Remark \ref{rmk:dangerous}, we mutate at $4$ to give us composable arrows passing through $1$, then mutate at $1$ to get our dangerous (unoriented) $3$-cycle $(2,3,4)$.  Then the vertex $3$ does not satisfy the positive 3-cycle condition, so mutating at $3$ would give a gss matrix which is not pure.
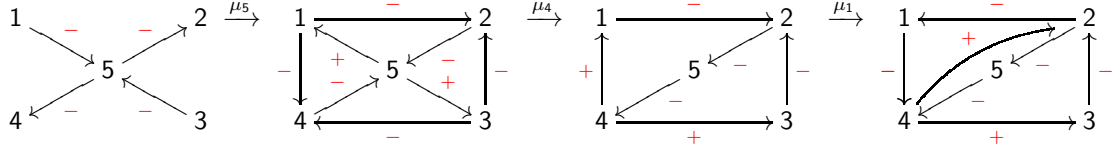
\[
\xymatrix@R=0.7em{ 
1\ar[dr]^{\color{red}{-}} &&2 \\
 &5\ar[ur]^{\color{red}{-}}\ar[dl]^{\color{red}{-}} & \\
4 && 3\ar[ul]^{\color{red}{-}}}
\xrightarrow{\mu_5}
\xymatrix@R=0.7em{ 
1\ar[rr]^{\color{red}{-}}\ar[dd]_{\color{red}{-}} &&2\ar[dl]^{\color{red}{-}} \\
 &5\ar[ul]^{\color{red}{+}}\ar[dr]^{\color{red}{+}} & \\
4\ar[ur]^{\color{red}{-}} && 3\ar[uu]_{\color{red}{-}}\ar[ll]^{\color{red}{-}}}
\xrightarrow{\mu_4}
\xymatrix@R=0.7em{ 
1\ar[rr]^{\color{red}{-}} &&2\ar[dl]^(0.6){\color{red}{-}} \\
 &5\ar[dl]^(0.3){\color{red}{-}} & \\
4\ar[uu]^{\color{red}{+}}\ar[rr]_{\color{red}{+}} && 3\ar[uu]_{\color{red}{-}}
}
\xrightarrow{\mu_1}
\xymatrix@R=0.7em{ 
1\ar[dd]_{\color{red}{-}}\ar[dd]_{\color{red}{-}} &&2\ar[dl]^(0.6){\color{red}{-}}\ar[ll]_{\color{red}{-}} \\
 &5\ar[dl]^(0.3){\color{red}{-}} & \\
4\ar@/^/[]!<1ex,1.5ex>;[uurr]!<-3ex,-0.8ex>^{\color{red}{+}}\ar[rr]_{\color{red}{+}} && 3\ar[uu]_{\color{red}{-}}
}
\]
\end{example}

We now record some useful results about cycles in mutation Dynkin valued quivers.
\begin{lemma}\label{lem_oriented_cycles_values_3_cycles} 
Let $(Q,v)$ be a mutation Dynkin (unsigned) valued quiver. 
Then:
\begin{enumerate}
    \item Every chordless cycle in $Q$ is oriented.
    \item For any 3-cycle $C=(i,j,k)$ in $Q$, the values of the arrows in $C$ are either $\{(1,1),(1,1),(1,1)\}$ or $\{(1,1),(1,2),(2,1)\}$, ordered clockwise or anticlockwise.
\end{enumerate}
\end{lemma}
\begin{proof}
For 1., by \cite[Theorem 1.2]{bgz} we know that every chordless cycle in a skew-symmetrizable matrix of mutation Dynkin type is oriented.
As for 2., by \cite[Lemma 7.6]{fz2} we know that the edge weights for a 3-cycle are $\{1,1,1\}$ or $\{1,2,2\}$.  But the values $\{(1,1),(1,2),(1,2)\}$ are not skew-symmetrizable.
\end{proof}

\begin{proposition}\label{prop:mutat-eq-Dynkin-no-dang}
If the specialization $B(1)$ of $B$ is mutation Dynkin, and $B$ contains no dangerous cycles, then every gss matrix mutation equivalent to $B$ is pure.
\end{proposition}
\begin{proof}
Suppose, for contradiction, that there is a mutation sequence $(k_1,\ldots, k_{\ell})$ such that 
$$
B^{(\ell)}:=\mu_{k_{\ell}}\ldots\mu_{k_1}B=(b_{ij}^{(\ell)})
$$ 
is not pure. Let $\ell$ be minimal, so any mutation sequence of length smaller than $\ell$ yields a pure gss matrix.
We will prove by downward induction on $m=\ell-1,\ldots, 0$ that 
$B^{(m)}$ has a dangerous oriented $(\ell -m+2)$-cycle.  The case $m=0$ will give our contradiction.

For the base case, take $m=\ell -1$. As $B^{(\ell )}=\mu_{k_\ell}B^{(\ell -1)}$ is not pure, then $B^{(\ell -1)}$ does not satisfy the positive 3-cycle condition at $k_{\ell}$, and so it contains a dangerous 3-cycle by Remark \ref{rmk_dang_pos_3_cyc}. Such a 3-cycle must be oriented by Lemma \ref{lem_oriented_cycles_values_3_cycles}.

Now suppose that $m<\ell -1$ and assume that:
\[ \text{$B^{(m+1)}$ contains a dangerous oriented $(\ell -m+1)$-cycle $C=(i_1,\ldots, i_{\ell -m+1})$.} \] 
Our proof has two steps: first we will show that the next vertex we mutate at does not belong to $C$, i.e., that $k_{m+1}\notin\{i_1,\ldots, i_{\ell -m+1}\}$.  
Then we show that $B^{(m)}$ has a dangerous oriented $(\ell -m+2)$-cycle on the vertices of $C$ together with $k_{m+1}$.

\textbf{Step 1:} 
Suppose for contradiction that $k_{m+1}\in\{i_1,\ldots, i_{\ell -m+1}\}$ and,  without loss of generality,  $k_{m+1}=i_1$. By Lemma \ref{lem_oriented_cycles_values_3_cycles}(i), $C$ is oriented, and we can represent the specializations $C(1)$ and $\left(\mu_{i_1}(C)\right)(1)$ as in the following diagram:
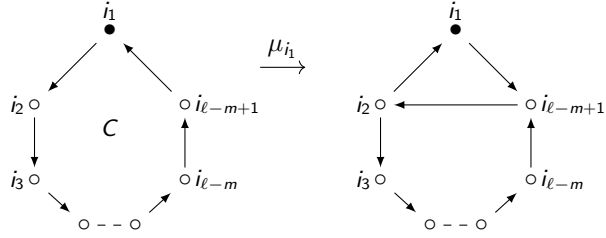
\[\begin{tikzpicture}[xscale=2,baseline=(bb.base),
  quivarrow/.style={black, -latex}] 
\path (0,0) node (bb) {}; 

\node (i21) at (0,0) {$\circ$};
\node [left] at (i21) {\small {$i_2$}}; 

\node (ilm11) at (1,0) {$\circ$};
\node [right] at (ilm11) {\small {$i_{\ell-m+1}$}};

\node (i11) at (0.5,1) {$\bullet$};
\node[above] at (i11) {\small {$i_1$}};

\node (i31) at (0,-1) {$\circ$};
\node [left] at (i31) {\small {$i_3$}};

\node (i41) at (0.33,-1.6) {$\circ$};

\node (ilm1) at (1,-1) {$\circ$};
\node [right] at (ilm1) {\small {$i_{\ell -m}$}};

\node (ilmm11) at (0.67,-1.6) {$\circ$};

\node (C) at (0.5,-0.3) {\small $C$};

\draw [quivarrow, shorten <=-1pt, shorten >=-1pt] (i11) -- (i21);
\draw [quivarrow, shorten <=-1pt, shorten >=-1pt] (i21) -- (i31);
\draw [quivarrow, shorten <=-1pt, shorten >=-1pt] (i31) -- (i41);
\draw [quivarrow, shorten <=-1pt, shorten >=-1pt] (ilmm11) -- (ilm1);
\draw [quivarrow, shorten <=-1pt, shorten >=-1pt] (ilm1) -- (ilm11);
\draw [quivarrow, shorten <=-1pt, shorten >=-1pt] (ilm11) -- (i11);
\draw [dashed, shorten <=-1pt, shorten >=-1pt] (i41) -- (ilmm11);

\draw [->] (1.5,0.5) to node[above] {$\mu_{i_1}$} (1.8,0.5);

\begin{scope}[shift={(2.3,0)}]

\node (i22) at (0,0) {$\circ$};
\node [left] at (i22) {\small {$i_2$}}; 

\node (ilm12) at (1,0) {$\circ$};
\node [right] at (ilm12) {\small {$i_{\ell-m+1}$}};

\node (i12) at (0.5,1) {$\bullet$};
\node[above] at (i12) {\small {$i_1$}};

\node (i32) at (0,-1) {$\circ$};
\node [left] at (i32) {\small {$i_3$}};

\node (i42) at (0.33,-1.6) {$\circ$};

\node (ilm2) at (1,-1) {$\circ$};
\node [right] at (ilm2) {\small {$i_{\ell -m}$}};

\node (ilmm12) at (0.67,-1.6) {$\circ$};


\draw [quivarrow, shorten <=-1pt, shorten >=-1pt] (i22) -- (i12);
\draw [quivarrow, shorten <=-1pt, shorten >=-1pt] (i22) -- (i32);
\draw [quivarrow, shorten <=-1pt, shorten >=-1pt] (i32) -- (i42);
\draw [quivarrow, shorten <=-1pt, shorten >=-1pt] (ilmm12) -- (ilm2);
\draw [quivarrow, shorten <=-1pt, shorten >=-1pt] (ilm2) -- (ilm12);
\draw [quivarrow, shorten <=-1pt, shorten >=-1pt] (i12) -- (ilm12);
\draw [quivarrow, shorten <=-1pt, shorten >=-1pt] (ilm12) -- (i22);
\draw [dashed, shorten <=-1pt, shorten >=-1pt] (i42) -- (ilmm12);

\end{scope}
\end{tikzpicture}
\]
By Lemma \ref{lem_non_spec_cycle}, we have that
    $\mu_{i_{\ell -m -1}}\ldots \mu_{i_1}(B^{(m+1)})$
is not pure.  Since $B^{(m+1)}=\mu_{k_{m+1}}B^{(m)}$ and $k_{m+1}=i_1$, we have 
    $$
    \mu_{i_{\ell -m -1}}\ldots \mu_{i_1}(B^{(m+1)})=\mu_{i_{\ell -m -1}}\ldots \mu_{i_2}(\mu_{i_1})^2(B^{(m)}).
    $$
But since $\mu_{i_1}^2$ just changes the sign of arrows incident to $i_1$ by 
Lemma \ref{lem:mutate_twice}, we also have that
$\mu_{i_{\ell -m -1}}\ldots \mu_{i_2}(B^{(m)})$
is not pure. 
But now   
\[    \mu_{i_{\ell -m -1}}\ldots \mu_{i_2}(B^{(m)})=\mu_{i_{\ell -m -1}}\ldots \mu_{i_2}\mu_{k_{m}}\ldots \mu_{k_1}(B)
\]
exhibits a mutation sequence of length $\ell -2$ from $B$ to a non-pure matrix, 
contradicting the minimality of $\ell$.

\textbf{Step 2:}    
Next we want to show that $B^{(m)}$ has a dangerous oriented $(\ell -m+2)$-cycle.  Notice that this is equivalent to showing that such a cycle exists for $\mu_{k_{m+1}}(B^{(m+1)})=\mu_{k_{m+1}}^2(B^{(m)})$, because applying $\mu_{k_{m+1}}$ twice has the only effect of changing the sign of the arrows incident with $k_{m+1}$ by Lemma \ref{lem:mutate_twice}.

To simplify notation, let 
\[ k=k_{m+1}, \;\;\;\; B^{(m+1)}=(b_{ij}), \;\;\;\;   \mu_{k}(B^{(m+1)})=(b_{ij}'), \;\;\text{ and }\;\;
a_h:=b_{i_h,k}
\]
for $1\leq h\leq \ell -m+1$.
We will show that exactly two $a_h$'s are nonzero, and they are consecutive, of the form $a_{j}$ and $a_{j+1}$.  We do this by ruling out all other possibilities.
Then we will show that
 \[ (i_1,\ldots,i_j,k_{m+1},i_{j+1},\ldots i_{\ell -m+1})\] 
is a dangerous oriented $(\ell -m+2)$-cycle in $B^{(m)}$.

\begin{itemize}
\item
   First suppose for contradiction that at most one of the $a_h$'s is nonzero.  Then, for $h\neq j$, we always have $b_{i_h,i_{k}}b_{i_k,i_j}=-a_ha_j=0$.
Consider the values of $b_{i_h,i_{h+1}}'$.  We showed in Step 1 that $\forall h$, $k\neq i_h$, so we are never in the first two cases of Definition \ref{defn_signed_mut}, and as $a_ha_j=0$ we are never in the third case.  So 

\begin{equation}\label{eq1}
b_{i_h,i_{h+1}}' = b_{i_h,i_{h+1}} \quad \text{for all } 1\leq h\leq \ell -m+1.
\end{equation}

By assumption, $B^{(m+1)}$ contains a dangerous oriented $(\ell -m+1)$-cycle $C=(i_1,\ldots, i_{\ell -m+1})$, and therefore $\mu_{k}(B^{(m+1)})=\mu_{k}^2(B^{(m)})$ contains the same dangerous cycle by \eqref{eq1}.  But then, by Lemma \ref{lem:mutate_twice}, $B^{(m)}$ contains a dangerous oriented $(\ell -m+1)$-cycle, so by Lemma \ref{lem_non_spec_cycle} the matrix
   $$
   \mu_{i_{\ell-m-1}}\ldots \mu_{i_1}(B^{(m)})=\mu_{i_{\ell-m-1}}\ldots \mu_{i_1}\mu_{k_m}\ldots\mu_{k_1}(B)
   $$
   is not pure.  But this  mutation sequence has length $\ell -1$, contradicting the minimality of $\ell$.
   Thus we can assume at least two of the $a_h$'s are nonzero.

\item
Now suppose for contradiction that there exist $h_1<h_2<h_3$ in $C$ such that $a_{h_1},a_{h_2},a_{h_3}\neq 0$ and $a_j=0$ for all $j$ such that $h_1<j<h_2$ and $h_2<j<h_3$.  Then at least one of the chordless cycles $(i_{h_1},\ldots, i_{h_2},k)$ or $(i_{h_2},\ldots, i_{h_3},k)$ is not oriented, as illustrated in the following picture: 
   $$
   
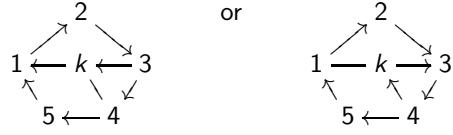
\xymatrix@C=0.1em@R=0.7em{ & & 2\ar[drr] & & \\
   1\ar[urr] & & k\ar[ll] & & 3\ar[dl]\ar[ll] \\
   & 5\ar[ul] & & 4\ar[ll]\ar@{-}[ul] & } \hspace{0.8cm} \text{or} \hspace{0.8cm}    \xymatrix@C=0.1em@R=0.7em{ & & 2\ar[drr] & & \\
   1\ar[urr] & & k\ar@{-}[ll]\ar[rr] & & 3\ar[dl] \\
   & 5\ar[ul] & & 4\ar[ll]\ar[ul] & }
   $$
   This contradicts Lemma \ref{lem_oriented_cycles_values_3_cycles}, and therefore cannot happen. 

\item
   The only case left to exclude is when there is some $h_1<h_2$ with $h_2-h_1>1$ such that $a_{h_1},a_{h_2}\neq 0$ and $a_j=0$ for all $j\neq h_1,h_2$. This case is similar to the one examined above: as shown in the figure below, at least one of the cycles $(i_{h_1},i_{h_1 +1},\ldots, i_{h_2},k)$ or $(i_{h_2},i_{h_2 +1},\ldots, i_{h_3},k)$ is non oriented in $B^{(m+1)}$, and therefore this case can never happen.
   $$
   
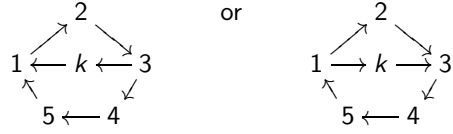
\xymatrix@C=0.1em@R=0.7em{ & & 2\ar[drr] & & \\
   1\ar[urr] & & k\ar[ll] & & 3\ar[dl]\ar[ll] \\
   & 5\ar[ul] & & 4\ar[ll] & } \hspace{0.8cm} \text{or} \hspace{0.8cm}    \xymatrix@C=0.1em@R=0.7em{ & & 2\ar[drr] & & \\
   1\ar[urr]\ar[rr] & & k\ar[rr] & & 3\ar[dl] \\
   & 5\ar[ul] & & 4\ar[ll] & }
   $$ 
\end{itemize}

So we have shown that $a_j,a_{j+1}\neq0$, and $a_h=0$ for $h\neq j,j+1$.

Therefore, by Lemma \ref{lem_oriented_cycles_values_3_cycles}(ii), we have that the valued quiver associated to the cycle $(i_j,i_{j+1},k)$ in $B^{(m+1)}(1)$ is one of the following:
   $$
   
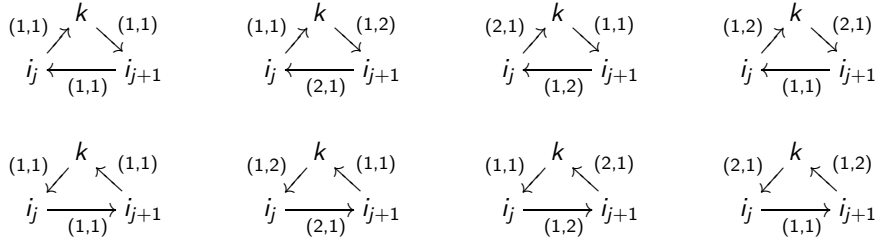
\xymatrix@=0.7em{& k\ar[dr]^{(1,1)} & \\
   i_{j}\ar[ur]^{(1,1)} & & i_{j+1}\ar[ll]^{(1,1)}} \hspace{0.9cm} \xymatrix@=0.7em{& k\ar[dr]^{(1,2)} & \\
   i_{j}\ar[ur]^{(1,1)} & & i_{j+1}\ar[ll]^{(2,1)}} \hspace{0.9cm} \xymatrix@=0.7em{& k\ar[dr]^{(1,1)} & \\
   i_{j}\ar[ur]^{(2,1)} & & i_{j+1}\ar[ll]^{(1,2)}} \hspace{0.9cm} \xymatrix@=0.7em{& k\ar[dr]^{(2,1)} & \\
   i_{j}\ar[ur]^{(1,2)} & & i_{j+1}\ar[ll]^{(1,1)}}
   $$
   $$
   \xymatrix@=0.7em{& k\ar[dl]_{(1,1)} & \\
   i_{j}\ar[rr]_{(1,1)} & & i_{j+1}\ar[ul]_{(1,1)}} \hspace{0.9cm} \xymatrix@=0.7em{& k\ar[dl]_{(1,2)} & \\
   i_{j}\ar[rr]_{(2,1)} & & i_{j+1}\ar[ul]_{(1,1)}} \hspace{0.9cm} \xymatrix@=0.7em{& k\ar[dl]_{(1,1)} & \\
   i_{j}\ar[rr]_{(1,2)} & & i_{j+1}\ar[ul]_{(2,1)}} \hspace{0.9cm} \xymatrix@=0.7em{& k\ar[dl]_{(2,1)} & \\
   i_{j}\ar[rr]_{(1,1)} & & i_{j+1}\ar[ul]_{(1,2)}}
   $$

   We can assume that $b_{i_j,i_{j+1}}b_{i_j,k}b_{k,i_{j+1}}\in \mathbb{Z}$, since otherwise we could use Lemma \ref{lem_non_spec_cycle} to contradict the minimality of $\ell$. Furthermore, the two rightmost cycles $(i_j,i_{j+1},k)$ in the picture cannot appear in $B^{(m+1)}$, since otherwise $b_{i_j,i_{j+1}}'=t b_{i_j,i_{j+1}}$, and the chordless cycle $(i_1,\ldots, i_{\ell-m+1})$ in $\mu_k(B^{(m+1)})$ would not be oriented, contradicting Lemma \ref{lem_oriented_cycles_values_3_cycles}(i). The other valued quivers yield $b_{i_j,i_{j+1}}'=0$, and therefore $(i_1,\ldots, i_j,k,i_{j+1},\ldots, i_{\ell -m+1})$ is a chordless $(\ell -m+2)$-cycle in $\mu_k(B^{(m+1)})$, and thus in $B^{(m)}$. 
   
Now we have found our cycle in $B^{(m)}$.  It just remains to show that the product of its signs is $(-1)^{\ell-m}$.

First note that, by the usual minimality argument, the 3-cycle $(i_j,k,i_{j+1})$ in $B^{(m+1)}$ must be oriented.  So we have $b_{i_j,k}b_{k,i_{j+1}}b_{i_{j+1},i_j}\in\Z$.
By assumption, we have
\[\prod_{j=1}^{\ell-m+1} b_{i_j,i_{j+1}}\in t^{\ell-m+1}\mathbb{Z}\]
so together we get
\[  (\prod_{h=1}^{j-1}b_{i_h,i_{h+1}})b_{i_j,k}b_{k,i_{j+1}}(\prod_{h=j+1}^{\ell-m-1}b_{i_h,i_{h+1}})\in t^{\ell-m+1}\mathbb{Z}.
\]

Second, as $a_h=0$ for $h\neq j,j+1$ we have $b_{i_h,i_{h+1}}'=b_{i_h,i_{h+1}}$ when $h\neq j$, and Definition \ref{defn_signed_mut} gives $b_{i_j,k}'=-tb_{i_j,k}$ and $b_{k,i_{j+1}}'=-b_{k,i_{j+1}}$.

Therefore
\[   (\prod_{h=1}^{j-1}b_{i_h,i_{h+1}}')b_{i_j,k}'b_{k,i_{j+1}}'(\prod_{h=j+1}^{\ell-m-1}b_{i_h,i_{h+1}}')= t(\prod_{h=1}^{j-1}b_{i_h,i_{h+1}})b_{i_j,k}b_{k,i_{j+1}}(\prod_{h=j+1}^{\ell-m-1}b_{i_h,i_{h+1}})\in t^{\ell-m+1}\mathbb{Z}
\]
so our chordless $(\ell-m+2)$-cycle in $B^{(m)}$ is dangerous and our downwards induction is complete.
\end{proof}

\begin{theorem}\label{thm:mutD-is-pure}
Every mutation Dynkin gss matrix is pure.
\end{theorem}
\begin{proof}
By definition, a mutation Dynkin gss matrix is mutation equivalent to a pure gss matrix $B$ such that $B(1)$ is Dynkin.  As every Dynkin diagram is a tree, $B$ has no dangerous cycles, so the result follows by Proposition \ref{prop:mutat-eq-Dynkin-no-dang}.
\end{proof}

\begin{theorem}\label{thm:mutD-fromspec}
The mutation class of Dynkin type $\Delta$ consists of all gss matrices $B$ such that $B(1)$ is of (unsigned) mutation type $\Delta$ and $B$ contains no dangerous cycles.
\end{theorem}
\begin{proof}
If $B$ is mutation Dynkin then $B(1)$ must be of (unsigned) mutation type $\Delta$ by Proposition \ref{prop_mutat_agree}.  Then $B$ cannot contain a dangerous cycle, as the cycle would be oriented by Lemma \ref{lem_oriented_cycles_values_3_cycles}, and Lemma \ref{lem_non_spec_cycle} would say that $B$ is mutation equivalent to a non-pure matrix, contradicting Theorem \ref{thm:mutD-is-pure}.

Now suppose that $B$ is a gss matrix with no dangerous cycles and that $\tilde B=B(1)$ is mutation Dynkin.  
Then there exists a sequence $i_1,\ldots,i_n$ such that $\tilde B'=\mu^{FZ}_{i_k}\ldots\mu^{FZ}_{i_1}\tilde B$ is Dynkin.  
Define the gss matrix $B'=\mu_{i_k}\ldots\mu_{i_1} B$, so $B'(1)=\tilde B'$.  Then $B'$ is pure by Proposition \ref{prop:mutat-eq-Dynkin-no-dang}, so it is mutation Dynkin, and therefore so is $\mu_{i_1}^3\ldots\mu_{i_k}^3B'$.  But, by Proposition \ref{prop_order4}, this is $B$.
\end{proof}


\section{Roots}

\subsection{The Cartan counterpart}

Given a skew-symmetrizable matrix $\tilde B$, its \emph{Cartan counterpart} is the $n\times n$ matrix $\tilde C(\tilde{B})=(\tilde c_{ij})$ with $\tilde c_{ii}=2$ and $\tilde c_{ij}=-
|\tilde b_{ij}|$: 
see \cite[Section 1.3]{fz2}, following \cite[Remark 4.6]{fz1}.  

We revise some definitions from the introduction of \cite{bgz}.
Let $\tilde C$ denote a symmetrizable matrix (so $D\tilde C=(D\tilde C)^T$ for $D$ diagonal) and $\tilde B$ a skew-symmetrizable matrix.  We say $\tilde C$ is:
\begin{itemize}
\item \emph{quasi-Cartan} if its diagonal entries $\tilde c_{ii}$ are all equal to $2$;
\item \emph{positive} if the symmetrized matrix $D\tilde C$ is positive definite;
\item a \emph{quasi-Cartan companion} of $\tilde B$ if it is quasi-Cartan and $\abs{\tilde c_{ij}}=|\tilde b_{ij}|$ 
 for all $i\neq j$.
\end{itemize}
The Cartan counterpart of a skew-symmetrizable matrix $\tilde B$ 
is an example of a quasi-Cartan companion of $\tilde B$, but it will not be positive in general, even if $\tilde B$ is mutation Dynkin: for example, if we start with a linearly oriented $A_3$ quiver and mutate to get a 3-cycle, the Cartan counterpart has determinant $0$ and so cannot be positive.

In this section we give a signed variant of the Cartan counterpart.  
\begin{definition}\label{def:cc}
Given a gss matrix $B=(b_{ij})$, its \textbf{Cartan counterpart} is the $n\times n$ integer matrix $C(B)=(c_{ij})$ where:
\[
c_{ij}=\begin{cases} 2 & \text{if } i=j; \\
\abs a- \abs b & \text{if } i\neq j \text{ and }b_{ij}=a+bt.
\end{cases}
\]
\end{definition}
For $C=C(B)$, note that $DB=-(BD)^T$ implies $DC=(CD)^T$.

The following two results are immediate from the definitions.
\begin{lemma}\label{lem:qCartan}
If $B$ is a pure gss matrix then $C(B)$ is a quasi-Cartan companion of $B(1)$.
\end{lemma}
\begin{proposition}\label{prop:cc-btilde}
If $\tilde B$ is a skew-symmetrizable matrix then 
$\tilde C(\tilde B)=C(t\tilde B)$.
\end{proposition}
\begin{remark}\label{rmk:embed}
We interpret Proposition \ref{prop:cc-btilde} as saying that the canonical way to embed valued quivers into signed valued quivers is to give all arrows negative sign.
\end{remark}

If $(Q,v)$ is a signed valued quiver, write $C(Q,v)=C(B(Q,v))$.
The following result is immediate from Definitions \ref{def:associated} and \ref{def:cc}.
\begin{lemma}\label{lem:CQv}
If $(Q,v)$ is a signed valued quiver then $C(Q,v)=(c_{ij})$ where:
\[ c_{ij} =
\begin{cases}
2 & \text{if } i=j, \\
v_1(\alpha) &\text{if }\exists\alpha:i\to j;\\
v_2(\alpha) &\text{if }\exists\alpha:j\to i;\\
0 &\text{if there is no arrow between $i$ and $j$.}
\end{cases}
\]
\end{lemma}

\begin{example}
Consider the following signed valued quiver of mutation type $A_3$:
$$
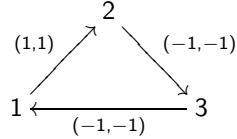
\xymatrix{ & 2\ar[dr]^{(-1,-1)} & \\
1\ar[ur]^{(1,1)} & & 3\ar[ll]^{(-1,-1)}}
$$

Its associated gss matrix and Cartan counterpart are as follows. 
$$
B=\left(\begin{matrix}
    0 & 1 & -t \\
    -1 & 0 & t \\
    t & -t & 0
\end{matrix}\right) \hspace{1cm} C(B)=\left(\begin{matrix}
    2 & 1 & -1 \\
    1 & 2 & -1 \\
    -1 & -1 & 2
\end{matrix}\right)
$$
\end{example}

We give a signed variant of Definition \ref{def:unsigned-diagram}.
\begin{definition}\label{def:signed-diagram}
The \textbf{signed diagram} $\Gamma=\Gamma(Q,v)$ of a signed valued quiver $(Q,v)$ is the unsigned diagram $\tilde\Gamma(Q,v)$ together with a sign $\e_{ij}=\e_{ji}$ attached to each unordered pair $\{i,j\}$ of vertices which are joined by at least one edge.  If $Q$ has an arrow $\alpha:i\to j$ then the sign of $\{i,j\}$ is $\e_{ij}=\sgn(v_1(\alpha))$.
\end{definition}
For a signed valued quiver of mutation type Dynkin, the next result shows that its Cartan counterpart and the associated signed diagram capture the same information.  Note that a similar definition of the diagram of a quasi-Cartan matrix appears in \cite[Definition 2.2]{bgz}, without recording the ordering on multiple edges.

\begin{lemma}
Suppose $B$ is a mutation Dynkin gss matrix with associated signed valued quiver $(Q,v)$.  Let $\Gamma=\Gamma(Q,v)$ and $C=C(Q,v)$.  Then, for $i\neq j$, the number of edges between $i$ and $j$ in $\Gamma$ is $c_{ij}c_{ji}$, with $i>j$ if and only if $\abs{c_{ij}}>\abs{c_{ji}}$.  Moreover, if $\e_{ij}$ denotes the sign of an edge between $i$ and $j$ in $\Gamma$, then
\[ c_{ij}=\begin{cases}
0 &\text{if there are no edges between $i$ and $j$};\\
\e_{ij} &\text{if there is one edge between $i$ and $j$};\\
\e_{ij} &\text{if there are $d>1$ edges between $i$ and $j$, and $i<j$ in $\Gamma$};\\
\e_{ij}d &\text{if there are $d>1$ edges between $i$ and $j$, and $i>j$ in $\Gamma$}.
\end{cases}\]
\end{lemma}
\begin{proof}
This is a case-by-case check according to whether we have arrows $\alpha:i\to j$, $\alpha:j\to i$, or no such arrows. 

The idea is to use Lemma \ref{lem:CQv} together with the fact that the (unsigned) valued quiver associated to $B(1)$ is ``2-finite'' (see the comment after Theorem 8.6 in \cite{fz2}): by Proposition \ref{prop_mutat_agree} any arrow in $(Q,v)$ has weight $\leq 3$. This property implies that, for $i,j$ vertices of $\Gamma(Q,v)$, the data given in Definition \ref{def:signed-diagram} uniquely determines the $(i,j)$-th and $(j,i)$-th entries of $C(Q,v)$, and viceversa.
\end{proof}

We now describe how Cartan counterparts change under mutation.
\begin{lemma}\label{lem:mutateCartan}
Let $(Q,v)$ be mutation Dynkin and let $(Q',v')=\mu_k(Q,v)$.  
Let $C=(c_{ij})$ and $C'=(c'_{ij})$ be their respective Cartan counterparts.  
Then:
    \[c'_{ij}
=\begin{cases}
    c_{ij} &\text{ if neither or both of $i,j\to k$ in $Q$};\\
    c_{ij}-c_{ik}c_{kj} &\text{ if exactly one of $i,j\to k$ in $Q$}.
\end{cases}
    \]
\end{lemma}

\begin{proof}
Write $B=B(Q,v)$ and $B'=B(Q',v')$.  
By Definition \ref{defn_signed_mut} of signed mutation, if neither or both $i,j\to k$, then $b_{ij}'=b_{ij}$ if $i,j\neq k$ or $b_{ij}'=-b_{ij}$ if $i\rra{} k=j$ or $k=i\arr{} j$, and therefore in this case the statement follows trivially.

If $i\to k=j$, then $b_{ik}'=-tb_{ik}$, so $c'_{ij}=-c_{ij}=c_{ik}-c_{ik}c_{kj}$ as $c_{kj}=c_{kk}=2$.  A similar argument works for $j\to k=i$.

The only case left to examine is if there is exactly one arrow $i,j\to k$ and $k\neq i,j$. Without loss of generality, suppose $i\to k$. 
\begin{itemize}
\item If $b_{kj}=0$, then we have $b_{ij}'=b_{ij}$. Therefore:
$$
c_{ij}'=c_{ij}-c_{ik}c_{kj},
$$
since $c_{kj}=0$.
\item If $b_{kj}\neq 0$, then we have $i\arr{} k\arr{} j$. The signed mutation rule yields $b_{ij}'=t(b_{ij}+b_{ik}b_{kj})$. 
\begin{itemize}
\item If $b_{ij}=0$ then $c_{ij}=0$ and $b_{ij}'=tb_{ik}b_{kj}$.  So, as $B$ is pure, one checks that in each case we have $c_{ij}'=-c_{ik}c_{kj}=c_{ij}-c_{ik}c_{kj}.$
\item If $b_{ij}\neq 0$ then, by Lemma \ref{lem_oriented_cycles_values_3_cycles}, $\sgn (b_{ij})=-1$ since all 3-cycles in $B$ are oriented.  
Now we use that $B$ is pure and satisfies the positive 3-cycle condition at $k$.
If $b_{ij},b_{ik}b_{kj}\in\mathbb{Z}$ then $c_{ij}=-b_{ij}$ and $c_{ik}c_{kj}=b_{ik}b_{kj}$.  If $(i,j,k)$ is a 3-cycle in $B'$, then it is oriented by Lemma \ref{lem_oriented_cycles_values_3_cycles}. Therefore $b_{ij}'=t(b_{ij}+b_{ik}b_{kj})$ implies $b_{ij}+b_{ik}b_{kj}\geq 0$.  So 
\[ c_{ij}'=-|b_{ij}+b_{ik}b_{kj}|=-(b_{ij}+b_{ik}b_{kj})=c_{ij}-c_{ik}c_{kj}.\] 
The case $b_{ij},b_{ik}b_{kj}\in t\mathbb{Z}$ can be treated similarly.
\end{itemize}
\end{itemize}

Therefore we get the statement.
\end{proof}

\subsection{Root systems}

Given an $n\times n$ gss matrix $B$, let $S_B=\{\alpha_1^B, \ldots, \alpha_n^B\}$, where $\alpha_i^{B}$ are formal simple roots and
let $V_B$ be the real vector space with basis $S_B$.  
We sometimes drop the super/subscript $B$ when it is clear from context.

\begin{definition}\label{def:reflections}
Let $C=C(B)=(c_{ij})$ be the Cartan counterpart of $B$.  Define linear functions $s_i=s_{\alpha_i^B}:V_B\to V_B$ called \textbf{simple reflections} by letting
\[ s_i(\alpha_j)=\alpha_j - c_{ij}\alpha_i \]
and extending linearly to $V_B$.  Let $W_B$ denote the subgroup of $\GL(V_B)$ generated by $s_i$, $1\leq i\leq n$.
\end{definition}
\begin{definition}\label{def:groot}
Suppose $B$ is mutation Dynkin.  The \textbf{root system} of $B$ is the orbit $\Phi_B=W_B\cdot S_B$ of $S_B$ under $W_B$.
\end{definition}
Note that if $B$ is Dynkin then this is exactly the classical root system, by Theorem 2 of \cite[Section V.10]{serre}.  In particular, if $B$ is Dynkin then $W_B$ is finite.  We will show that the name root system is justified in the mutation Dynkin case in Theorem \ref{thm:root-sys-iso} below.

Recall that we have a skew-symmetrizing matrix $D$ such that $DB$ is skew-symmetric. If $C=C(B)$ then $DC$ is symmetric, and so $DC=(DC)^T=C^TD$.  

\begin{definition}\label{defn_inner_product}
Let $M=DC$ and define an inner product (i.e., a symmetric, positive-definite bilinear form) on $V_B$ by $(u,\, v)=u^TMv$.
\end{definition}

For $\beta\in \Phi_B$, define $\beta^{\vee}=2\beta/(\beta,\beta)$. Then \[ (\alpha_i,\alpha_j^\vee)=\frac{2(\alpha_i,\alpha_j)}{(\alpha_j,\alpha_j)}
=\frac{2m_{ij}}{m_{jj}}=\frac{2c_{ji}d_j}{c_{jj}d_j}=c_{ji}\]
and the length of $\alpha_i$ is $|\alpha_i|=(\alpha_i,\alpha_i)=m_{ii}=2d_i$.

\begin{proposition}
The inner product $(-,-)$ is $W_B$-invariant.
\end{proposition}
\begin{proof}
Since $M=C^T D=DC$, then $m_{ij}=c_{ji}d_j=d_ic_{ij}$ for $i,j=1,\ldots, n$. Using this and $c_{ii}=2$, we get:
\begin{align*}
    (s_i(\alpha_j),\,s_i(\alpha_k)) &= (\alpha_j-c_{ij}\alpha_i,\,\alpha_k-c_{ik}\alpha_i) \\
    &= m_{jk}-c_{ik}m_{ji}-c_{ij}m_{ik}+c_{ij}c_{ik}m_{ii} \\
    &= m_{jk}-c_{ik}c_{ij}d_i-c_{ij}c_{ki}d_k+2c_{ij}c_{ik}d_i \\
    &= m_{jk}+c_{ij}c_{ik}d_i-c_{ij}c_{ki}d_k \\
    &= m_{jk}+c_{ij}(c_{ik}d_i-c_{ki}d_k) \\
    &= m_{jk}\\
    &=(\alpha_j,\,\alpha_k)
\end{align*}
\end{proof}

\begin{remark}\label{rmk:diff_conv}
Note that our convention in Definition \ref{def:reflections} differs from some other texts, such as \cite{hum_la}, where $s_i(\alpha_j)$ is defined using $c_{ji}$ and not $c_{ij}$.  Later, in Definition \ref{def:r1234}, we will define a Lie algebra relation (R4) which specialises to $(\ad e_i)^{1-c_{ij}}(e_j)=0$ for a classical Cartan matrix: we use this convention in order to match certain other results in the literature \cite{bkl,br,pr}.  Then, if we want the root space decomposition of our Lie algebra to be compatible with Definition \ref{def:groot}, this forces us to use the conventions in Definition \ref{def:reflections}.

For convenience, the following table collects in one place the construction of classical Lie-theoretic objects from the Cartan matrix $C$ when using our conventions.
\begin{align*}
&C \leftrightarrow \g& & [h_i,e_j]=c_{ij}e_j\text{ and }(\ad e_i)^{1-c_{ij}}(e_j)=0\\
&C \leftrightarrow \Gamma& & \abs{c_{ij}}>\abs{c_{ji}}\iff i>j\text{ in }\Gamma\\ 
&C \leftrightarrow (-,-)& & c_{ij}=(\alpha_j,\,\alpha_i^\vee)\\
&C \leftrightarrow W&& s_i(\alpha_j)=\alpha_j - c_{ij}\alpha_i
\end{align*}
\end{remark}

The following definition was given in \cite[Definition 4.1]{par} for the simply-laced case, and extended to the skew-symmetrizable case in \cite[Section 6]{bm}.  Let $\Phi_\Delta$ be a root system of Dynkin type $\Delta$ with rank $n$ and let $\tilde B$ be a skew-symmetrizable matrix of mutation class $\Delta$.  A subset $\{\gamma_i\}_{i=1}^n$ of $\Phi_\Delta$ is called a \emph{companion basis for $\tilde B$} if:
\begin{enumerate}
\item[(i)] $\{\gamma_i\}_{i=1}^n$ is a $\Z$-basis of $\Z\Phi_\Delta$, and
\item[(ii)] the matrix $A$ with $a_{ij}=(\gamma_j,\gamma_i^\vee)$ is a positive quasi-Cartan companion of $\tilde B$.
\end{enumerate}
Note that \cite{bm} asks for $a_{ij}=(\gamma_i,\gamma_j^\vee)$ in the above definition, but we have adapted the condition to match our convention as outlined in Remark \ref{rmk:diff_conv}.

\begin{definition}
Let $B$ be a mutation Dynkin gss matrix of Dynkin type $\Delta$.
We say that $\{\gamma_i\}_{i=1}^n\subset\Phi_\Delta$ is a \textbf{signed companion basis for $B$} if:
\begin{enumerate}
\item[(i)] $\{\gamma_i\}_{i=1}^n$ is a $\Z$-basis of $\Z\Phi_\Delta$, and
\item[(ii)] the matrix $C$ with $c_{ij}=(\gamma_j,\gamma_i^\vee)$ is equal to $C(B)$.
\end{enumerate}
\end{definition}
\begin{lemma}\label{lem:signedCBisCB}
If $\{\gamma_i\}_{i=1}^n\subset\Phi_\Delta$ is a signed companion basis for $B$, then it is a companion basis for $B(1)$.
\end{lemma}
\begin{proof}
By Theorem \ref{thm:mutD-is-pure} and Lemma \ref{lem:qCartan}, $C$ is a quasi-Cartan companion of $B(1)$.  As $\{\gamma_i\}_{i=1}^n\subset\Phi_\Delta$ is a $\Z$-basis for $\Z\Phi$, it is related to the Cartan matrix of $\Delta$ by a change of basis, so positivity of the Cartan matrix implies positivity of $C$.
\end{proof}

The following lemma is useful.  For a proof, see for example \cite[III.9.2]{hum_la}.
\begin{lemma}\label{lem:conj-refl}
If $f:V_B\to V_B$ is an isomorphism such that $f(\Phi)=\Phi$ then $s_{f(\delta)}f=fs_\delta$ for all $\delta\in\Phi$.
\end{lemma}

Now let $B=B(Q,v)$, $B'=\mu_kB$, $\Phi=\Phi_B=\{\alpha_i\}$, and $\Phi'=\Phi_{B'}=\{\alpha'_i\}$.
The following formulas send our base roots to the inward mutation of roots as defined by Parsons \cite[Theorem 6.1]{par}.  Note that $\rho_k(\Phi')\in \Phi$ follows by Lemma \ref{lem:conj-refl} and Definition \ref{def:groot}.
\begin{definition}\label{def:mutate-roots}
The \textbf{mutation of roots} at $k$ is the function $\rho_k:\Phi'\to \Phi$ defined by extending the following rule linearly:
\[
\rho_k: \;\; \alpha'_i \mapsto 
\begin{cases}
s_k(\alpha_i) & \text{ if } i\to k \text{ in } Q;\\
\alpha_i & \text{ otherwise.}
\end{cases}
\]
\end{definition}
Notice that the direction of the morphism $\Phi'\to \Phi$ is opposite to what one might at first expect.

The formula in the next result comes from the outward mutation of roots \cite[Theorem 6.1]{par}.
\begin{lemma}\label{def:mutate-roots-inverse}
Mutation of roots is a bijection with inverse $\rho_k^{-1}:\Phi\to\Phi'$ given by
\[
\rho_k^{-1}: \;\; \alpha_j \mapsto 
\begin{cases}
s'_k(\alpha_j') & \text{ if } k\to j \text{ in } Q';\\
\alpha'_j & \text{ otherwise.}
\end{cases}
\]
\end{lemma}
\begin{proof}
We check that the maps from Definitions \ref{def:mutate-roots} and \ref{def:mutate-roots-inverse} are inverse.  Note that we have an arrow $i\to k$ in $Q$ if and only if we have an arrow $k\to i$ in $Q'$.  
We just check the composition $\Phi'\to\Phi\to\Phi'$; the other composition is similar.  If we do not have an arrow $i\to k$ in $Q$, then $\alpha_i'\mapsto \alpha_i\mapsto\alpha_i'$.  If we do have $i\to k$ in $Q$, then $\alpha_i'\mapsto s_k(\alpha_i)=\alpha_i-c_{ki}\alpha_k\mapsto (\alpha_i'-c'_{ki}\alpha'_k)-c_{ki}\alpha'_k$.  But $c'_{ki}=-c_{ki}$ by Lemma \ref{lem:mutateCartan}, so this is $\alpha_i'$.
\end{proof}

\begin{lemma}\label{lem:mutateroots}Write $\gamma_i=\alpha'_i$.  Then:
    \[ (\rho_k(\gamma_i), \, \rho_k(\gamma_j)^\vee)
=\begin{cases}
    (\alpha_i,\,\alpha_j^\vee) &\text{ if neither or both of $i,j\to k$ in $Q_B$};\\
    (\alpha_i,\,\alpha_j^\vee)-(\alpha_i,\,\alpha_k^\vee)(\alpha_k,\,\alpha_j^\vee) &\text{ if exactly one of $i,j\to k$ in $Q_B$}.
\end{cases}
    \]
\end{lemma}
\begin{proof}
The check for the case where neither of $i,j$ has an arrow to $k$ is trivial.
For the remaining checks, we first note that the $W_{B'}$-invariance of $(-,-)$ implies $W_{B'}$-invariance of $(-,(-)^\vee)$.  

If both $i,j\to k$ then 
$ (\rho_k(\gamma_i), \, \rho_k(\gamma_j)^\vee)=(s_{\alpha_k}\alpha_i,\, (s_{\alpha_k}\alpha_j)^\vee) =(\alpha_i,\,\alpha_j^\vee)$ as required.

The check for the case where only $i\to k$ is an easy computation, using that $(-,(-)^\vee)$ is linear in the first component:
\[ (\rho_k(\gamma_i), \, \rho_k(\gamma_j)^\vee)
= (s_k\alpha_i,\,\alpha_j^\vee) 
= (\alpha_i-c_{ki}\alpha_k,\,\alpha_j^\vee) 
= (\alpha_i,\,\alpha_j^\vee) - (\alpha_i,\,\alpha_k^\vee)(\alpha_k,\,\alpha_j^\vee).
\]
The check for $j\to k$ follows from the previous check and invariance under reflections:
\[ (\rho_k(\gamma_i), \, \rho_k(\gamma_j)^\vee)
= (\alpha_i,\,(s_k\alpha_j)^\vee)
= (s_k\alpha_i,\,\alpha_j^\vee)
= (\alpha_i,\,\alpha_j^\vee) - (\alpha_i,\,\alpha_k^\vee)(\alpha_k,\,\alpha_j^\vee).
\]
\end{proof}

Recall the definition of a (reduced) root system $R$ in a real vector space $V$ with inner product (see \cite[V.2]{serre} or \cite[III.9.2]{hum_la}):
\begin{enumerate}
\item $R$ is finite, spans $V$, and $0\notin R$;
\item for all $\alpha\in R$ we have $s_\alpha(R)=R$;
\item for all $\alpha,\beta\in R$ we have $(\alpha,\beta^\vee)\in\Z$;
\item for all $\alpha\in R$ we have $R\cap \R\alpha=\{\pm\alpha\}$.
\end{enumerate}
A vector space isomorphism $f:V\to V'$ is an isomorphism $R\arr\sim R'$ of root systems if $f(R)=R'$ and for all $\alpha,\beta\in\R$ we have $(f(\alpha),f(\beta)^\vee)_{V'}=(\alpha,\beta^\vee)_{V}$.
\begin{theorem}\label{thm:root-sys-iso}
If $B$ is mutation Dynkin then $\Phi_B$ is a root system, and $\rho_k:\Phi_{B'}'\to\Phi_B$ extends to an isomorphism of root systems.
\end{theorem}
\begin{proof}
By induction, we just need to check that $\Phi_B$ is a root system when $B'=\mu_k B$ is Dynkin.  Property 1 holds by Lemma \ref{def:mutate-roots-inverse}, property 2 holds by Lemma \ref{lem:conj-refl}, property 3 holds by Lemma \ref{lem:mutateroots}, and property 4 holds by linearity of $\rho_k$.  Then Lemmas \ref{lem:mutateCartan} and \ref{lem:mutateroots} show that $\rho_k$ preserves the inner products, and so is an isomorphism of root systems.
\end{proof}

Recall $S_B=\{\alpha_1,\ldots, \alpha_n\}$ is a basis for $V_B$.

\begin{theorem}\label{thm:signedCB}
Suppose $B$ is mutation Dynkin, so $B=\mu_{k_\ell}\ldots \mu_{k_1} B_\Delta$ where $\tilde B_\Delta$ is a skew-symmetric matrix of Dynkin type and $B_\Delta=t\tilde B_\Delta$.  Write $\Phi=\Phi_B$ and $\Phi_\Delta=\Phi_{B_\Delta}$.
Let $\rho=\rho_{k_1}\ldots\rho_{k_\ell}:\Phi\to \Phi_{\Delta}$.  Then $\rho(S_B)$ is a signed companion basis for $B$.
\end{theorem}
\begin{proof}
We use induction on $\ell$.  The base case $\ell=0$ is clear.  So suppose the statement is true for mutation sequences of length $\ell$ and let $B'=\mu_{k_{\ell+1}}B$.  The mutation in Definition \ref{def:mutate-roots} clearly preserves $\Z$-bases of $\Z\Phi_\Delta$.  By induction we know that $C(B)_{ij}=(\rho(\alpha_i),\rho(\alpha_j)^\vee)$, then comparing Lemmas \ref{lem:mutateCartan} and \ref{lem:mutateroots} shows that $C(B')_{ij}=(\rho_{k_{\ell+1}}\rho(\alpha_i),\rho_{k_{\ell+1}}\rho(\alpha_j)^\vee)$, so we are done.
\end{proof}

\begin{corollary}\label{cor:pos-qCcomp}
If $(Q,v)$ is mutation Dynkin then $C(Q,v)$ is a positive quasi-Cartan companion of $B(1)$.
\end{corollary}
\begin{proof}
Combine Theorem \ref{thm:signedCB} and Lemma \ref{lem:signedCBisCB}.
\end{proof}
Given Corollary \ref{cor:pos-qCcomp}, we could apply results of \cite{br,pr} to get presentations of simple Lie algebras from mutation Dynkin signed valued quivers.  We want to understand homomorphisms between these presentations, so we conduct a more in-depth analysis in Section \ref{s:Lie}.

\subsection{Mutation type $A$}

In this subsection we give a recursive description of the root system associated to a gss matrix of mutation Dynkin type $A$.

Recall the classification of (unsigned) quivers of mutation type $A$: see \cite[Section 2]{bv}.  Roughly, all such quivers look like the following diagram:
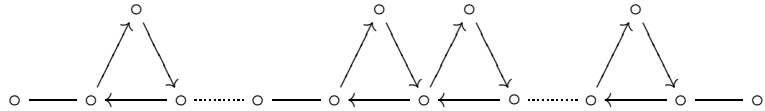
\[\xymatrix@C=0.6em{& & & \circ\ar[dr] & & & & & & \circ\ar[dr] & & \circ\ar[dr] & & & & \circ\ar[dr] & & &  \\
\circ\ar@{-}[rr] & & \circ\ar[ur]  & & \circ \ar[ll]  \ar@{.}[rr] & & \circ\ar@{-}[rr] & & \circ\ar[ur]  & & \circ \ar[ll] \ar[ur]  & & \circ \ar[ll]\ar@{.}[rr] & & \circ\ar[ur]  & & \circ \ar[ll]\ar@{-}[rr] & & \circ}\]

Let $\mathcal{M}_{n}^A$ denote the set of signed valued quivers of mutation type $A_n$ and let $(Q,v)\in\mathcal{M}_{n}^A$.  By Theorem \ref{thm:mutD-fromspec}, we know that $(Q,v)$ is one of the shapes given in the following diagram, with signs of certain arrows labelled in red by $\gamma, \delta,\e\in\{\pm1\}$.
\[
    \begin{tikzpicture}[xscale=2,baseline=(bb.base),
  quivarrow/.style={black, -latex}] 
  
    \node(a1) at (-2,0) {$(a)$};
  
    \node (a) at (0,0) {$\circ$};
    \node [above] at (a) {\scriptsize {$n-1$}};
    
    \node (b) at (0.5,0) {$\circ$};
    \node [above] at (b) {\scriptsize {$n$}};

    \node (p1) at (0.1,0) {};
    \node (p2) at (-0.7,0.3) {};
    \node (p3) at (-0.7,-0.3) {};

    \node (Q) at (-1.2,0) {$Q=$};
    
    \node (MA) at (-0.5,0) {$\mathcal{M}_{n-1}^A$};
    \begin{pgfonlayer}{background}
      \path[expand bubble]plot [smooth cycle,tension=1] coordinates {(p1) (p2) (p3)};
    \end{pgfonlayer}

    \draw [shorten <=-1pt, shorten >=-1pt] (a) -- (b) node[midway, above]{\scriptsize \color{red}{$\varepsilon$}};

    \begin{scope}[yshift=-2cm, xshift=0.3cm]
        \node (b1) at (-2.3,0){$(b)$};
  
    \node (a) at (0,0) {$\circ$};
    \node [left] at (a) {\scriptsize {$n-2$}};
    
    \node (c) at (0.25,0.5) {$\circ$};
    \node [above] at (c) {\scriptsize {$n-1$}};
    
    \node (b) at (0.5,0) {$\circ$};
    \node [right] at (b) {\scriptsize {$n$}};

    \node (p1) at (0.1,0) {};
    \node (p2) at (-0.9,0.4) {};
    \node (p3) at (-0.9,-0.4) {};

    \node (Q) at (-1.5,0) {$Q=$};
    
    \node (MA) at (-0.7,0) {$\mathcal{M}_{n-2}^A$};
    \begin{pgfonlayer}{background}
      \path[expand bubble]plot [smooth cycle,tension=1] coordinates {(p1) (p2) (p3)};
    \end{pgfonlayer}

    \draw [shorten <=-2pt, shorten >=-2pt] (a) -- (c) node[midway, above left]{\scriptsize \color{red}{$\gamma$}};
    \draw [shorten <=-2pt, shorten >=-2pt] (c) -- (b) node[midway, above right]{\scriptsize \color{red}{$\varepsilon$}};
    \draw [shorten <=-1pt, shorten >=-1pt] (a) -- (b) node[midway, below]{\scriptsize \color{red}{$\delta$}};
    \end{scope}
    \end{tikzpicture}
%
\]
By Theorem \ref{thm:mutD-fromspec}, the product of the signs of the arrows in the 3-cycle in (b) is $\gamma\delta\varepsilon=1$.

Let $\Phi=\Phi{(Q,v)}$ with simple roots $\{\alpha_1,\ldots,\alpha_n\}$.  If $\beta=\sum_{i=1}^n\lambda_i\alpha_i\in\Phi$, let $m_i:\Phi\to\Z$ pick out the coefficient of $\alpha_i$, so $m_i(\beta)=\lambda_i$.

For a full subquiver $Q'$ of $Q$ on vertices $\{i_1,\ldots, i_k\}$, let $\Phi_{\{i_1,\ldots, i_k\}}$ denote the root system associated to $(Q',v|_{Q'})$.
\begin{proposition}
Let $(Q,v)\in\mathcal{M}_{n}^A$, and $\e$ be the sign of the arrow between $n-1$ and $n$ in $Q$. 
The root system $\Phi_{(Q,v)}$ is as follows:
\begin{itemize}
    \item[(a)] If $(Q,v)$ is of type (a), then 
    \begin{equation}\label{eq_root_typeAa}
    \Phi_{(Q,v)}=\Phi_{\{1,\ldots, n-1\}}\cup \{\beta-\varepsilon m_{n-1}(\beta)\alpha_n \st \beta
    \in \Phi_{\{1,\ldots, n-1\}}\}\cup\{\pm \alpha_n\}.
    \end{equation}
    \item[(b)] If $(Q,v)$ is of type (b), then
    \begin{equation}\label{eq_root_typeAb}
    \Phi_{(Q,v)}=\Phi_{\{1,\ldots, n-1\}}\cup\Phi_{\{1,\ldots, n-2,n\}}\cup \{\pm (\alpha_{n-1}-\varepsilon \alpha_n)\}.
    \end{equation}
\end{itemize}
\end{proposition}
\begin{proof}
Let $\Theta_{(Q,v)}$ denote the right hand side of \eqref{eq_root_typeAa} if $Q$ is of type (a) or the right hand side of \eqref{eq_root_typeAb} if $Q$ is of type (b).

We proceed by induction on $n$.  If $n=1$ then $(Q,v)$ is of type (a) and $\Phi_{(Q,v)}=
\{\pm \alpha_1\}$, so the claim is true.

For $n>1$, we start by showing that $\Theta_{(Q,v)}\subseteq \Phi_{(Q,v)}$.
\begin{itemize}
    \item[(a)] Suppose $Q$ is of type (a). Then $\Phi_{\{1,\ldots, n-1\}}\subseteq \Phi_{(Q,v)}$ by definition of root system. Furthermore, for
$\beta\in\Phi_{\{1,\ldots, n-1\}}$ we have
\[ s_n(\beta)=\beta-\sum_{i=1}^nm_i(\beta) c_{ni}\alpha_n=\beta-\varepsilon m_{n-1}(\beta)\alpha_n. \] 
    Hence, since $\pm\alpha_n\in \Phi_{(Q,v)}$ by definition, we get $\Theta_{(Q,v)}\subseteq \Phi_{(Q,v)}$.
    \item[(b)] Suppose $Q$ is of type (b). Once again, by definition of root system, we have
    $\Phi_{\{1,\ldots, n-1\}}, \Phi_{\{1,\ldots, n-2,n\}}\subseteq \Phi_{(Q,v)}$. Furthermore
    $$
    s_n(\pm \alpha_{n-1})=\pm (\alpha_{n-1}-\varepsilon \alpha_n),
    $$
    and therefore $\Theta_{(Q,v)}\subseteq\Phi_{(Q,v)}$.
\end{itemize}

To prove the inverse inclusion, it is enough to show that $\Theta_{(Q,v)}$ is closed under reflections.
\begin{itemize}
    \item[(a)] 
    Let $\beta\in \Phi_{\{1,\ldots, n-1\}}$. 
    First, we need to show that $s_i(\beta)\in \Theta_{(Q,v)}$ for $1\leq i \leq n$.
%
    But $s_i(\beta)\in\Phi_{\{1,\ldots, n-1\}}\subseteq \Theta_{(Q,v)}$ for $i=1,\ldots, n-1$ by definition, and $s_n(\beta)=\beta-\varepsilon m_{n-1}(\beta)\alpha_n\in \Theta_{(Q,v)}$. 

Second, consider $\beta-\e m_{n-1}(\beta)\alpha_n$.  If $i\neq n-1$ then $m_{n-1}(s_i(\beta))=m_{n-1}(\beta)$, so
\[ s_i(\beta-\varepsilon m_{n-1}(\beta)\alpha_n)=s_i(\beta)-\varepsilon m_{n-1}(s_i(\beta))\alpha_n\in \Theta_{(Q,v)}. \]
For $i=n-1$ we split into two cases.
    \begin{itemize}
        \item If $\beta=\alpha_{n-1}$, then
        $$
        s_{n-1}(\beta-\varepsilon m_{n-1}(\beta)\alpha_n)= s_{n-1}(\alpha_{n-1}-\varepsilon\alpha_n)=-\alpha_{n-1}-\varepsilon (\alpha_n -\varepsilon\alpha_{n-1})=-\varepsilon\alpha_n\in\Theta_{(Q,v)}.
        $$
        The check for $\beta=-\alpha_{n-1}$ is analogous.
        \item 
Now consider $\beta\in \Phi_{\{1,\ldots, n-1\}}$ with $\beta\neq\pm\alpha_{n-1}$.  If $m_{n-1}(\beta)=0$ then we have already checked that $W\cdot( \beta-\varepsilon m_{n-1}(\beta)\alpha_n)=W\cdot\beta\subseteq \Theta_{(Q,v)}$, so we can assume that $m_{n-1}(\beta)\neq0$.

We claim that there is exactly one neighbour $j\neq n$ of $n-1$ such that $m_j(\beta)\neq0$.  To see this we use the inductive hypothesis.  If $Q_{\{1,\ldots, n-1\}}$ is of type (a) then the only neighbour $j\neq n$ of $n-1$ is $j=n-2$, and as $\beta\neq\pm\alpha_{n-1}$ the only possibility given the inductive hypothesis is that $m_{n-2}(\beta)\neq0$.  We write $\eta$ for the sign of the arrow between $n-2$ and $n-1$.
               \begin{center}
    \begin{tikzpicture}[xscale=2,baseline=(bb.base),
  quivarrow/.style={black, -latex}] 
  
    \node (a) at (-0.5,0) {$\circ$};
    \node [below] at (a) {\scriptsize {$n-2$}};
    
    \node (b) at (0,0) {$\circ$};
    \node [below] at (b) {\scriptsize {$n-1$}};

    \node (c) at (0.5,0) {$\circ$};
    \node [below] at (c) {\scriptsize {$n$}};

    \draw [shorten <=-1pt, shorten >=-1pt] (a) -- (b) node[midway, above]{\scriptsize \color{red}{$\eta$}};
    \draw [shorten <=-1pt, shorten >=-1pt] (b) -- (c) node[midway, above]{\scriptsize \color{red}{$\varepsilon$}};
    \end{tikzpicture}
    \end{center}
If $Q_{\{1,\ldots, n-1\}}$ is of type (b) then we have two neighbours $j\neq n$ of $n-1$: these are $j=n-2$ and $j=n-3$.  
               \begin{center}
    \begin{tikzpicture}[xscale=2,baseline=(bb.base),
  quivarrow/.style={black, -latex}]

  
    \node (a) at (0,0) {$\circ$};
    \node [below] at (a) {\scriptsize {$n-3$}};
    
    \node (c) at (0.25,0.5) {$\circ$};
    \node [above] at (c) {\scriptsize {$n-2$}};
    
    \node (b) at (0.5,0) {$\circ$};
    \node [below] at (b) {\scriptsize {$n-1$}};

   \node (d) at (1,0) {$\circ$};
    \node [below] at (d) {\scriptsize {$\phantom{3}n\phantom{3}$}};

    \draw [shorten <=-2pt, shorten >=-2pt] (a) -- (c); 
    \draw [shorten <=-2pt, shorten >=-2pt] (c) -- (b) node[midway, above right]{\scriptsize \color{red}{$\zeta$}};
    \draw [shorten <=-1pt, shorten >=-1pt] (a) -- (b); 
    \draw [shorten <=-1pt, shorten >=-1pt] (b) -- (d) node[midway, above]{\scriptsize \color{red}{$\e$}};
    \end{tikzpicture}
    \end{center}
By the inductive hypothesis, $\beta$ belongs to the union
\[\Phi_{\{1,\ldots, n-2\}}\cup\Phi_{\{1,\ldots, n-3,n-1\}}\cup \{\pm (\alpha_{n-2}-\zeta \alpha_{n-1})\}.\]
We have $m_{n-1}(\beta)\neq0$, so $\beta$ cannot belong to the first of these sets.  If it belongs to the second, we have $m_{n-2}(\beta)=0$ and we set $j=n-3$ and let $\eta$ denote the sign of the arrow between $n-3$ and $n-1$.  If it belongs to the third, we have $m_{n-3}(\beta)=0$; we set $j=n-2$ and $\eta=\zeta$.

So we have the diagram
               \begin{center}
    \begin{tikzpicture}[xscale=2,baseline=(bb.base),
  quivarrow/.style={black, -latex}] 
  
    \node (a) at (-0.5,0) {$\circ$};
    \node [below] at (a) {\scriptsize {$j$}};
    
    \node (b) at (0,0) {$\circ$};
    \node [below] at (b) {\scriptsize {$n-1$}};

    \node (c) at (0.5,0) {$\circ$};
    \node [below] at (c) {\scriptsize {$\phantom{3}n\phantom{3}$}};

    \draw [shorten <=-1pt, shorten >=-1pt] (a) -- (b) node[midway, above]{\scriptsize \color{red}{$\eta$}};
    \draw [shorten <=-1pt, shorten >=-1pt] (b) -- (c) node[midway, above]{\scriptsize \color{red}{$\varepsilon$}};
    \end{tikzpicture}
    \end{center}
    with $m_j(\beta)\neq 0$. Furthermore, by checking each of the above cases, we know that $m_{n-1}(\beta)=-\eta m_j(\beta)$. Therefore
    \begin{align*}
        s_{n-1}(\beta) & = \sum_{i\neq j,n-1} m_i(\beta)\alpha_i +m_j(\beta)s_{n-1}(\alpha_j)+m_{n-1}(\beta)s_{n-1}(\alpha_{n-1})\\
& = \sum_{i\neq j,n-1} m_i(\beta)\alpha_i +m_j(\beta)\left( \alpha_j-\eta\alpha_{n-1} \right)-m_{n-1}(\beta)\alpha_{n-1}\\
&=\beta-m_{n-1}(\beta)\alpha_{n-1}
\end{align*}
so
\begin{align*}
        s_{n-1}(\beta -\varepsilon m_{n-1}(\beta)\alpha_n) & = \beta-m_{n-1}(\beta)\alpha_{n-1} - \varepsilon m_{n-1}(\beta)\left(\alpha_n-\e \alpha_{n-1}\right)\\
& = \beta - \varepsilon m_{n-1}(\beta)\alpha_n\in \Theta_{(Q,v)}.
\end{align*}
\end{itemize}

Third, consider $\pm\alpha_n$.  As $\Theta_{(Q,v)}=-\Theta_{(Q,v)}$, it is enough to consider $\alpha_n$.  If $i\neq n-1$ then $s_i(\alpha_n)\in\{\pm\alpha_n\}$.  For $i=n-1$ we have $s_{n-1}(\alpha_n)=\alpha_n-\e\alpha_{n-1}$.  Let $\beta = -\e\alpha_{n-1}\in\Phi_{\{1,\ldots, n-1\}}$.  Then
\[ \beta - \e m_{n-1}(\beta)\alpha_n = -\e\alpha_{n-1}+\alpha_n=s_{n-1}(\alpha_n) \]
so $s_i(\alpha_n)\in\Theta_{(Q,v)}$.

\item[(b)] Let $\beta\in \Phi_{\{1,\ldots, n-1\}}$. 
For $i=1,\ldots,n-1$ we have $s_i(\beta)\in \Phi_{\{1,\ldots, n-1\}}\subseteq \Theta_{(Q,v)}$ by definition. 

Furthermore, if $m_{n-1}(\beta)=0$, then $\beta\in \Phi_{\{1,\ldots, n-2\}}$ and therefore $s_n(\beta)\in\Phi_{\{1,\ldots, n-2,n\}}\subseteq \Theta_{(Q,v)}$. 

Therefore, we can assume $m_{n-1}(\beta)\neq 0$. If $\beta=\pm \alpha_{n-1}$, then $s_n(\beta)=\pm (\alpha_{n-1}-\varepsilon\alpha_n)\in\Theta_{(Q,v)}$. Otherwise, using the same argument as for case (a), we have $m_{n-1}(\beta)=-\gamma m_{n-2}(\beta)$. Thus
$$
s_n(\beta)=\beta -\delta m_{n-2}(\beta) \alpha_n -\varepsilon m_{n-1}(\beta) \alpha_n =\beta - (\delta -\varepsilon\gamma)m_{n-2}(\beta)\alpha_n=\beta\in \Theta_{(Q,v)},
$$
where the last equality follows since $\gamma\delta\varepsilon=1$.

One proves similarly that $W\cdot (\Phi_{\{1,\ldots, n-2,n\}})\subseteq \Theta_{(Q,v)}$. 

Finally, one checks that $s_i$ fixes $\pm (\alpha_{n-1}-\varepsilon\alpha_n)$ for $i=1,\ldots,n-2$, and that $s_{n-1} (\alpha_{n-1}-\varepsilon\alpha_n)=-\e\alpha_n$ and $s_{n} (\alpha_{n-1}-\varepsilon\alpha_n)=\alpha_{n-1}$.  So $s_i(\pm (\alpha_{n-1}-\varepsilon\alpha_n))\in\Theta_{(Q,v)}$ for all $i=1,\ldots, n.$

Hence $\Phi_{(Q,v)}\subseteq\Theta_{(Q,v)}$, and we are done.
\end{itemize}
\end{proof}

%
%

\section{Presentations of Lie algebras}\label{s:Lie}

\subsection{Presentations}

In this section we work with Lie algebras over the complex numbers.
We use iterated bracket notation:
\[ [x_1,x_2,\ldots, x_n]:= [x_1,[x_2,[\ldots, x_n]]].\]

We first consider four types of relations:
\begin{definition}\label{def:r1234}
Let $C=(c_{ij})$ be a quasi-Cartan matrix.  Define the Lie algebra $\g_4(C)$ with generators $\{h_i,e_{\e i} \st \e\in\{\pm 1\}, 1\leq i\leq n\}$ subject to the following relations, for $\e,\delta\in\{\pm1\}$:
\begin{itemize}
    \item[(R1)] $[h_i,h_j]=0$;
    \item[(R2)] $[e_i,e_{-i}]=h_i$;
    \item[(R3)] $[h_i,e_{\e j}]=\e c_{ij}e_{\e j}$;
    \item[(R4)] $(\ad e_{\e i})^{-M+1}(e_{\delta j})=0$, where $M=\min\{0, \e\delta c_{ij}\}$.
\end{itemize}
\end{definition}

Note that if $C$ is a Cartan matrix of Dynkin type $\Delta$, then Definition \ref{def:r1234} reduces to the classical Serre presentation of the Lie algebra $\g(\Delta)$ \cite{serre}.

We record two useful calculational results which always hold in $\g_4(C)$.  We will use both often.
\begin{lemma}\label{lem:h-prod}
$[h_i,e_{\e_1 j_1}, \ldots e_{\e_\ell j_\ell}]=\sum_{k=1}^\ell \e_kc_{ik}[e_{\e_1 j_1}, \ldots e_{\e_\ell j_\ell}]$.
\end{lemma}
\begin{proof}
Use (R3), induction, and the Jacobi identity.
\end{proof}
\begin{lemma}\label{lem:commute-es}
If $\delta=\sgn(c_{ij})$ then $[e_i,e_{\delta j}]=[e_{-i},e_{-\delta j}]=0$.
\end{lemma}
\begin{proof}
This follows from (R4).  To show the first term is zero, set $\e=1$ and $\delta=\sgn(c_{ij})$, so $\delta c_{ij}\geq0$ and thus $M=0$.   For the second, set $\e=-1$ and $\delta=-\sgn(c_{ij})$.
\end{proof}

The next set of relations depends on chordless cycles of $C$.  If $C$ is mutation Dynkin then, by \cite[Proposition 2.6]{bgz}, there are only four possibilities for the unsigned diagram of the cycle.  Either every edge has weight 1, or we are in one of the following three cases:

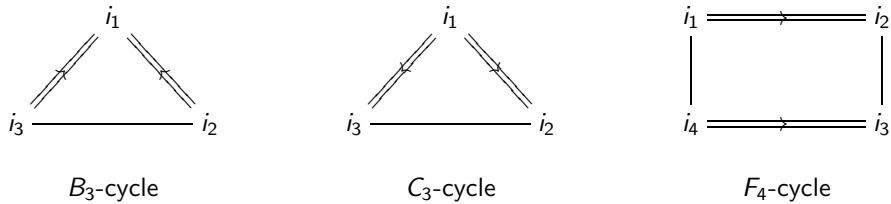
\[ \xymatrix @=10pt{
 & i_1 & &&& & i_1 & &&&  i_1\ar@{=}[rr] |{\SelectTips{cm}{}\object@{>}}  && i_2 \\
&\\
i_3\ar@{=}[uur]  |{\SelectTips{cm}{}\object@{>}}
\ar@{-}[rr] && i_2\ar@{=}[uul] |{\SelectTips{cm}{}\object@{>}}
&&&i_3\ar@{=}[uur] |{\SelectTips{cm}{}\object@{<}} && i_2\ar@{=}[uul]  |{\SelectTips{cm}{}\object@{<}} \ar@{-}[ll]   &&& i_4\ar@{=}[rr] |{\SelectTips{cm}{}\object@{>}}  \ar@{-}[uu] && i_3\ar@{-}[uu]\\
&\text{$B_3$-cycle}&&&&&\text{$C_3$-cycle}&&&&&\text{$F_4$-cycle}
} \]
These cycles occur in mutation types $B_n$, $C_n$, and $F_4$, respectively.  Note that the $B_3$-cycle and the $C_3$-cycle also appear in mutation type $F_4$.

In cycles with double edges, the vertices correspond to roots which are either long or short: $\xymatrix{i\ar@{=}[r] |{\SelectTips{cm}{}\object@{>}} &j}$ tells us that $i$ is long and $j$ is short.  So in the $B_3$-cycle displayed above, the vertex $i_1$ is short and the vertices $i_2$ and $i_3$ are long.  
\begin{definition}\label{def:r5}
Define the Lie algebra $\g_5(C)$ as $\g_4(C)$ subject to the following extra relations.  To each chordless cycle $i_1,i_2,\ldots,i_t$ where every edge has weight 1 we impose the relations:
\[ \text{(R5)} \;\; [e_{\e_1 i_1}, \; e_{\e_2 i_2}, \; \ldots, \;  e_{\e_t i_t}]=0 \;\;\text{ where }\;\;\e_{q+1}=-\sgn(c_{i_q,i_{q+1}})\e_{q}\]
where $\e_1\in\{\pm 1\}$. 

Notice that for each chordless cycle as above we impose $4t$ relations: $2t$ going clockwise (two for every starting point $i_j$) and $2t$ going anticlockwise.

To each chordless cycle with an edge of higher weight, we only impose the subset of these relations where $i_1$ and $i_t$ are not both long.
\end{definition}

We unpack definition \ref{def:r5} for the three cycles with double edges, using the three diagrams above.  

For the $B_3$-cycle we impose only 3-cycle relations which do not go ``long, short, long''.  Explicitly, these are:
\[ [e_{\e_1 i_{\sigma(1)}}, \, e_{\e_2 i_{\sigma(2)}}, \;  e_{\e_3 i_{\sigma(3)}}] \;\;\text{ where }\;\;\e_{q+1}=-\sgn(c_{i_{\sigma(q)},i_{\sigma(q+1)}})\e_{q}\]
for $ \e_1\in\{\pm 1\}$ and $\sigma\in\ \Dih(3)$ such that $\{\sigma(1),\sigma(3)\}\neq\{2,3\}$, where $\Dih(p)$ denotes the dihedral group of order $2p$.

For the $C_3$-cycle there is no path of length 3 which starts and ends at a long root, so we impose all twelve 3-cycle relations.  Explicitly, these are:
\[ [e_{\e_1 i_{\sigma(1)}}, \, e_{\e_2 i_{\sigma(2)}}, \;  e_{\e_3 i_{\sigma(3)}}] \;\;\text{ where }\;\;\e_{q+1}=
-\sgn(c_{i_{\sigma(q)},i_{\sigma(q+1)}})
\e_{q}\]
for $\sigma\in \Dih(3)$.

For the $F_4$-cycle we impose only 4-cycle relations which do not go ``long, short, short, long''.  Explicitly, these are:
\[ [e_{\e_1 i_{\sigma(1)}}, \, e_{\e_2 i_{\sigma(2)}}, \;  e_{\e_3 i_{\sigma(3)}}, \;  e_{\e_4 i_{\sigma(4)}}] \;\;\text{ where }\;\;\e_{q+1}=-\sgn(c_{i_{\sigma(q)},i_{\sigma(q+1)}})\e_{q}\]
for $\sigma\in\ \Dih(4)$ such that $\{\sigma(1),\sigma(4)\}\neq\{1,4\}$.

In the introduction of \cite{pr}, following \cite{br}, a Lie algebra which we denote $\widehat{\g}_4(C)$ is defined.  We write its generators as $\hat h_i$ and $\hat e_{\e i}$.  The relations are similar to our relations in Definition \ref{def:r1234}, except that the second and third relations are swapped and their relation $[\hat h_i,\hat e_{\e j}]=-\e c_{ij}\hat e_{\e j}$ has a minus sign.
\begin{lemma}\label{lem:iso_PR_GM}
The map 
\begin{align*}
    \hat{e}_{\e i} & \mapsto \e e_{\e i} \\
    \hat{h}_i & \mapsto -h_i
\end{align*}
extends to an isomorphism $\widehat{\g}_4(C)\arr\sim \g_4(C)$.
\end{lemma}
\begin{proof}
This consists of straightforward checks that the relations (R1) to (R4) are preserved by the map defined above, since the map is clearly bijective.
\end{proof}

\begin{remark}
To state their fifth relation in non-simply laced types and define their Lie algebra $\widehat{\g}_5(C)$, P\'erez-Rivera use a 3-cycle denoted $\mathcal{C}_\mathbb{C}(i_1,i_2,i_3)$, but the diagram which defines it seems to contain a typo as only directed edges should be weighted.  If we replace their diagram by
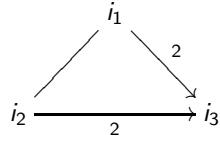
\[
\xymatrix{ & i_1 \ar@{-}[dl]\ar[dr]^2& \\
i_2\ar[rr]_2 & & i_3
}
\]
and interpret their statements correctly, our relations for the $B_3$- and $C_3$-cycles are consistent with theirs, and the isomorphism of Lemma \ref{lem:iso_PR_GM} extends to an isomorphism $\widehat{\g}_5(C)\arr\sim \g_5(C)$ for all positive quasi-Cartan matrices $C$ which are not of mutation type $F_4$.

However, P\'erez-Rivera include all potential relations in $F_4$-cycles, while we only include 12 of the 16 possible relations.  We will show in Lemma \ref{lem:f4zero} below that imposing all 16 relations would give the zero Lie algebra.  We believe there is an error in \cite[Lemma 5.3]{pr}: the final paragraph of their proof states that $[e_{-\e i}, e_{-\e s}, \tilde{e}_{\e r}, e_{\e j}]=0$ but we think this should be nonzero.
\end{remark}

\subsection{Relations}

In this section we first explain why we excluded certain relations in Definition \ref{def:r5}, and then we show that we could have imposed fewer relations in Definition \ref{def:r5} to get the same Lie algebra.

Here we explain why we do not include relations starting and ending at a long root.  
To simplify notation, we replace the vertices $i_j$ by $j$ in our diagrams.

We start by considering the $B_3$-cycle and its Cartan counterpart
\[ \begin{gathered} \xymatrix{
 & 1 & \\
3\ar@{=}[ur]^{\delta_{13}}  |{\SelectTips{cm}{}\object@{>}}
\ar@{-}[rr]_{\delta_{23}} && 2\ar@{=}[ul]_{\delta_{12}} |{\SelectTips{cm}{}\object@{>}}
}
\end{gathered}
\hspace{1cm} C=\begin{pmatrix}
    2 & \delta_{12} & \delta_{13} \\
    2\delta_{12} & 2 & \delta_{23} \\
    2\delta_{13} & \delta_{23} & 2
\end{pmatrix} \]
where the signs $\delta_{12},\delta_{23},\delta_{13}\in\{\pm 1\}$ are such that $\delta_{12}\delta_{23}\delta_{13}=1$.  The relations we did not include in Definition \ref{def:r5} are 
\begin{align}
[e_{-\e\delta_{12}\delta_{13}2},e_{\e \delta_{13}1},e_{-\e 3}]=0\label{eq_1_B_3} \\
[e_{-\e\delta_{13}\delta_{12}3},e_{\e \delta_{12}1},e_{-\e 2}]=0\label{eq_2_B_3}
\end{align}

\begin{lemma}
If $C$ is a quasi-Cartan matrix coming from a $B_3$-cycle, then the quotient of $\g_5(C)$ by the ideal generated by any of the relation \eqref{eq_1_B_3} or \eqref{eq_2_B_3} is the zero Lie algebra.
\end{lemma}
\begin{proof}
We show that if we impose a relation from \eqref{eq_1_B_3} to $\g_5(C)$ we get the zero Lie algebra, as the other check is similar.

As $\ad(e_{\e\delta_{12}\delta_{13}2})$ is a derivation, we have
\begin{align*}
    0 & = [e_{\e\delta_{12}\delta_{13}2},e_{-\e\delta_{12}\delta_{13}2},e_{\e \delta_{13}1},e_{-\e 3}] \\
   & =  \e\delta_{12}\delta_{13}[h_2,e_{\e \delta_{13}1},e_{-\e 3}]+[e_{-\e\delta_{12}\delta_{13}2},e_{\e\delta_{12}\delta_{13}2},e_{\e \delta_{13}1},e_{-\e 3}] \\
   & =  \e\delta_{12}\delta_{13}(2\e\delta_{12}\delta_{13}-\e\delta_{23})[e_{\e \delta_{13}1},e_{-\e 3}]-[e_{-\e\delta_{12}\delta_{13}2},e_{\e\delta_{12}\delta_{13}2},e_{-\e 3},e_{\e \delta_{13}1}].
\end{align*}
Notice that $[e_{\e\delta_{12}\delta_{13}2},e_{-\e 3},e_{\e \delta_{13}1}]=0$ by (R5), since $-\sgn (c_{23})\e\delta_{12}\delta_{13}=-\e\delta_{12}\delta_{23}\delta_{13}=-\e$ and $-\sgn(c_{31})(-\e)=\e\delta_{13}$. Therefore, as $\e\delta_{12}\delta_{13}(2\e\delta_{12}\delta_{13}-\e\delta_{23})\neq 0$, we get $[e_{\e \delta_{13}1},e_{-\e 3}]=0$. As a consequence, Lemma \ref{lem:h-prod} yields
\[
0=[e_{-\e\delta_{13}1},e_{\e\delta_{13}1},e_{-\e 3}]=-\e\delta_{13}[h_1,e_{-\e 3}]=e_{-\e3}.
\]
Hence we get $e_{3}=e_{-3}=0$. Similar computations yield $e_1=e_{-1}=e_2=e_{-2}=0$, and thus $h_i=0$ by (R2). Hence $\g_5(C)$ with an additional relation \eqref{eq_1_B_3} is the zero Lie algebra.
\end{proof}

Now consider the $F_4$-cycle and associated Cartan counterpart
\[
       \begin{gathered}  \xymatrix{1\ar@{-}[d]_{\delta_{14}} \ar@{=}[rr]^(0.6){\delta_{12}}  |{\SelectTips{cm}{}\object@{>}} & & 2\ar@{-}[d]^{\delta_{23}} \\
4\ar@{=}[rr]_(0.6){\delta_{34}}  |{\SelectTips{cm}{}\object@{>}} & & 3}  \end{gathered}
\hspace{1cm}C=\begin{pmatrix}
    2 & 2\delta_{12} & 0 & \delta_{14} \\
    \delta_{12} & 2 & \delta_{23} & 0 \\
    0 & \delta_{23} & 2 & \delta_{34} \\
    \delta_{14} & 0 & 2\delta_{34} & 2
\end{pmatrix}\]
where the signs $\delta_{12},\delta_{14},\delta_{23},\delta_{34}\in\{\pm 1\}$ are such that $\delta_{12}\delta_{23}\delta_{34}\delta_{14}=-1$. The relations we did not include in Definition \ref{def:r5} are 
\begin{align}
[e_{-\e\delta_{12}\delta_{23}\delta_{34} 4},e_{\e\delta_{12}\delta_{23} 3},e_{-\e \delta_{12} 2},e_{\e 1}]=0\label{eq_1_F_4} \\
[e_{-\e\delta_{12}\delta_{23}\delta_{34} 1},e_{\e\delta_{23}\delta_{34} 2},e_{-\e \delta_{34} 3},e_{\e 4}]=0\label{eq_2_F_4}
\end{align}

\begin{lemma}\label{lem:f4zero}
If $C$ is a quasi-Cartan matrix coming from a $F_4$-cycle, then
the quotient of $\g_5(C)$ by the ideal generated by any of the relation \eqref{eq_1_F_4} or \eqref{eq_2_F_4} is the zero Lie algebra.
\end{lemma}
\begin{proof}
Once again, we show that if we impose one of the relations from \eqref{eq_1_F_4} to $\g_5(C)$ we get the zero Lie algebra, as the other check is similar.

We have
\begin{align*}
0 &=  [e_{\e\delta_{12}\delta_{23}\delta_{34}4},e_{-\e\delta_{12}\delta_{23}\delta_{34} 4},e_{\e\delta_{12}\delta_{23} 3},e_{-\e \delta_{12} 2},e_{\e 1}] \\
&=  \e\delta_{12}\delta_{23}\delta_{34}[h_4,e_{\e\delta_{12}\delta_{23} 3},e_{-\e \delta_{12} 2},e_{\e 1}]+[e_{-\e\delta_{12}\delta_{23}\delta_{34} 4},e_{\e\delta_{12}\delta_{23}\delta_{34} 4},e_{\e\delta_{12}\delta_{23} 3},e_{-\e \delta_{12} 2},e_{\e 1}] 
\end{align*}
We will show the second summand is zero.  By (R4) we have $[e_{\e\delta_{12}\delta_{23}\delta_{34} 4},e_{\e\delta_{12}\delta_{23} 3}]=0$ and $[e_{\e\delta_{12}\delta_{23}\delta_{34} 4},e_{-\e \delta_{12} 2}]=0$.
Therefore
\begin{align*}
 [e_{\e\delta_{12}\delta_{23}\delta_{34} 4},e_{\e\delta_{12}\delta_{23} 3},e_{-\e \delta_{12} 2},e_{\e 1}]
 &=  [e_{\e\delta_{12}\delta_{23} 3},e_{\e\delta_{12}\delta_{23}\delta_{34} 4},e_{-\e \delta_{12} 2},e_{\e 1}] \\
 &=  [e_{\e\delta_{12}\delta_{23} 3},e_{-\e \delta_{12} 2},e_{\e\delta_{12}\delta_{23}\delta_{34} 4},e_{\e 1}]    \\
 &=  -[e_{\e\delta_{12}\delta_{23} 3},e_{-\e \delta_{12} 2},e_{\e 1},e_{\e\delta_{12}\delta_{23}\delta_{34} 4}]    \\
 & =0
\end{align*}
where the final equality is by (R5).  Therefore
\begin{align*}
0 &=  \e\delta_{12}\delta_{23}\delta_{34}[h_4,e_{\e\delta_{12}\delta_{23} 3},e_{-\e \delta_{12} 2},e_{\e 1}]\\
 & = \e\delta_{12}\delta_{23}\delta_{34}(2\e\delta_{12}\delta_{23}\delta_{34}+\e\delta_{14})[e_{\e\delta_{12}\delta_{23} 3},e_{-\e \delta_{12} 2},e_{\e 1}]
\end{align*}
so as $\e\delta_{12}\delta_{23}\delta_{34}(2\e\delta_{12}\delta_{23}\delta_{34}+\e\delta_{14})\neq 0$ we get $[e_{\e\delta_{12}\delta_{23} 3},e_{-\e \delta_{12} 2},e_{\e 1}]=0$. By applying $[e_{-\e\delta_{12}\delta_{23} 3},-]$ and using a similar strategy to above we get $[e_{-\e \delta_{12} 2},e_{\e 1}]=0$. Finally, applying $[e_{\e \delta_{12} 2},-]$ yields $e_{\e 1}=0$. Similar computations yield $e_{\e i}=0$ for all $i$, and therefore also $h_i=0$ by (R2).  Hence $\g_5(C)$ with an additional relation \eqref{eq_1_F_4} is the zero Lie algebra.
\end{proof}


Now we want to show that, for any chordless cycle in a signed valued quiver of mutation Dynkin type, most of the (R5) relations are redundant.  If our cycle is of type $B_3$ or $F_4$ we need to take at most two of the (R5) relations; in other cases just one (R5) relation is enough.

First we present a lemma that holds for arbitrary Lie algebras, and will be useful in the following (see \cite[Lemma 4.2]{br}).
\begin{lemma}\label{lem_flip_jac_id}
Let $\mathfrak{g}$ be a Lie algebra, and $x_1,\ldots, x_t\in \mathfrak{g}$ such that $[x_i,x_j]=0$ for all $|i-j|\neq 1$. Then
\[
[x_t,\ldots, x_1]=(-1)^{t+1}[x_1,\ldots, x_t].
\]
\end{lemma}

We start by analysing the simply laced case.
\begin{lemma}\label{lem_typeD_cycle_2rels}
Let $(Q,v)$ be a signed quiver of mutation Dynkin type with a chordless cycle $i_1\to i_2\to \ldots i_t\to i_1$ where every edge has weight 1 and $i_q\to i_{q+1}$ has sign $\delta_{q,q+1}$.
Let $\e_1\in\{\pm 1\}$ and quotient $\g_4(C)$ by the relation $[e_{\e_1 i_1},e_{\e_2 i_2},\ldots , e_{\e_t i_t}]=0$, where $\e_{q+1}=-\delta_{q,q+1}\e_q$.  Then
\[
[e_{\e_2 i_2},\ldots , e_{\e_t i_t},e_{\e_{t+1} i_1}]=0
\]
where $\e_{t+1}=-\delta_{t1}\e_{t}$.
\end{lemma}
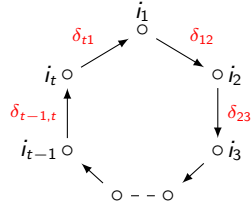
\[
\begin{tikzpicture}[scale=1,
  quivarrow/.style={black, -latex}] 
\begin{scope}[xscale=2]

\node (it) at (0,0) {$\circ$};
\node [left] at (it) {\small {$i_t$}}; 

\node (i2) at (1,0) {$\circ$};
\node [right] at (i2) {\small {$i_2$}};

\node (i1) at (0.5,0.6) {$\circ$};
\node[above] at (i1) {\small {$i_1$}};

\node (it1) at (0,-1) {$\circ$};
\node [left] at (it1) {\small {$i_{t-1}$}};

\node (j4) at (0.33,-1.6) {$\circ$};

\node (i3) at (1,-1) {$\circ$};
\node [right] at (i3) {\small {$i_3$}};

\node (jt2) at (0.67,-1.6) {$\circ$};


\draw [quivarrow, shorten <=-1pt, shorten >=-1pt] (it) -- (i1) node[midway, above left]{\scriptsize \color{red}{$\delta_{t1}$}};
\draw [quivarrow, shorten <=-1pt, shorten >=-1pt] (i1) --(i2) node[midway, above right]{\scriptsize \color{red}{$\delta_{12}$}};
\draw [quivarrow, shorten <=-1pt, shorten >=-1pt] (i2) -- (i3) node[midway, right]{\scriptsize \color{red}{$\delta_{23}$}};
\draw [quivarrow, shorten <=-1pt, shorten >=-1pt] (i3) -- (jt2); 
\draw [quivarrow, shorten <=-1pt, shorten >=-1pt] (j4) -- (it1);
\draw [quivarrow, shorten <=-1pt, shorten >=-1pt] (it1) -- (it) node[midway, left]{\scriptsize \color{red}{$\delta_{t-1,t}$}};
\draw [dashed, shorten <=-1pt, shorten >=-1pt] (jt2) -- (j4);

\end{scope}

\end{tikzpicture}
\]
\begin{proof}
To simplify the notation we let $i_q=q$.  Let $C=C(Q,v)$ be the Cartan counterpart of $(Q,v)$ and note that, as every edge has weight $1$, we have $\delta_{q,q+1}=c_{q,q+1}$.

Suppose $[e_{\e_1 1},e_{\e_2 2},\ldots , e_{\e_t t}]=0$ and apply $[e_{-\e_1 1},-]$.  Using the Jacobi identity, (R2), and Lemma \ref{lem:h-prod}, we get
\begin{align*}
0 &=  [e_{-\e_1 1},e_{\e_1 1},e_{\e_2 2},\ldots , e_{\e_t t}] \\
 &=  -\e_1[h_1,e_{\e_2 2},\ldots , e_{\e_t t}]+ [e_{\e_1 1},e_{-\e_1 1},e_{\e_2 2},\ldots , e_{\e_t t}] \\
 &=  -\e_1\lambda[e_{\e_2 2},\ldots , e_{\e_t t}] + [e_{\e_1 1},e_{-\e_1 1},e_{\e_2 2},\ldots , e_{\e_t t}]
\end{align*}
where $\lambda=\sum_{j=2}^t \e_jc_{1j}$.  As our cycle is chordless, $c_{1j}=0$ for $j=3,\ldots, t-1$, and the weight 1 assumption implies $c_{12}=\delta_{12}$ and $c_{1t}=\delta_{t1}$, so
\[ \lambda = \e_2c_{12} + \e_tc_{1t} = \e_1(-\delta_{12}^2+(-1)^{t-1}\delta_{12}\delta_{23}\cdots\delta_{t-1,t}\delta_{t1})
\]
Now, since $Q$ is mutation Dynkin, the cycle $C$ is not dangerous, and therefore $\delta_{12}\ldots\delta_{t-1,t}\delta_{1t}=(-1)^{t+1}$, which implies $\lambda=0$.
So $[e_{\e_1 1},e_{-\e_1 1},e_{\e_2 2},\ldots , e_{\e_t t}] =0$.

We have $[e_{-\e_1 1},e_{\e_2 2}]=[e_{-\e_1 1},e_{-\e_1\delta_{12} 2}]$, so by Lemma \ref{lem:commute-es} this is zero.  Also, as our cycle is chordless, we have $[e_{-\e_11},e_{\e_jj}]=0$ for $j=3,\ldots, t-1$.  Therefore, repeatedly using the Jacobi identity gives
\[ [e_{\e_1 1},\ldots , e_{\e_{t-1} t-1},e_{-\e_1 1},e_{\e_t t}] =0. \]

Applying $[e_{-\e_1 1},-]$ once again, we get
\begin{align*}
0 &= [e_{-\e_1 1},e_{\e_1 1},\ldots , e_{\e_{t-1} t-1},e_{-\e_1 1},e_{\e_t t}]  \\
&= -\e_1[h_1,e_{\e_2 2},\ldots , e_{\e_{t-1} t-1},e_{-\e_1 1},e_{\e_t t}] + [e_{\e_1 1},e_{-\e_1 1},\ldots , e_{\e_{t-1} t-1},e_{-\e_1 1},e_{\e_t t}]  \\
&= -\e_1\mu [e_{\e_2 2},\ldots , e_{\e_{t-1} t-1},e_{-\e_1 1},e_{\e_t t}] + [e_{\e_1 1},\ldots , e_{\e_{t-1} t-1},e_{-\e_1 1},e_{-\e_1 1},e_{\e_t t}]  
\end{align*}
where $\mu=\e_2c_{12}-\e_1c_{11}+\e_tc_{1t}=\lambda-2\e_1=-2\e_1\neq0$.  
 Also, the second summand of the above equation is zero, since $[e_{-\e_1 1},e_{-\e_1 1},e_{\e_t t}]=0$ by (R4). As a consequence, by skew-symmetry we have
\[
[e_{\e_2 2},e_{\e_3 3},\ldots,e_{\e_{t-1} t-1}, e_{\e_t t},e_{-\e_1 1}]=0
\]
which proves the claim since 
$\e_{t+1}=(-1)^{t}\delta_{12}\delta_{23}\cdots\delta_{t-1,t}\delta_{t1}\e_1=-\e_1$.
\end{proof}

\begin{corollary}\label{cor_typeD_cycle_2rels}
Under the assumptions of Lemma \ref{lem_typeD_cycle_2rels}, all the relations of type (R5) associated to the cycle $i_1\to i_2\to \ldots i_t\to i_1$ are satisfied.
\end{corollary}
\begin{proof}
Notice that $\e_{t+1}=-\delta_{t1}\e_t=-(-1)^{t+1}\delta_{12}\ldots\delta_{t-1,t}\delta_{t1}\e_1=-\e_1$, and therefore applying  Lemma \ref{lem_typeD_cycle_2rels} $2t$ times we get all the possible clockwise assignments of signs that satisfy the conditions of (R5).

The above computation also shows that ${\e_t}=\delta_{t1}{\e_1}$. Therefore, using (R4) and Lemma \ref{lem:commute-es} one sees that $e_{\e_1 i_1}$ commutes with $e_{\e_j i_j}$ for all $|i-j|\neq 1$. Thus Lemma \ref{lem_flip_jac_id} implies that
\[
[e_{\e_t i_t},\ldots , e_{\e_1 i_1}]=(-1)^{t+1}[e_{\e_1 i_1},\ldots, e_{\e_t i_t}]=0.
\]
Hence, using the same argument as above we also get all possible anticlockwise assignments of signs that satisfy the conditions of (R5).
\end{proof}

We now give the statement for $B_3$, $C_3$ and $F_4$-cycles. We consider signs on such cycles as follows.
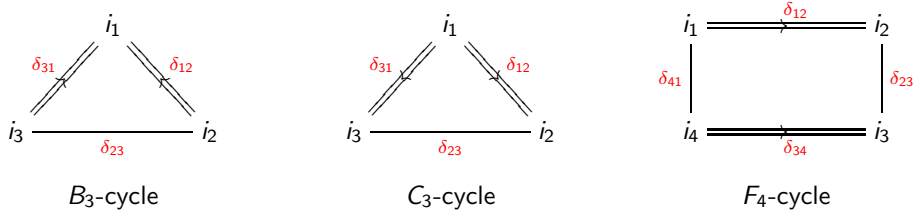
\[ \xymatrix @=10pt{
 & i_1 & &&& & i_1 & &&&  i_1\ar@{=}[rr]^(0.55){\color{red}\delta_{12}} |{\SelectTips{cm}{}\object@{>}}  && i_2 \\
&\\
i_3\ar@{=}[uur]^{\color{red}\delta_{31}}  |{\SelectTips{cm}{}\object@{>}}
\ar@{-}[rr]_{\color{red}\delta_{23}} && i_2\ar@{=}[uul]_{\color{red}\delta_{12}} |{\SelectTips{cm}{}\object@{>}}
&&&i_3\ar@{=}[uur]^{\color{red}\delta_{31}} |{\SelectTips{cm}{}\object@{<}} && i_2\ar@{=}[uul]_{\color{red}\delta_{12}}  |{\SelectTips{cm}{}\object@{<}} \ar@{-}[ll]^{\color{red}\delta_{23}}   &&& i_4\ar@{=}[rr]_(0.55){\color{red}\delta_{34}} |{\SelectTips{cm}{}\object@{>}}  \ar@{-}[uu]^{\color{red}\delta_{41}} && i_3\ar@{-}[uu]_{\color{red}\delta_{23}}\\
&\text{$B_3$-cycle}&&&&&\text{$C_3$-cycle}&&&&&\text{$F_4$-cycle}
} \]

\begin{lemma}
Let $(Q,v)$ be a signed valued quiver of mutation Dynkin type. 
\begin{enumerate}
    \item Suppose $Q$ contains a $B_3$-cycle. Then the following two relations imply all the other (R5) relations
    \[
    [e_{\e i_1},e_{-\delta_{12}\e i_2},e_{\delta_{12}\delta_{23}\e i_3}]= 0 
    \]
    for all $\e\in\{\pm 1\}$.
    \item Suppose $(Q,v)$ contains a $C_3$-cycle. 
    Then any one (R5) relation for this cycle implies all other (R5) relations for this cycle.
    \item Suppose $Q$ contains an $F_4$-cycle.  Then the following two relations imply all the other (R5) relations
    \[
    [e_{\e i_2},e_{-\delta_{23}\e i_3},e_{\delta_{23}\delta_{34}\e i_4},e_{-\delta_{23}\delta_{34}\delta_{41}\e i_1}]=0
    \]
    for all $\e\in\{\pm 1\}$.
\end{enumerate}
\end{lemma}
\begin{proof}
Again, let $i_j=j$.  
For the $C_3$-cycle, the same argument as in Lemma \ref{lem_typeD_cycle_2rels} goes through in all cases except one: that $[e_{\e_1 1},e_{\e_2 2},e_{\e_3 3}]=0$ implies $[e_{\e_2 2},e_{\e_3 3},e_{\e_4 1}]=0$.

Therefore, suppose $[e_{\e_1 1},e_{\e_2 2},e_{\e_3 3}]=0$.  As in the proof of Lemma \ref{lem_typeD_cycle_2rels}, we get that $[e_{\e_1 1},e_{\e_2 2},e_{-\e_1 1},e_{\e_3 3}]=0$.  
Our Cartan matrix is
\[ C=\begin{pmatrix}
    2 & 2\delta_{12} & 2\delta_{31} \\
    \delta_{12} & 2 & \delta_{23} \\
    \delta_{31} & \delta_{23} & 2
\end{pmatrix}\] so,
applying $[e_{-\e_1 1},-]$ once more, we get
\begin{align*}
0 & = [e_{-\e_1 1},e_{\e_1 1},e_{\e_2 2},e_{-\e_1 1},e_{\e_3 3}] \\
0 & = -\e_1[h_1,e_{\e_2 2},e_{-\e_1 1},e_{\e_3 3}] + [e_{\e_1 1},e_{-\e_1 1},e_{\e_2 2},e_{-\e_1 1},e_{\e_3 3}] \\
0 & = -\e_1\mu[e_{\e_2 2},e_{-\e_1 1},e_{\e_3 3}] + [e_{\e_1 1},e_{\e_2 2},e_{-\e_1 1},e_{-\e_1 1},e_{\e_3 3}] 
\end{align*}
where $\mu=2\delta_{12}(-\delta_{12}\e)-2\e_1+2\delta_{13}\delta_{23}\delta_{12}\e_1=-2\e_1$. Notice that, in contrast to the simply laced case, relation (R4) does not imply that the second summand is 0. In any case, we have
\[ 2[e_{\e_2 2},e_{-\e_1 1},e_{\e_3 3}] + [e_{\e_1 1},e_{\e_2 2},e_{-\e_1 1},e_{-\e_1 1},e_{\e_3 3}] =0
\]
Applying $[e_{-\e_1 1},-]$ again, we get
\begin{align*}
0 & = 2[e_{-\e_1 1},e_{\e_2 2},e_{-\e_1 1},e_{\e_3 3}] + [e_{-\e_1 1},e_{\e_1 1},e_{\e_2 2},e_{-\e_1 1},e_{-\e_1 1},e_{\e_3 3}] \\
 & = 2[e_{\e_2 2},e_{-\e_1 1},e_{-\e_1 1},e_{\e_3 3}] -\e_1 [h_1,e_{\e_2 2},e_{-\e_1 1},e_{-\e_1 1},e_{\e_3 3}] + [e_{\e_1 1},e_{-\e_1 1},e_{\e_2 2},e_{-\e_1 1},e_{-\e_1 1},e_{\e_3 3}] \\
 & = (2-\e_1\nu)[e_{\e_2 2},e_{-\e_1 1},e_{-\e_1 1},e_{\e_3 3}] + [e_{\e_1 1},e_{\e_2 2},e_{-\e_1 1},e_{-\e_1 1},e_{-\e_1 1},e_{\e_3 3}] \\
 & = (2-\e_1\nu)[e_{\e_2 2},e_{-\e_1 1},e_{-\e_1 1},e_{\e_3 3}] 
\end{align*}
where the last identity follows from (R4), since $[e_{-\e_1 1},e_{-\e_1 1},e_{-\e_1 1},e_{\e_3 3}]=0$. Here 
\[\nu=\e_2c_{12}-2\e_1c_{11}+\e_3c_{13}=2\e_2\delta_{12}-4\e_1+2\e_3\delta_{31}
=-2\e_1-4\e_1+2\e_1\delta_{12}\delta_{23}\delta_{31}=-4\e_1
\]
so $2-\e_1\nu=6\neq0$, and therefore we get 
\[
[e_{\e_2 2},e_{-\e_1 1},e_{-\e_1 1},e_{\e_3 3}] =0.
\]
Finally, if we apply $[e_{\e_1 1},-]$ we get
\begin{align*}
0 &= [e_{\e_1 1},e_{\e_2 2},e_{-\e_1 1},e_{-\e_1 1},e_{\e_3 3}] \\
&= [e_{\e_1 1},e_{-\e_1 1},e_{-\e_1 1},e_{\e_2 2},e_{\e_3 3}] \\
&= \e_1[h_1,e_{-\e_1 1},e_{\e_2 2},e_{\e_3 3}] + [e_{-\e_1 1},e_{\e_1 1},e_{-\e_1 1},e_{\e_2 2},e_{\e_3 3}] \\
&= \e_1\lambda_1[e_{-\e_1 1},e_{\e_2 2},e_{\e_3 3}] + \e_1[e_{-\e_1 1},h_1,e_{\e_2 2},e_{\e_3 3}] + [e_{-\e_1 1},e_{-\e_1 1},e_{\e_1 1},e_{\e_2 2},e_{\e_3 3}] \\
&= \e_1\lambda_1[e_{-\e_1 1},e_{\e_2 2},e_{\e_3 3}] + \e_1\lambda_2[e_{-\e_1 1},e_{\e_2 2},e_{\e_3 3}] + [e_{-\e_1 1},e_{-\e_1 1},e_{\e_2 2},e_{\e_1 1},e_{\e_3 3}].
\end{align*}
Now, $e_{\e_1 1}$ and $e_{\e_3 3}$ commute by (R4), and therefore the last term is zero. Furthermore, direct computation shows that $\lambda_1=-\e_1c_{11}+\e_2c_{12}+\e_3c_{13}=-2\e_1(\delta_{12}\delta_{23}-\delta_{12}-1)\neq0$ and 
$\lambda_2=\e_2c_{12}+\e_3c_{13}=2\e_2(\delta_{12}-\delta_{31}\delta_{23})=0$.  Therefore
\[ 0=[e_{-\e_1 1},e_{\e_2 2},e_{\e_3 3}] = [e_{\e_2 2},e_{-\e_1 1},e_{\e_3 3}] = -[e_{\e_2 2},e_{\e_3 3},e_{-\e_1 1}]\]
and, as $\e_4=-\delta_{31}\e_3=-\e_1$, we have the claim.

In particular, for type $C_3$-cycles, a result that is analogous to Corollary \ref{cor_typeD_cycle_2rels} follows, and therefore we just need one relation from (R5) to get all the others.

For $B_3$-cycles, Definition \ref{def:r5} specifies that we should introduce 8 relations. However, the argument of Lemma \ref{lem_typeD_cycle_2rels} works in all cases but for 
\[
[e_{\e i_2},e_{-\delta_{23}\e i_3},e_{\delta_{23}\delta_{31}\e i_1}] \implies [e_{-\delta_{23}\e i_3},e_{\delta_{23}\delta_{31}\e i_1},e_{-\e i_2}].\] 
Thus we can reduce (R4) to 4 relations. But Lemma \ref{lem_flip_jac_id} allows us to get the anticlockwise relations in term of the clockwise ones, and therefore we just need the two relations from the statement.

For $F_4$-cycles, Definition \ref{def:r5} specifies that we should introduce 12 relations. However, the argument of Lemma \ref{lem_typeD_cycle_2rels} works in all cases but for 
\[[e_{\e i_4},e_{-\delta_{41}\e i_1},e_{\delta_{41}\delta_{12}\e i_2},e_{\delta_{41}\delta_{12}\delta_{13}\e i_3}] \implies [e_{-\delta_{41}\e i_1},e_{\delta_{41}\delta_{12}\e i_2},e_{\delta_{41}\delta_{12}\delta_{13}\e i_3},e_{-\e i_4}].\] 
Thus we can reduce (R4) to 4 relations. But Lemma \ref{lem_flip_jac_id} allows us to get the anticlockwise relations in term of the clockwise ones, and therefore we just need the two relations from the statement.
\end{proof}

\begin{remark}
Computer calculations with GAP \cite{gap} suggest that even in types $B$ and $F$, in the Lie algebra $\g_4(C)$
any one of our relations from (R5) implies all the others. 
\end{remark}

\subsection{Homomorphisms}

\begin{definition}
Let $(Q,v)$ be a signed valued quiver of mutation Dynkin type.  We define $\g(Q,v)=\g_5(C)$, where $C=C(Q,v)$ is the Cartan counterpart of $(Q,v)$.
\end{definition}
Fix $k\in Q_0$, and let $C'=C(\mu_k(Q,v))$ and $\g(\mu_k(Q,v))=\g_5(C')$, with generators of $\g_5(C')$ written as $h'_i$ and $e'_{\e i}$.  

In the following formulas we write $i\arr{\delta} k$, where $\delta\in\{\pm1\}$, to denote the sign of an arrow $i\to k$.  Write $C=(c_{ij})$.
\begin{theorem}\label{thm:isoLie}
There is an isomorphism of Lie algebras $\varphi_k:\g(\mu_k(Q,v))\arr\sim \g(Q,v)$ given by the formula
\begin{align*}
    e'_{\varepsilon i} & \mapsto \begin{cases}
        \dfrac{(-\e)^{c_{ki}}}{\abs{c_{ki}}!}\ad^{\abs{c_{ki}}}(e_{-\delta\varepsilon k})(e_{\varepsilon i}) & \text{if } i\arr{\delta} k \text{ in }Q; \\
        e_{\varepsilon i} & \text{otherwise;}
    \end{cases} \\
    h'_i & \mapsto \begin{cases}
        h_i - c_{ik}h_k & \text{if }i\arr{} k \text{ in }Q; \\
        h_i \quad \quad  \quad  \quad \quad \quad \quad \quad \quad \;\;
         & \text{otherwise}.
    \end{cases}
\end{align*}
Its inverse $\psi_k:\g(Q,v)\arr\sim \g(\mu_k(Q,v))$ is given by
\begin{align*}
    e_{\varepsilon i} & \mapsto \begin{cases}
        \dfrac{(-\e)^{c_{ki}}}{\abs{c_{ki}}!}\ad^{\abs{c_{ki}}}(e'_{\delta\varepsilon k})(e'_{\varepsilon i}) & \text{if } i\arr{\delta} k \text{ in }Q; \\
        e_{\varepsilon i} & \text{otherwise;}
    \end{cases} \\
    h_i & \mapsto \begin{cases}
        h'_i + c_{ik}h'_k & \text{if }i\arr{} k \text{ in }Q; \\
        h'_i \quad \quad  \quad  \quad \quad \quad \quad \quad  \;\;\;
         & \text{otherwise}.
    \end{cases}
\end{align*}
\end{theorem}

The rest of this section is devoted to proving Theorem \ref{thm:isoLie}.  Our proof will rely on the isomorphisms of Lie algebras given in \cite{br} (for simply laced types) and \cite{pr} (for non-symmetric Cartan matrices).  

We recall some notation from \cite{pr}.  We have elementary matrices 
\[ E^\sigma_{sr}=I+\sigma e_se_r^T\] 
where $I$ is the $n\times n$ identity matrix, $\sigma\in\Z$, $e_r$ is the $r$th basis vector of the space of column vectors $\R^n$, and $e_r^T$ is its transpose.  Write $I_r=E_{rr}^{-2}$, so that conjugating by $I_r$ negates the $r$th row and column.  If $C$ is a quasi-Cartan matrix with symmetrizer $D$ and $s\neq r$, let 
\[ T^\sigma_{sr}(C)=D^{-1}E^\sigma_{rs}DCE^\sigma_{sr} \;\; \text{ and } \;\; J_r(C)=D^{-1}I_rDCI_r.\]  
Note that $T^0_{sr}(C)=C$ and $J_r$ is self-inverse.

We now state the fundamental results of Barot-Rivera and P\'erez-Rivera in terms of our notation and conventions.
\begin{proposition}[Barot-Rivera, P\'erez-Rivera]\label{prop:pr-J}
For any positive quasi-Cartan matrix $C$ and $1\leq r\leq n$, there is an isomorphism of Lie algebras $\g_5(J_r(C))\to \g_5(C)$ sending
\[ e'_{\e i}\mapsto\begin{cases}
e_{\e i} &\text{if }i\neq r; \\
-e_{-\e r} &\text{if }i=r;
\end{cases}
\;\;\;\;
h'_i\mapsto\begin{cases}
h_i &\text{if }i\neq r; \\
-h_r &\text{if }i=r.
\end{cases}
\]
\end{proposition}
\begin{proof} 
By \cite[Proposition 2.12]{pr}, following \cite[Proposition 2.6]{br} in simply-laced types, there is an isomorphism $\widehat \g_5(J_r(C))\to \widehat \g_5(C)$ sending $\hat e'_{\e r}$ to $\hat e_{-\e r}$, $\hat h'_{r}$ to $-\hat h_{r}$, and for $i\neq r$, $\hat e'_{\e i}\mapsto \hat e_{\e i}$ and $\hat h'_{i}\mapsto \hat h_{i}$.  We compose this with isomorphisms from Lemma \ref{lem:iso_PR_GM} to get the isomorphisms 
$\g_5(J_r(C))\to \widehat \g_5(J_r(C))\to \widehat \g_5(C)\to \g_5(C)$
in the given formula.
\end{proof}

\begin{proposition}[Barot-Rivera, P\'erez-Rivera]\label{prop:pr-T}
For any positive quasi-Cartan matrix $C$ 
and any $1\leq r,s\leq n$ with $r\neq s$, if $c_{sr}=-1$ or $-2$ then there is an isomorphism of Lie algebras $\g_5(T_{sr}^{-c_{sr}}(C))\to \g_5(C)$ sending
\[ e'_{\e i}\mapsto\begin{cases}
e_{\e i} &\text{if }i\neq r; \\
\dfrac{(-\e)^{\abs{c_{sr}}}}{\abs{c_{sr}}} \ad^{\abs{c_{sr}}}( e_{\varepsilon s})( e_{\varepsilon r}) &\text{if }i=r;
\end{cases}
\;\;\;\;
h'_i\mapsto\begin{cases}
h_i &\text{if }i\neq r; \\
h_r- c_{rs} h_s &\text{if }i=r.
\end{cases}
\]
\end{proposition}
\begin{proof} 
By \cite[Proposition 2.14]{pr}, following \cite[Proposition 2.8]{br} in simply-laced types, there is an isomorphism $\widehat \g_5(J_r(C))\to \widehat \g_5(C)$.  Note the typo in the statement of \cite[Proposition 2.14]{pr}: the superscript on $T$ should have a minus sign.  According to equations (4.2) and (4.3) of \cite{pr}, this isomorphism sends $\hat e'_{\e i}\mapsto \hat e_{\e i}$ and $\hat h'_i\mapsto \hat h_i$ for $i\neq r$, and for $i=r$ it sends:
\[ \hat e'_{\e i}\mapsto\begin{cases}
[\hat e_{\e r},\hat e_{\e s}] &\text{if }c_{sr}=-1; \\
\frac12 [\hat e_{\e s},\hat e_{\e s},\hat e_{\e r}]  &\text{if }c_{sr}=-2;
\end{cases}
\;\;\;\;
\hat h'_i\mapsto\begin{cases}
\hat h_r- c_{rs}\hat h_s, &\text{if }c_{sr}=-1; \\
\hat h_r+\hat h_s  &\text{if }c_{sr}=-2.
\end{cases}
\]
By the usual Dynkin classification, together with Lemma \ref{lem_oriented_cycles_values_3_cycles}, we have that $c_{sr}=-2$ implies $c_{rs}=-1$.  So we can unify the formulas for $i=r$ as follows:
\[ \hat e'_{\e r} \mapsto 
\frac{(-1)^{c_{sr}}}{\abs{c_{sr}}} \ad^{\abs{c_{sr}}}(\hat e_{\varepsilon s})(\hat e_{\varepsilon r});
\;\;\;\;
\hat h'_r\mapsto
\hat h_r- c_{rs}\hat h_s.\]
Now conjugate by the isomorphisms from Lemma \ref{lem:iso_PR_GM}.  For $e'_{\e i}$ we get $e'_{\e i}\mapsto \e \hat e'_{\e i}$, giving one $\e$, then from $\ad^{\abs{c_{sr}}}(\hat e_{\varepsilon s})(\hat e_{\varepsilon i})$ we pick up another $\e^{1+c_{sr}}$. 
This gives the formula in the statement.
\end{proof}

\begin{lemma}\label{lem:g2}
For any positive quasi-Cartan matrix $C$ 
and any $1\leq r,s\leq n$ with $r\neq s$, if $c_{sr}=-3$ then there is an isomorphism of Lie algebras $\g_5(T_{sr}^{3}(C))\to \g_5(C)$ sending
\[ e'_{\e i}\mapsto\begin{cases}
e_{\e i} &\text{if }i\neq r; \\
\dfrac{-\e}6 \ad^{3}( e_{\varepsilon s})( e_{\varepsilon r}) &\text{if }i=r;
\end{cases}
\;\;\;\;
h'_i\mapsto\begin{cases}
h_i &\text{if }i\neq r; \\
h_r+h_s &\text{if }i=r.
\end{cases}
\]
\end{lemma}
\begin{proof}
By Lemma \ref{lem_oriented_cycles_values_3_cycles}, the weight $3$ cannot occur in a chordless cycle, so the only possibility is classical type $G_2$.  Up to relabelling, this means we must have
\[ C=\begin{pmatrix}2 &-3\\-1 &2\end{pmatrix} \;\; \text{ and } \;\;
T^3_{12}C=\begin{pmatrix}2 &3\\1 &2\end{pmatrix}.\]
Then our candidate isomorphism acts as follows:
\[ h_1'\mapsto h_1, \;\; h_2'\mapsto h_1+h_2,  \;\; e_1'\mapsto e_1,  \;\; e_2'\mapsto \frac{-1}6[e_1,e_1,e_1,e_2], \;\; e_{-1}'\mapsto e_{-1},  \;\; e_{-2}'\mapsto \frac{-1}6[e_{-1},e_{-1},e_{-1},e_{-2}].\]
The checks that this preserves relations (R1)-(R4) are straightforward, except for the relation $[e'_1,e'_{-1}]=h'_1$, but a careful calculation shows that
\[ [[e_1,e_1,e_1,e_2],[e_{-1},e_{-1},e_{-1},e_{-2}]]=3\left(c_{11}+c_{12}\right)\left(c_{11}+2c_{12}\right)
\left(3c_{21}h_1+c_{12}h_2\right) = -36(h_1+h_2). \]
Then the inverse is easy to write down directly.
\end{proof}

We will consider cases where $c_{sr}$ is positive, as well as negative, so we define:
\[ U_{sr}(C)=\begin{cases}
T_{sr}^{-c_{sr}}(C) &\text{if }c_{sr}\leq 0;\\
J_rT_{sr}^{c_{sr}}J_r(C) &\text{if }c_{sr}>0.
\end{cases} \]
In Lemma \ref{lem_form_U} we will see that we can interpret $U_{sr}$ in a more straightforward way just in terms of the transformation $T_{sr}^{-c_{sr}}$, regardless of whether $c_{sr}$ is positive or negative. For the moment we stick to the definition above because the setting of the proof of the next result requires $c_{sr}<0$.

\begin{proposition}\label{prop:pr-U}
For any positive quasi-Cartan matrix $C$ and any $1\leq r,s\leq n$ with $r\neq s$, there is an isomorphism of Lie algebras $\g_5(U_{sr}(C))\to \g_5(C)$ sending
\[ e'_{\e i}\mapsto\begin{cases}
e_{\e i} &\text{if }i\neq r; \\
\dfrac{(-\e)^{\abs{c_{sr}}}}{\abs{c_{sr}}!} \ad^{\abs{c_{sr}}}( e_{-\delta\varepsilon s})( e_{\varepsilon r}) &\text{if }i=r;
\end{cases}
\;\;\;\;
h'_i\mapsto\begin{cases}
h_{\e i} &\text{if }i\neq r; \\
h_r- c_{rs} h_s &\text{if }i=r,
\end{cases}
\]
where $\delta=\sgn(c_{sr})$.
\end{proposition}
\begin{proof}
When $c_{sr}=0$ we have $U_{sr}(C)=C$ and we interpret our homomorphism as the identity map.
When $c_{sr}<0$ we have $\delta=-1$ and the formulas reduce to those in Proposition \ref{prop:pr-T} (or Lemma \ref{lem:g2}); this also works for $c_{sr}=0$.  When $c_{sr}>0$ we compose the isomorphisms from Propositions \ref{prop:pr-J}, \ref {prop:pr-T} and Lemma \ref{lem:g2}:
\[\g(U_{sr}(C))=
\g(J_rT_{sr}^{c_{sr}}J_r(C))\arrr{\ref{prop:pr-J}}
\g(T_{sr}^{c_{sr}}J_r(C))\arrr{\ref{prop:pr-T}}
\g(J_r(C))\arrr{\ref{prop:pr-J}}
\g(C)\]
Write $C^{(3)}=J_rT_{sr}^{c_{sr}}J_r(C)$, $C^{(2)}=T_{sr}^{c_{sr}}J_r(C)$, etc., using the same superscripts for the generators of the Lie algebras.  Then the isomorphism sends $h'_i=h^{(3)}_i\mapsto h_i$ for $i\neq r$, and 
\[h'_r=h^{(3)}_r\mapsto -h^{(2)}_r\mapsto -h^{(1)}_r+c^{(1)}_{rs}h^{(1)}_s\mapsto h_r+c^{(1)}_{rs}h_s=h_r-c_{rs}h_s\]
where $C^{(1)}=J_r(C)$ gives $c^{(1)}_{rs}=-c_{rs}$.  So the result for $h_i$ does not depend on $\delta$.

The isomorphism
 sends $e^{(3)}_{\e i}\mapsto e_{\e i}$ for $i\neq r$, and 
\[e'_{\e r}\mapsto -e^{(2)}_{-\e r}
\mapsto 
\dfrac{(-\e)^{|c^{(1)}_{sr}+1|}}{|c^{(1)}_{sr}|!} \ad^{|c^{(1)}_{sr}|}( e^{(1)}_{-\varepsilon s})( e^{(1)}_{-\varepsilon r})
\mapsto 
\dfrac{(-\e)^{\abs{c_{sr}}}}{\abs{c_{sr}}!} \ad^{\abs{c_{sr}}}( e_{-\varepsilon s})( e_{\varepsilon r})
\]
as $|c^{(1)}_{sr}|=\abs{c_{sr}}$.  So when $\delta=1$ we have an  $\ad^{\abs{c_{sr}}}( e_{-\varepsilon s})$; when $\delta=-1$ we have an  $\ad^{\abs{c_{sr}}}( e_{\varepsilon s})$.  The $-\delta$ in our formula covers both cases.
\end{proof}

\begin{lemma}\label{lem_form_U}
$U_{sr}(C)=T_{sr}^{-{c_{sr}}}(C)$.
\end{lemma}
\begin{proof}
As noted in the proof of \cite[Lemma 2.4]{br}, it is easily verified that $E_{rs}^\sigma I_i=I_iE_{rs}^{-\sigma} \text{if }i\in\{r,s\}$, and $E_{rs}^\sigma I_i=I_iE_{rs}^{\sigma}$ otherwise.  From this, it follows from the definitions that 
\[
T_{rs}^\sigma J_i=\begin{cases}
J_iT_{rs}^{-\sigma} &\text{if }i\in\{r,s\};\\
J_iT_{rs}^{\sigma} &\text{if }i\notin\{r,s\}
\end{cases}
\]
and so we get $J_rT_{sr}^{-\sigma}J_r=T_{sr}^{\sigma}$, and the statement follows. 
\end{proof}

We get the next result by using the definition of $U_{sr}$ in terms of $E_{sr}^{\sigma}$.
\begin{lemma}
$U_{sr_1}U_{sr_2}=U_{sr_2}U_{sr_1}$.
\end{lemma}

Now we use our signed valued quiver $(Q,v)$.  Given the previous lemma, we can unambiguously write the product in the following lemma.
\begin{lemma}\label{lem:Cartan-agree}
$C(\mu_k(Q,v))=\left( \prod_{i\to k} U_{ki}\right)C(Q,v).$
\end{lemma}
\begin{proof}
We understand $C(\mu_k(Q,v))$ by Lemma \ref{lem:mutateCartan}, so we just need to understand $\left( \prod_{i\to k} U_{kk}\right)C(Q,v)$.  If $c_{sr}<0$ then a useful formula for $C'=U_{sr}(C)=T_{sr}^{-c_{sr}}(C)$, when $c_{sr}<0$, is given in \cite[Equation (4.1)]{pr}.  First note that their three-case equation can be expressed as the following two-case equation:
\[ c'_{xy} = \begin{cases}
c_{xy} &\text{if }(x= r\text{ and }y= r)\text{ or }(x\neq r\text{ and }y\neq r); \\
c_{xy}-c_{xs}c_{sy} &\text{if }(x=r\text{ and }y\neq r)\text{ or }(x\neq r\text{ and }y=r). \\
\end{cases}\]
A similar computation shows that this formula also holds when $c_{sr}\geq0$.

Suppose that the vertices $i$ with arrows $i\to k$ are precisely $i=r_1,\ldots,r_\ell$.  Let $C^{(0)}=C$ and $C^{(q+1)}=U_{kr_q}C^{(q)}$, and write $C'=C^{(\ell)}$.  We use the above formula to calculate $c'_{xy}$ in four cases, depending on whether or not $x,y\in \{r_1,\ldots,r_\ell\}$.
\begin{itemize}
\item
 If $x,y\notin \{r_1,\ldots,r_\ell\}$ then $c^{(q+1)}_{xy}=c^{(q)}_{xy}$ at each stage, so $c'_{xy}=c_{xy}$. 
\item
If $x=r_u$ and $y\notin \{r_1,\ldots,r_\ell\}$ then $c'_{xy}=c^{(u+1)}_{xy}=c^{(u)}_{xy}-c^{(u)}_{xs}c^{(u)}_{sy}=c_{xy}-c_{xs}c_{sy}$.
\item
The case $x \notin \{r_1,\ldots,r_\ell\}$ and $y=r_u$ is exactly the same.
\item
Suppose $x=r_u$ and $y=r_v$ with $u<v$; the case $v<u$ is similar.  Then 
 \[c^{(u+1)}_{xy}=c^{(u)}_{xy}-c^{(u)}_{xs}c^{(u)}_{sy}\;\; \text{ and } \;\;
c^{(v+1)}_{xy}=c^{(v)}_{xy}-c^{(v)}_{xs}c^{(v)}_{sy}\]
with $c^{(p+1)}_{xy}=c^{(p)}_{xy}\text{ for }p\neq u,v$. 
So, using $c_{ss}=2$, we have
\begin{align*} 
c'_{xy} &=c^{(v)}_{xy}-c^{(v)}_{xs}c^{(v)}_{sy}=c^{(u+1)}_{xy}-c^{(u+1)}_{xs}c^{(u+1)}_{sy} \\
  &=\left( c^{(u)}_{xy}-c^{(u)}_{xs}c^{(u)}_{sy} \right) - \left( c^{(u)}_{xs}-c^{(u)}_{xs}c^{(u)}_{ss} \right)c^{(u)}_{sy} \\
  &= c^{(u)}_{xy}-c^{(u)}_{xs}c^{(u)}_{sy}+c^{(u)}_{xs}c^{(u)}_{sy}=c_{xy}
\end{align*}
\end{itemize}
and our answer agrees with Lemma \ref{lem:mutateCartan} in all cases.
\end{proof}

\begin{lemma}\label{lem_semisimpl_kk}
Let $k,i\in Q_0$, and $c_{ki}=\delta \lambda$, where $\delta\in\{\pm 1\}$. Then
\[
[e_{\delta\e k},e_{\delta\e k},e_{-\delta\e k},e_{-\delta\e k},e_{\e i}]=2\lambda(\lambda-1)e_{\e i}.
\]
\end{lemma}
\begin{proof}
Using the Jacobi identity, (R2), (R3) and (R4) we have 
\begin{align*}
[e_{\delta\e k},e_{\delta\e k},e_{-\delta\e k},e_{-\delta\e k},e_{\e i}] & = \delta\e [e_{\delta\e k},h_k,e_{-\delta\e k},e_{\e i}] + [e_{\delta\e k},e_{-\delta\e k},e_{\delta\e k},e_{-\delta\e k},e_{\e i}] \\
& = \delta\e (-2\delta\e +\lambda\delta\e)[e_{\delta\e k},e_{-\delta\e k},e_{\e i}]+\delta\e [e_{\delta\e k},e_{-\delta\e k},h_k,e_{\e i}] \\
& = (-2+\lambda)[e_{\delta\e k},e_{-\delta\e k},e_{\e i}] +\lambda[e_{\delta\e k},e_{-\delta\e k},e_{\e i}] \\
& = (-2+\lambda+\lambda)\delta\e [h_k,e_{\e i}] \\
& = 2\lambda(\lambda-1) e_{\e i}.
\end{align*}
\end{proof}

\begin{proof}[Proof of Theorem \ref{thm:isoLie}]
If $(Q,v)$ has arrows $i\to k$ precisely for vertices $i=r_1,\ldots,r_\ell$,
 we get the Lie algebra homomorphism $\g(\mu_k(Q,v))\to \g(Q,v)$ by composing the morphisms from Proposition \ref{prop:pr-U},
 using Lemma \ref{lem:Cartan-agree}, as follows:
\[ \varphi_k:\g(\mu_k(Q,v))=\g(U_{sr_\ell}U_{sr_{\ell-1}}\cdots U_{sr_1}C)\to \g(U_{sr_{\ell-1}}\cdots U_{sr_1}C)\to\cdots\to \g(U_{sr_1}C)\to \g(C)=\g(Q,v).\]
This gives the formula for $\varphi_k$ in the statement of the Theorem.

Now we prove that $\varphi_k$ and $\psi_k$ are inverse.
For simplicity we omit the index $k$. We will prove that $\varphi\circ \psi=\text{id}$ on the generators. Notice that for generators at vertices $i$, where $i$ does not have an arrow going into $k$, then the identity follows by definition of $\psi$ and $\varphi$. Therefore, we can assume there is an arrow $\xymatrix@1{i\to k}$ in $Q$. 

We have 
\[
\varphi\circ\psi(h_i)=\varphi (h_i'+c_{ik}h_k')=h_i-c_{ik}h_k+c_{ik}h_k=h_i.
\]

To show that $\varphi\circ\psi(e_{\e i})=e_{\e i}$ we divide into cases according to the value of $c_{ki}$. 
\begin{itemize}
    \item If $c_{ki}=\delta\in\{\pm 1\}$ then 
    \[
    \varphi\circ\psi (e_{\e i})= -\e\varphi ([e_{\delta\e k}',e_{\e i}'])= -\e\cdot -\e [e_{\delta\e k},e_{-\delta\e k},e_{\e i}] = \delta\e [h_k,e_{\e i}]= \delta\e\delta\e e_{\e i}= e_{\e i}
    \]
    where we used $[e_{-\delta\e k},e_{\e i}]=0$ from Lemma \ref{lem:commute-es}.
    \item If $c_{ki}=2\delta\in\{\pm 2\}$ then
    \begin{align*}
    \varphi\circ\psi (e_{\e i})= \frac{1}{2}\varphi ([e_{\delta\e k}',e_{\delta\e k}',e_{\e i}']) = \frac{1}{4}[e_{\delta\e k},e_{\delta\e k},e_{-\delta\e k},e_{-\delta\e k},e_{\e i}]= e_{\e i},
    \end{align*}
    where the last equality follows from Lemma \ref{lem_semisimpl_kk}.
    \item If $c_{ki}=3\delta\in\{\pm 3\}$ then
    \begin{align*}
    \varphi\circ\psi (e_{\e i}) & = -\frac{1}{6}\e\varphi ([e_{\delta\e k}',e_{\delta\e k}',e_{\delta\e k}',e_{\e i}']) = \frac{1}{36}[e_{\delta\e k},e_{\delta\e k},e_{\delta\e k},e_{-\delta\e k},e_{-\delta\e k},e_{-\delta\e k},e_{\e i}] \\
    & = \frac{1}{36} (\delta\e [e_{\delta\e k},e_{\delta\e k},h_k,e_{-\delta\e k},e_{-\delta\e k},e_{\e i}]+[e_{\delta\e k},e_{\delta\e k},e_{-\delta\e k},e_{\delta\e k},e_{-\delta\e k},e_{-\delta\e k},e_{\e i}]) \\
    & = \frac{1}{36} (\delta\e (-2\delta\e -2\delta\e +3\delta\e)[e_{\delta\e k},e_{\delta\e k},e_{-\delta\e k},e_{-\delta\e k},e_{\e i}] \\
    & \hspace{0.5cm}+\delta\e [e_{\delta\e k},e_{\delta\e k},e_{-\delta\e k},h_k,e_{-\delta\e k},e_{\e i}]+[e_{\delta\e k},e_{\delta\e k},e_{-\delta\e k},e_{-\delta\e k},e_{\delta\e k},e_{-\delta\e k},e_{\e i}]) \\
    & = \frac{1}{36}(-[e_{\delta\e k},e_{\delta\e k},e_{-\delta\e k},e_{-\delta\e k},e_{\e i}]+\delta\e (-2\delta\e +3\delta\e)[e_{\delta\e k},e_{\delta\e k},e_{-\delta\e k},e_{-\delta\e k},e_{\e i}] \\
    & \hspace{0.5cm} +\delta\e [e_{\delta\e k},e_{\delta\e k},e_{-\delta\e k},e_{-\delta\e k},h_k,e_{\e i}]) \\
    & = \frac{1}{36}(\delta\e(3\delta\e)[e_{\delta\e k},e_{\delta\e k},e_{-\delta\e k},e_{-\delta\e k},e_{\e i}]) \\
    & = \frac{1}{36}(3[e_{\delta\e k},e_{\delta\e k},e_{-\delta\e k},e_{-\delta\e k},e_{\e i}])= e_{\e i},
    \end{align*}
where the last equality follows from Lemma \ref{lem_semisimpl_kk}.
\end{itemize} 

One checks $\psi\circ\varphi=\text{id}$ similarly.
\end{proof}

\begin{corollary}\label{cor:presDynk}
If the signed valued quiver $(Q,v)$ is mutation Dynkin of type $\Delta$ then $\g(Q,v)\cong\g(\Delta)$.
\end{corollary}

\begin{remark}
By considering mutation theories from higher cluster categories, one expects to find other presentations of semisimple Lie algebras, in analogy with \cite{m}.
\end{remark}

\subsection{Root space decompositions}

Let $(Q,v)$ be a signed valued quiver of mutation Dynkin type, and $(Q',v')$ its mutation at $k$.  
In this section we give a compatibility result between root space decompositions of the Lie algebras $\g=\g(Q,v)$ and $\g'=\g'(Q',v')$. 

For basics about root space decomposition of simple Lie algebras we refer to \cite[Section II.8]{hum_la}.

Let $\h,\h'$ denote the Cartan subalgebras of $\g,\g'$ generated by $h_i$ and $h_i'$, respectively.
Denote by $\Phi$ (resp. $\Phi'$) the root system associated to $(Q,v)$ (resp. $(Q',v')$) according to Definition \ref{def:groot}.
We define a map $\Phi\times\h\to\mathbb{C}$ by $\alpha_i(h_j)=c_{ji}$ \cite[Section VI.4]{serre}.

By Corollary \ref{cor:presDynk} we know that $\g$ and $\g'$ are simple complex Lie algebras.
By Theorem \ref{thm:root-sys-iso} their root systems are $\Phi$ and $\Phi'$.  So we have root space decompositions
\begin{align*}
\g & =\h\oplus\bigoplus_{\beta\in\Phi}\g^{\beta} \\
\g' &= \h'\oplus\bigoplus_{\beta'\in\Phi'}\g'^{\beta'}
\end{align*}
where
\[g^{\beta}=\{x\in\g\st [h,x]=\beta(h)x\text{ for all } h\in\h\}.\]
By (R4) we have $[h_i,e_j]=c_{ij}e_j=\alpha_j(h_i)e_j$, so $e_j\in\g^{\alpha_j}$.

Recall the isomorphism of root systems $\rho_k:\Phi'\to\Phi$ from Definition \ref{def:mutate-roots}.

\begin{theorem}\label{thm:rootspacedecomp}
$ 
\varphi_k(\g'^{\beta'})=\g^{\rho_k(\beta')}.
$
\end{theorem}
\begin{proof}
Suppose $x'\in\g'^{\beta'}$, so that $[h',x']=\beta'(h')x'$ for all $h'\in \h'$. Therefore $\varphi_k(x')\in \g^{\rho_k(\beta')}$ if and only if
\[
[h,\varphi_k (x')]=\rho_k (\beta')(h)\varphi_k (x')
\]
for all $h\in\h$. 

Now, fix $h\in\h$. Notice that $\varphi_k$ induces an isomorphism ${\varphi_k}_{|\h'}:\h'\stackrel{\sim}{\to}\h$ by definition, and therefore we can write $h=\varphi_k (h')$ for some $h'\in\h'$. In particular, since $\varphi_k$ is a Lie algebra homomorphism, then we can rewrite the left hand side of the above equation as
\[
[h,\varphi_k (x')]= \varphi_k ([h',x'])=\varphi_k (\beta'(h')x')=\beta'(h')\varphi_k(x').
\]
Hence the result follows if we prove that
\begin{equation}\label{eq_beta'h'}
\beta'(h')=\rho_k (\beta')(h),
\end{equation}
where $h=\varphi_k(h')$.

Notice that it is enough to prove \eqref{eq_beta'h'} for $h=h_i'$ and  $\beta'=\alpha_j'$, where $h_i'$ is a generator of $\h'$ and $\alpha_j'$ is a simple root; the result will then follow by linearity.

Let $C=C(Q,v)=(c_{ij})$ (resp. $C'=C(Q',v')=c'_{ij}$) denote the Cartan counterpart of $(Q,v)$ (resp. $(Q',v')$). Then Definition \ref{def:mutate-roots} and Theorem \ref{thm:isoLie} imply that
\[
\varphi_k (h_i')=\begin{cases}
    h_i-c_{ik}h_k, & \text{ if } i\to k \\
    h_i, & \text{ otherwise},
\end{cases} \hspace{1cm}
\rho_k (\alpha_j')= \begin{cases}
    \alpha_j -c_{kj}\alpha_k, & \text{ if } j\to k \\
    \alpha_j, & \text{ otherwise}.
\end{cases}
\]
Therefore, using $\alpha_{\ell}(h_m)=c_{m\ell}$ and $c_{kk}=2$ we get
\[
\rho_k (\alpha_j')\varphi_k (h_i')= \begin{cases}
c_{ij}-c_{kj}c_{ik}-c_{ik}c_{kj}+c_{ik}c_{kj}c_{kk}=c_{ij}, & \text{ if } i,j\to k \\
c_{ij}-c_{ik}c_{kj}, & \text{ if } i\to k \text{ and } j\not\to k \\
c_{ij}-c_{kj}c_{ik}, & \text{ if } i\not\to k \text{ and } j\to k \\
c_{ij}, & \text{ if }i,j\not\to k.
\end{cases}
\]
Therefore, by Lemma \ref{lem:mutateCartan} we have $\rho_k (\alpha_j')\varphi_k (h_i')=c_{ij}'=\alpha_j'(h_i')$. 

Thus we have the inclusion $\varphi_k(\g'^{\beta'})\subseteq \g^{\rho_k(\beta')}$. The other inclusion follows from the classical result that root spaces are 1-dimensional (see e.g. Theorem 2(b), Chapter VI in \cite{serre}), together with the fact that the map $\varphi_k$ is an isomorphism by Theorem \ref{thm:isoLie}.
\end{proof}

\subsection{Example}

\begin{example}[Type $A_3$]
We start with the signed quiver
\[ Q = 1 \stackrel{-}{\leftarrow} 2\stackrel{-}{\leftarrow} 3\]
which has associated gss matrix and Cartan counterpart
\[ B = \begin{pmatrix}
0 &-t &0\\
t &0 &-t\\
0 &t &0
\end{pmatrix}
\;\;\text{ and }\;\;
C = \begin{pmatrix}
2 &-1 &0\\
-1 &2 &-1\\
0 &-1 &2
\end{pmatrix}.
\]
As $B$ is skew-symmetric, $D$ is the identity matrix.
Its simple roots are $S=\{\alpha_1,\alpha_2,\alpha_3\}$ and the group $W$ is generated by $\{s_1,s_2,s_3\}$.   Its root system is
\[ \Phi = W\cdot S = \pm\{\alpha_1, \; \alpha_2, \; \alpha_3, \; \alpha_1+\alpha_2, \; \alpha_2+\alpha_3, \; \alpha_1+\alpha_2+\alpha_3\}.\]
The Lie algebra $\g$ is generated by $\{e_i,f_i,h_i\}_{i=1,2,3}$ where we write $f_i=e_{-i}$.  
We have the usual (R1), (R2), and (R3) relations, and our (R4) relations are:
\begin{align*}
 \e=1,\,\delta=1: \;\;\;\; & [e_1,e_1,e_2]=[e_1,e_3]=[e_2,e_2,e_1]=[e_2,e_2,e_3]=[e_3,e_3,e_2]=0 \\
 \e=1,\,\delta=-1: \;\;\;\; & [e_1,f_2]=[e_1,f_3]
 =[e_2,f_3]=0
\end{align*}
together with the relations obtained by interchanging $e$'s and $f$'s and some redundant relations.  There are no (R5) relations.
So $\g$ has dimension 15 with basis:
\[ h_1, \, h_2, \, h_3, \, 
e_1, \, e_2, \, e_3, \, [e_1,e_2], \, [e_2,e_3], \, [e_1,e_2,e_3], \, 
f_1, \, f_2, \, f_3, \, [f_1,f_2], \, [f_2,f_3], \, [f_1,f_2,f_3].
\]
Classically, it is isomorphic to the Lie algebra $\mathfrak{sl}_4(\mathbb{C})$ of traceless $4\times4$ matrices by sending the above ordered basis to:
\[ x_{11}-x_{22}, \, x_{22}-x_{33}, \, x_{33}-x_{44}, \, 
x_{12}, \, x_{23}, \, x_{34}, \, x_{13}, \, x_{24}, \, x_{14}, \, 
x_{21}, \, x_{32}, \, x_{43}, \, 
-x_{31}, \, -x_{42}, \, x_{43}
\]
where $x_{ij}$ denotes a matrix with $1$ in the $(i,j)$ entry and $0$ everywhere else.
We define an action of $\alpha_j$ on $h_i$ by $\alpha_j(h_i)=c_{ij}$.  Consider weight spaces $\g^\beta$ for $\beta\in\Phi$.  
By construction we have $[h_i,e_j]=c_{ij}e_j$, so $e_{\e j}\in \g^{\e \alpha_j}$, and Lemma \ref{lem:h-prod} gives the other weight spaces, such as $[e_1,e_2]\in\g^{\alpha_1+\alpha_2}$.

Now we perform signed mutation of $Q$ at the vertex $2$ to get
\[ Q' = \xymatrix @=10pt{
 & 2\ar[dr]^{+} &  \\
1\ar[ur]^{-} && 3\ar[ll]^{-}  
} \]
which has associated gss matrix and Cartan counterpart
\[ B' = \begin{pmatrix}
0 &t &-t\\
-t &0 &1\\
t &-1 &0
\end{pmatrix}
\;\;\text{ and }\;\;
C' = \begin{pmatrix}
2 &-1 &-1\\
-1 &2 &+1\\
-1 &+1 &2
\end{pmatrix}.
\]
Its simple roots are $S'=\{\alpha'_1,\alpha'_2,\alpha'_3\}$ and the group $W'$ is generated by $\{s'_1,s'_2,s'_3\}$.  We get roots
\[s'_1(\alpha'_2)=\alpha'_1+\alpha'_2, \;\; s'_1(\alpha'_3)=\alpha'_1+\alpha'_3, \;\; s'_2(\alpha'_3)=-\alpha'_2+\alpha'_3\]
so we get
\[ \Phi' = W'\cdot S' = \pm\{\alpha'_1, \; \alpha'_2, \; \alpha'_3, \; \alpha'_1+\alpha'_2,  \; \alpha'_1+\alpha'_3, \; -\alpha'_2+\alpha'_3\}.\]

The mutation of roots function acts as follows:
\begin{align*}
\rho_2: \Phi' \to \;&\Phi \\
\alpha_1' \mapsto \;& \alpha_1\\
\alpha_2' \mapsto \;& \alpha_2\\
\alpha_3' \mapsto \;& \alpha_2+\alpha_3\\
\alpha_1'+\alpha'_2 \mapsto \;& \alpha_1+\alpha_2\\
\alpha_1'+\alpha'_3 \mapsto \;& \alpha_1+\alpha_2+\alpha_3\\
-\alpha_2'+\alpha'_3 \mapsto \;& \alpha_3
\end{align*}

Our Lie algebra $\g$ is generated by $\{e'_i,f'_i,h'_i\}_{i=1,2,3}$ where we write $f'_i=e'_{-i}$.  Our (R4) relations are:
\begin{align*}
 \e=1,\,\delta=1: \;\;\;\; & [e'_1,e'_1,e'_2]=[e'_1,e'_1,e'_3]=[e'_2,e'_2,e'_1]=[e'_2,e'_3]=[e'_3,e'_3,e'_2]=0 \\
 \e=1,\,\delta=-1: \;\;\;\; & [e'_1,f'_2]=[e'_1,f'_3]=[e'_2,e_2',f'_3]=[e'_3,e'_3,f'_2]=0 
\end{align*}
together with the relations obtained by interchanging $e$'s and $f$'s and some redundant relations.
Our (R5) relations are:
\[ [e'_1,e'_2,f'_3] = [e'_2, f'_3, f'_1] = [e'_3, e'_1, e'_2 ] = 0 \]
together with the relations obtained by interchanging $e$'s and $f$'s and the (redundant) relations going the other way round the cycle.
So $\g$ has dimension 15 with basis:
\[ h'_1, \, h'_2, \, h'_3, \, 
e'_1, \, e'_2, \, e'_3, \, [e'_1,e'_2], \, [e'_1,e'_3], \, [e'_2,f'_3], \, 
f_1, \, f_2, \, f_3, \, [f'_1,f'_2], \, [f'_1,f'_3], \, [f'_2,e'_3].
\]
Lemma \ref{lem:h-prod} gives the weight spaces: for example, $[e'_2,f'_3]\in\g'^{\alpha_2'-\alpha'_3}$.

The isomorphism $\varphi_2: \g' \to \g$ of Lie algebras acts as follows:
\begin{align*}
h'_1 \mapsto \;& h_1  &
e'_1 \mapsto \;& e_1  &
f_1 \mapsto \;& f_1\\
h'_2 \mapsto \;& h_2  &
e'_2 \mapsto \;& e_2  &
f_2 \mapsto \;& f_2\\
h'_3 \mapsto \;& h_2+h_3  &
e'_3 \mapsto \;& -[e_2,e_3]  &
f_3 \mapsto \;& [f_2,f_3].
\end{align*}
Note that it preserves weight spaces, as predicted by Theorem \ref{thm:rootspacedecomp}.
For example, we have:
\[ \xymatrix{
\Phi\ar[d] & \Phi' \ar[l]_{\rho_2}
\ar[d] && {\phantom{XX}}-\alpha_3\ar@{|->}@<3ex>[d] & 
\alpha_2'-\alpha'_3 {\phantom{XXXXX}}\ar@{|->}[l]\ar@{|->}@<-4ex>[d] \\
\{\g^\beta\st \beta\in\Phi\} & \{\g'^{\beta'}\st \beta'\in\Phi'\} \ar[l]_{\phi_2} && -f_3\in\g^{-\alpha_3} & \g'^{\alpha_2'-\alpha'_3}\ni[e'_2,f'_3] \ar@{|->}[l]
}\]
Composing this isomorphism with the classical realization of $\g$ as $\mathfrak{sl}_4(\mathbb{C})$, we get
\[ e'_1\mapsto x_{12}, \;\;\;\; e'_2 \mapsto x_{23}, \;\;\;\;
e'_3 \mapsto  -x_{24}.\]
\end{example}

%
%

\end{document}